\documentclass[reqno,fleqn,11pt]{amsart}
\pdfoutput1

\usepackage{amssymb,color,hyperref,mathrsfs,stmaryrd}
\usepackage{amsmath}

\usepackage[usenames,dvipsnames]{xcolor}

\usepackage[curve,matrix,arrow]{xy}

\numberwithin{equation}{section}
\newtheorem{theorem}{Theorem}[section]
\newtheorem{lemma}[theorem]{Lemma}
\newtheorem{hypo}[theorem]{Hypothesis}

\newtheorem{prop}[theorem]{Proposition}

\newtheorem{Th}{Theorem}
\newtheorem{bth}{Background Theorem}

\newtheorem{Ex}{Example}

\newtheorem{corollary}[theorem]{Corollary}

\theoremstyle{definition}
\newtheorem{example}[theorem]{Example}

\newtheorem{remark}[theorem]{Remark}
\newtheorem{Remark}{Remark}

\newtheorem{notation}[theorem]{Notation}
\newtheorem{definition}[theorem]{Definition}
\newtheorem{defn}{Definition}
\newtheorem{notn}[defn]{Notation}

\newcommand{\bN}{\mathbb{N}}

\newcommand{\B}{\mathcal{B}}
\newcommand{\F}{\mathcal{F}}
\newcommand{\E}{\mathcal{E}}
\newcommand{\G}{\mathcal{G}}
\newcommand{\C}{\mathcal{C}}
\newcommand{\X}{\mathcal{X}}
\renewcommand{\L}{\mathcal{L}}
\newcommand{\mD}{\mathcal{D}}

\newcommand{\tL}{\widetilde{\mathcal{L}}}
\newcommand{\tDelta}{\tilde{\Delta}}

\newcommand{\tE}{\tilde{\mathcal{E}}}
\newcommand{\tT}{\tilde{T}}
\newcommand{\tS}{\tilde{S}}

\newcommand{\tG}{\tilde{G}}
\newcommand{\tH}{\tilde{H}}
\newcommand{\M}{\mathcal{M}}
\newcommand{\N}{\mathcal{N}}
\newcommand{\K}{\mathcal{K}}
\newcommand{\D}{\mathbf{D}}

\renewcommand{\H}{\mathcal{H}}

\newcommand{\Hom}{\operatorname{Hom}}
\newcommand{\Aut}{\operatorname{Aut}}
\newcommand{\Inn}{\operatorname{Inn}}

\newcommand{\Syl}{\operatorname{Syl}}

\newcommand{\m}{\mathcal}

\newcommand{\id}{\operatorname{id}}
\newcommand{\Ac}{\operatorname{A}^\circ}
\newcommand{\One}{\operatorname{\mathbf{1}}}
\newcommand{\W}{\mathbf{W}}

\newcommand{\hyp}{\mathfrak{hyp}}
\newcommand{\Comp}{\operatorname{Comp}}

\newcommand{\fC}{\mathfrak{C}}
\newcommand{\foc}{\mathfrak{foc}}
\newcommand{\Y}{\mathcal{Y}}

\def \<{\langle }
\def \>{\rangle }
\newcommand{\subn}{\unlhd\!\unlhd\;}

\renewcommand{\phi}{\varphi}
\title{Normalizers and centralizers of subnormal subsystems of fusion systems}

\author[E.~Henke]{Ellen Henke}
\address{Institut f{\"u}r Algebra, Fakult{\"a}t Mathematik, Technische Universit{\"a}t Dresden, 01062 Dresden, Germany}
\email{ellen.henke@tu-dresden.de}

\begin{document}

\maketitle

\begin{center}
 \small \textit{Dedicated to Bernd Stellmacher on the occasion of his eightieth birthday}
\end{center}

\begin{abstract}
Every saturated fusion system corresponds to a group-like structure called a regular locality. In this paper we study (suitably defined) normalizers and centralizers of partial subnormal subgroups of regular localities. This leads to a reasonable notion of normalizers and centralizers of subnormal subsystems of fusion systems.
\end{abstract}

\section{Introduction}

In this paper we introduce and study normalizers and centralizers of subnormal subsystems of saturated fusion systems. In doing so, we generalize and unify some concepts which were known before: Aschbacher defined centralizers of normal subsystems as well as normalizers and centralizers of components of fusion systems (cf. \cite[Chapter~6]{Aschbacher:2011} and \cite[Sections~2.1 and 2.2]{AschbacherFSCT}). The latter construction was generalized by Oliver \cite{Oliver:NormalizerComponents}  to define normalizers of products of components. All these concepts are special cases of the normalizers and centralizers introduced in this paper. However, the existence of centralizers of normal subsystems (and the concrete construction of such centralizers given in \cite{Henke:2018}) is used in the proof of our results.

\smallskip

Our approach uses the fact that there is an (essentially unique) regular locality associated to every saturated fusion system. Here regular localities are group-like structures introduced by Chermak. Recently, Chermak and the author of this paper \cite{Chermak/Henke} established a kind of dictionary which makes it possible to translate between concepts in saturated fusion systems and corresponding concepts in regular localities. In particular, it is shown that there is a natural bijection between the subnormal subsystems of a given saturated fusion system and the partial subnormal subgroups of an associated regular locality. It is generally hoped that some flexibility in moving between fusion systems and localities will lead to simplifications in Aschbacher's program to give a new treatment of the classification of finite simple groups.

\smallskip

In the present paper, we investigate first ``regular normalizers'' and centralizers of partial subnormal subgroups of regular localities. Only afterwards, this is used to define and study normalizers and centralizers of subnormal subsystems.  Along the way, it becomes automatically transparent how one can translate between regular localities and fusion systems with regards to the concepts we will be concerned with. 

\smallskip

Before we state our results on regular localities, we collect  some definitions, notation and background results. \textbf{The reader who is only interested in the fusion-theoretic results is invited to go immediately to Subsection~\ref{SS:FusionResults}.}

\subsection{Some background}\label{SS:BackgroundIntro}

Throughout this paper we write homomorphisms on the right hand side of the argument. The reader is referred e.g. to \cite[Section~3]{Chermak/Henke} for a detailed summary of the required definitions and results concerning partial groups and localities. As in this reference and in Chermak's original papers \cite{Chermak:2013,Chermak:2015}, we write $\W(\L)$ for the set of words in a set $\L$, $\emptyset$ for the empty word, and $v_1\circ v_2\circ\cdots\circ v_n$ for the concatenation of words $v_1,\dots,v_n\in\W(\L)$. A partial group consists of a set $\L$, a ``product'' $\Pi\colon \D\rightarrow \L$ defined only on the words in some subset $\D$ of $\W(\L)$, and an involutory bijection  $\L\rightarrow \L,f\mapsto f^{-1}$ called an ``inversion map'', subject to certain group-like axioms. If $\L$ is a partial group with product $\Pi\colon \D\rightarrow \L$ then, given $(x_1,\dots,x_n)\in\D$, we write also $x_1x_2\cdots x_n$ for $\Pi(x_1,\dots,x_n)$. Below we collect some  important definitions.

\smallskip

\begin{defn}~\label{D:CollectLocalities}
\begin{itemize}
 \item A subset $\H\subseteq\L$ is called a \emph{partial subgroup} of $\L$ if $\Pi(w)\in\H$ for every $w\in\W(\H)\cap \D$ and $h^{-1}\in\H$ for every $h\in\H$.
\item A partial subgroup $\H$ of $\L$ is called a \emph{subgroup} of $\L$ if $\W(\H)\subseteq\D$. Observe that every subgroup of $\L$ forms an actual group. We call a subgroup $\H$ of $\L$ a \emph{$p$-subgroup} if it is a $p$-group.
\item Given $f\in\L$, we write $\D(f):=\{x\in\L\colon (f^{-1},x,f)\in\D\}$ for the set of elements $x\in\L$ for which the \emph{conjugate} $x^f:=\Pi(f^{-1},x,f)$ is defined. This gives us a conjugation map $c_f\colon\D(f)\rightarrow \L$ defined by $x\mapsto x^f$.
\item We say that a partial subgroup $\N$ of $\L$ is a \emph{partial normal subgroup} of $\L$ (and write $\N\unlhd\L$) if $x^f\in\N$ for every $f\in\L$ and every $x\in \D(f)\cap \N$. 
\item We call a partial subgroup $\H$ of $\L$ a \emph{partial subnormal subgroup} of $\L$ (and write $\H\subn\L$) if there is a sequence $\H_0,\H_1,\dots,\H_k$ of partial subgroups of $\L$ such that 
\[\H=\H_0\unlhd\H_1\unlhd\cdots\unlhd \H_{k-1}\unlhd \H_k=\L.\]
\item If $f\in\L$ and $\X\subseteq \D(f)$ set $\X^f:=\{x^f\colon x\in\X\}$. For every $\X\subseteq\L$, call 
\[N_\L(\X):=\{f\in\L\colon \X\subseteq\D(f)\mbox{ and }\X^f=\X\}\]
the \emph{normalizer} of $\X$ in $\L$.
\item For $\X\subseteq \L$, we call 
\[C_\L(\X):=\{f\in\L\colon \X\subseteq\D(f),\;x^f=x\mbox{ for all }x\in\X\}\]
the \emph{centralizer} of $\X$ in $\L$.
\item For any two subsets $\X,\Y\subseteq\L$, their \emph{product} can be naturally defined by 
\[\X\Y:=\{\Pi(x,y)\colon x\in\X,\;y\in\Y,\;(x,y)\in\D\}.\]
\end{itemize}
\end{defn}

A \emph{locality} is a triple $(\L,\Delta,S)$ such that $\L$ is a partial group, $S$ is maximal among the $p$-subgroups of $\L$, and $\Delta$ is a set of subgroups of $S$ subject to certain axioms, which imply in particular that $N_\L(P)$ is a subgroup of $\L$ for every $P\in\Delta$. If $R\leq S$ and $f\in\L$ with $R\subseteq\D(f)$ and $R^f\subseteq S$, then $R^f$ is a subgroup of $S$ and $c_f|_R$ turns out to be an injective group homomorphism from $R$ to $R^f$. Moreover, for every $f\in\L$, the subset 
\[S_f:=\{x\in S\colon x\in\D(f),\;x^f\in S\}\]
is a subgroup of $S$ and indeed an element of $\Delta$. In particular, $c_f|_{S_f}$ is an injective group homomorphism $S_f\rightarrow S$. 

\begin{defn}
If $\H$ is a partial subgroup of $\L$ and $T=S\cap\H$, then $\F_T(\H)$ denotes the fusion system over $T$ generated by the maps of the form $c_h\colon R\rightarrow R^h$  where $h\in\H$ and $R\leq T$ with $R^h\leq T$. Equivalently, $\F_T(\H)$ is generated by all the maps $c_h|_{S_h\cap \H}$  with $h\in\H$.
\end{defn}

Notice that we have in particular $\F_S(\L)$ defined for every locality $(\L,\Delta,S)$. We say that $(\L,\Delta,S)$ is a locality \emph{over} $\F$ if $\F=\F_S(\L)$.

\smallskip

Below we will consider situations in which we are given a partial subgroup $\X$ of $\L$ such that, for some $p$-subgroup $S_0$ and some set $\Gamma$ of subgroups of $S_0$, the triple $(\X,\Gamma,S_0)$ is a locality. Notice that we may then also form the fusion system $\F_{S_0}(\X)$ according to the definition above, even if $S_0$ is not a subgroup of $S$.

\smallskip

Chermak \cite{ChermakIII} defined a locality $(\L,\Delta,S)$ over a saturated fusion system $\F$ to be \emph{regular} if the set $\Delta$ equals a certain set $\delta(\F)$ of subgroups of $S$ and $N_\L(P)$ is a group of characteristic $p$ for every $P\in\Delta$. Here a finite group $G$ is \emph{of characteristic $p$} if $C_G(O_p(G))\leq O_p(G)$. The set $\delta(\F)$ can be characterized as the set of all subgroups $P\leq S$ such that $P\cap T^*$ is subcentric in $\F$, where $T^*$ is the Sylow subgroup of the generalized Fitting subsystem $F^*(\F)$ (as defined by Aschbacher \cite{AschbacherFSCT}). Alternatively, $\delta(\F)$ can be characterized in a similar fashion in terms of a suitable locality and its generalized Fitting subgroup (cf. \eqref{E:deltaF0}.

\smallskip

We will list now some theorems which will be crucial in our approach. The first one follows essentially from the existence and uniqueness of centric linking systems. For a detailed proof see e.g. \cite[Theorem~A]{Henke:2015} and \cite[Lemma~10.4]{Henke:Regular}.

\begin{bth}\label{bT:0}
 Let $\F$ be a saturated fusion system. Then there exists a regular locality $(\L,\Delta,S)$ over $\F$, which is unique up to an isomorphism which restricts to the identity on $S$. 
\end{bth}

The next theorem was first observed by Chermak \cite[Theorem~C]{ChermakIII}. A proof can also be found in  \cite[Theorem~10.16(e), Corollary~10.19]{Henke:Regular}.

\begin{bth}\label{bT:1}
 Let $(\L,\Delta,S)$ be a regular locality over $\F$ and $\H\subn\L$. Then $\E:=\F_{S\cap\H}(\H)$ is saturated and $(\H,\delta(\E),S\cap \H)$ is a regular locality over $\E$. If $\H\unlhd\L$, then $C_\L(\H)\unlhd\L$. 
\end{bth}

As a consequence of the above theorem, for every partial normal subgroup $\H$ of $\L$, the centralizer $C_\L(\H)$ can be given the structure of a regular locality.

\smallskip

There is a natural definition of quasisimple regular localities. Thus, Background Theorem~\ref{bT:1} leads to a notion of components of regular localities. If $(\L,\Delta,S)$ is regular, then we write $\Comp(\L)$ for the set of components of $\L$. The product $E(\L)$ of components of $\L$ forms a partial normal subgroup of $\L$. Similarly, there is a natural notion of components of fusion systems introduced by Aschbacher \cite[Chapter~9]{Aschbacher:2011} such that there is a normal subsystem $E(\F)$ of $\F$ which is a central product of the components. The set of components of $\F$ is denoted by $\Comp(\F)$. We write moreover $\E\subn\F$ to indicate that $\E$ is a subnormal subsystem of $\F$.

\begin{bth}[{\cite[Proposition~7.1(a), Lemma~7.3(a),(b)]{Chermak/Henke}}]\label{bT:2}
The map \[\hat{\Psi}\colon \{\H\colon\H\subn\L\}\rightarrow\{\E\colon\E\subn\F\},\H\mapsto \F_{S\cap\H}(\H)\] is well-defined and a bijection. It restricts to a bijection from the set of partial normal subgroups of $\L$ to the set of normal subsystems of $\F$, and to a bijection from $\Comp(\L)$ to $\Comp(\F)$. Moreover, 
\[\hat{\Psi}(E(\L))=E(\F).\]
\end{bth}

\subsection{Normalizers and centralizers in regular localities}

Suppose $(\L,\Delta,S)$ is a regular locality and $\H\subn\L$. The normalizer $N_\L(\H)$ as introduced in Definition~\ref{D:CollectLocalities} is not the right subset to consider if we want to find a partial subgroup of $\L$ in which $\H$ is normal. Indeed, $\H$ is in most cases not even contained in $N_\L(\H)$. We will however show that there is another partial subgroup $\bN_\L(\H)$ of $\L$ which one can reasonably think of as a ``normalizer in $\L$ of $\H$'' and which contains $\H$ as a partial normal subgroup.

\smallskip

One can observe relatively easily that $\tT:=S\cap E(\L)$ is an element of $\Delta$. In particular, $G:=N_\L(\tT)$ is a finite group of characteristic $p$. We show in Lemma~\ref{L:BasisNLtT} that $G$ acts on $\L$ via conjugation. More precisely, for every $g\in G$, we have $\L\subseteq\D(g)$ (i.e. conjugation by $g$ is defined on all elements of $\L$), and this gives an action of the group $G$ on the set $\L$.  It turns also out that the conjugation map $c_g\colon \L\rightarrow \L$ is for every element $g\in G$ an automorphism of $\L$. As $G$ acts on $\L$, the set-wise stabilizer 
\[N_G(\X):=\{g\in G\colon \X^g=\X\}\]
of $\X$ in the group $G$ is for every subset $\X$ of $\L$ a subgroup of $G$. Set $N_S(\X):=S\cap N_G(\X)=S\cap N_\L(\X)$.

\smallskip

If we consider now a partial subnormal subgroup $\H$ of $\L$ as above then, similarly as in finite groups, every component of $\L$ is either contained in $\H$  or in the centralizer of $\H$. Hence, it is reasonable to think that the product $E(\L)$ of components should be contained in $\bN_\L(\H)$. There is moreover a Frattini Argument for localities (cf. \cite[Corollary~3.11]{Chermak:2015}) which implies that
\[\L=E(\L)N_\L(\tT)=E(\L)G.\]
Therefore, the following definition seems natural.

\begin{defn}
Let $(\L,\Delta,S)$ be a regular locality over $\F$, set $\tT:=E(\L)\cap S$ and $G:=N_\L(\tT)$. For every partial subnormal subgroup $\H$ of $\L$, the subset  
\[\bN_\L(\H):=E(\L)N_G(\H)\]
is called the \emph{regular normalizer} of $\H$ in $\L$.  
\end{defn}

If $\H\unlhd\L$, then $\H^g=\H$ for every $g\in G$. So we have in this case that $N_G(\H)=G$ and thus $\bN_\L(\H)=E(\L)G=\L$ as one would expect. In the general case, the properties of $\bN_\L(\H)$ are summarized in the following theorem.

\begin{Th}\label{T:MainbNLH}
Let $(\L,\Delta,S)$ be a regular locality over $\F$ and $\H\subn\L$. Then the regular normalizer $\bN_\L(\H)$ is a partial subgroup of $\L$. Moreover, setting 
\[\tT:=E(\L)\cap S,\;G:=N_\L(\tT),\;T:=\H\cap S,\mbox{ and }\E:=\F_T(\H),\]
the following hold for every $S_0\in\Syl_p(N_G(\H))$:
\begin{itemize}
 \item [(a)] There exists a set $\Gamma$ of subgroups of $S_0$ such that $(\bN_\L(\H),\Gamma,S_0)$ is a regular locality. In particular, the fusion system $\F_0:=\F_{S_0}(\bN_\L(\H))$ is well-defined and saturated.
 \item [(b)] $\H\unlhd\bN_\L(\H)$. If $S_0$ is chosen such that $T\leq S_0$, then $T=\H\cap S_0$ and $\E\unlhd\F_0$. 
 \item [(c)] $\Comp(\bN_\L(\H))=\Comp(\L)$, $E(\bN_\L(\H))=E(\L)$ and $\tT=E(\L)\cap S_0$. In particular, $\Comp(\F)=\Comp(\F_0)$ and  $E(\F)=E(\F_0)$.
 \item [(d)] $S\cap \bN_\L(\H)=N_S(\H)$.
 \item [(e)] Suppose $\X$ is a partial subgroup of $\L$ such that $\H\unlhd\X$ and $(\X,\Gamma_\X,S_\X)$ is a locality for some $p$-subgroup $S_\X$ of $\X$ with $T\leq S_\X$ and some set $\Gamma_\X$ of subgroups of $S_\X$. Then $\X\subseteq \bN_\L(\H)$.
 \item [(f)] $C_\L(T)\subseteq \bN_\L(\H)$ and 
\[C_\L(\H)=C_{\bN_\L(\H)}(\H)\unlhd \bN_\L(\H).\]
\end{itemize}
\end{Th}

The proof of Theorem~\ref{T:MainbNLH} is given at the end of Subsection~\ref{SS:NLH}. With the hypothesis as above, the theorem says that $\bN_\L(\H)$ gives rise to a saturated fusion system, which is however not necessarily a subsystem of $\F:=\F_S(\L)$ as $S_0$ might not be a subgroup of $S$. It follows from Background Theorem~\ref{bT:1} and parts (a),(f) of Theorem~\ref{T:MainbNLH} that $C_\L(\H)$ can be given the structure of a regular locality. Thus, $C_\L(\H)$ gives also rise to a saturated fusion system $\C_0$, which again is not necessarily a subsystem of $\F$. One can see here that it might be an advantage to work with regular localities rather than with fusion systems in some contexts. Nevertheless, we want to have some more fusion-theoretic results in place as well. We therefore define normalizers and centralizers of subnormal subsystems of saturated fusion systems now. The relationship between $\F_0$ and $N_\F(\E)$ and between $\C_0$ and $C_\F(\E)$ is discussed in Remark~\ref{R:Main} below.

\begin{defn}\label{D:NFECFE}
Let $\F$ be a saturated fusion system over $S$ and let $\E$ be a subnormal subsystem of $\F$ over $T\leq S$. Fix a regular locality $(\L,\Delta,S)$ over $\F$ and $\H\subn\L$ with $T=S\cap \H$ and $\E=\F_T(\H)$ (which exist by Background Theorems~\ref{bT:0} and \ref{bT:2}). We call then
\[N_\F(\E):=\F_{N_S(\H)}(\bN_\L(\H))\]
the \emph{normalizer of $\E$ in $\F$} and
\[C_\F(\E):=\F_{C_S(\H)}(C_\L(\H)).\]
the \emph{centralizer of $\E$ in $\F$}.
\end{defn}

Recall that a regular locality $(\L,\Delta,S)$ over $\F$ is by Background Theorem~\ref{bT:0} essentially unique. Moreover, for a given regular locality $(\L,\Delta,S)$ over $\F$, a partial subnormal subgroup $\H$ as above is also unique by Background Theorem~\ref{bT:2}. This can be used to show that $N_\F(\E)$ and $C_\F(\E)$ do not depend on the choice of $(\L,\Delta,S)$ and $\H$. Alternatively, we prove some  characterizations of  $N_\F(\E)$ and $C_\F(\E)$ in terms of fusion systems below, which imply that $N_\F(\E)$ and $C_\F(\E)$ are well-defined.

\subsection{Normalizers and centralizers of subnormal subsystems}\label{SS:FusionResults} 

In this subsection we will summarize some properties of $N_\F(\E)$ and $C_\F(\E)$ as introduced in Definition~\ref{D:NFECFE}. Even though we defined this normalizer and centralizer using regular localities, the results we will state in this subsection are formulated purely in fusion-theoretic terms. The reader is referred to \cite[Chapter~I]{Aschbacher/Kessar/Oliver:2011} for an introduction to the theory of fusion systems. Throughout this subsection we assume the following hypothesis:

\smallskip

\noindent \textbf{Let $\F$ be a saturated fusion system over a $p$-group $S$ and $\E$ a subnormal subsystem of $\F$ over $T\leq S$.}

\vspace{0.3 cm}

\begin{notn}\label{N:FusionBasic}~
\begin{itemize}
\item If $R$ is any $p$-group and $\Phi$ a set of injective group homomorphisms between subgroups of $R$, then $\<\Phi\>_R$ denotes the fusion system over $R$ generated by $\Phi$, i.e. the smallest fusion system over $R$ containing every morphism in $\Phi$.
 \item If $Q,R$ are $p$-groups, $\mD$ is a fusion system over $R$ and $\phi\colon R\rightarrow Q$ is an injective group homomorphism, then write $\mD^\phi$ for the fusion system over $R\phi$ with
\[\Hom_{\mD^\phi}(P\phi,Q\phi)=\{\phi^{-1}\psi\phi\colon \psi\in\Hom_{\mD}(P,Q)\}\mbox{ for all }P,Q\leq R.\] 
 \item By $\Aut(\F)$ we denote the set of automorphisms $\alpha$ of $S$ with $\F^\alpha=\F$. 
 \item Set $N_S(\E):=\{s\in N_S(T)\colon c_s|_T\in\Aut(\E)\}$.
\end{itemize}
\end{notn}

Before we turn attention to the normalizer and the centralizer of $\E$ in $\F$,
let us mention another result which we obtain along the way. Namely, we show that there is a nice notion of a product subsystem $\E R=(\E R)_\F$ for every subgroup $R$ of $N_S(\E)$. The existence of such a product was known before in the case that $\E$ is normal (cf. \cite[Chapter~8]{Aschbacher:2011} or \cite{Henke:2013}). For arbitrary $\E$, an explicit description of the product subsystem $\E R=(\E R)_\F$ is given in Definition~\ref{D:ProductER}. The subsystem $\E R$ is also characterized by part (a) of the following theorem.

\begin{Th}\label{T:MainER}
Let $R\leq N_S(\E)$. Then the following hold.
\begin{itemize}
 \item [(a)] The subsystem $\E R=(\E R)_\F$ is saturated and is the unique saturated subsystem $\mD$ of $\F$ over $TR$ with $O^p(\mD)=O^p(\E)$. Moreover, $\E\unlhd \E R$.
\item [(b)] If $(\L,\Delta,S)$ is a regular locality over $\F$ and $\H\subn\L$ with $T=S\cap\H$ and $\F_T(\H)=\E$, then $\E R=\F_{TR}(\H R)$.
\end{itemize}
\end{Th}

The proof of Theorem~\ref{T:MainER} can be found after Proposition~\ref{P:DinNFE}. 

\smallskip

To summarize the most important properties of the normalizer $N_\F(\E)$ we will use the following definition.

\begin{defn}\label{D:6}~
\begin{itemize}
 \item Set
 \[N_\F(T,\E):=\<\phi\in\Hom_\F(X,Y)\colon X,Y\leq N_S(\E),\;T\leq X\cap Y,\;\phi|_T\in\Aut(\E)\>_{N_S(\E)}.\]
\item Every element of $\E^\F:=\{\E^\phi\colon \phi\in\Hom_\F(T,S)\}$ is called an \emph{$\F$-conjugate} of $\E$.
\item The subsystem $\E$ is said to be \emph{fully $\F$-normalized}, if 
\[|N_S(\E)|\geq |N_S(\tE)|\mbox{ for all }\tE\in\E^\F.\]
\end{itemize}
\end{defn}

It is shown in Lemma~\ref{L:ConjugatesSubnormal} below that every $\F$-conjugate of $\E$ is again a subnormal subsystem of $\F$. Note that there is always a fully $\F$-normalized $\F$-conjugate of $\E$.

\begin{Ex}\label{Ex:Easy}
Let $G:=S_4$ be the symmetric group on $4$ letters and let $S\in\Syl_2(G)$. Consider the subnormal subsystems of $\F:=\F_S(G)$ of the form $\F_R(R)$ with $R\leq O_2(G)$ of order $2$. The set of these subnormal subsystems forms an $\F$-conjugacy class. Moreover, such a subsystem $\F_R(R)$ is fully normalized precisely when $R=Z(S)$ (i.e. precisely when $R$ is fully normalized).
\end{Ex}

The main properties of the normalizer of $\E$ are summarized in the following theorem.

\begin{Th}\label{T:FusionNormalizerNew}~
 \begin{itemize}
 \item [(a)] $N_\F(\E)$ is the subsystem of $\F$ over $N_S(\E)$ given by
 \[N_\F(\E)=\<\E N_S(\E),\;N_\F(T,\E)\>_{N_S(\E)}.\]
 \item [(b)] $\E$ is a subsystem of $N_\F(\E)$, which is $N_\F(\E)$-invariant.
 \item [(c)] Every saturated subsystem $\mD$ of $\F$ with $\E\unlhd\mD$ is contained in $N_\F(\E)$.
 \item [(d)] If $\E$ is fully $\F$-normalized, then $N_\F(\E)$ is saturated, $\E$ is normal in $N_\F(\E)$, and
\[E(N_\F(\E))=E(\F).\]
\item [(e)] $N_{N_\F(\E)}(T)=N_\F(T,\E)$ and $C_{N_\F(\E)}(T)=C_\F(T)$.
\item [(f)] Suppose $\F$ is constrained and $G$ is a model for $\F$. If $H\subn G$ with $T=S\cap H$ and $\F_T(H)=\E$, then $N_\F(\E)=\F_{N_S(H)}(N_G(H))$.
\end{itemize}
\end{Th}

The proof of Theorem~\ref{T:FusionNormalizerNew} can be found after Lemma~\ref{L:EfullyNormalized}. In part (a) of the theorem we use the notion of the product subsystem mentioned before, where $\E N_S(\E)$ stands more precisely for $(\E N_S(\E))_\F$. Another concrete description of $N_\F(\E)$ (which  somehow resembles the definition of $\bN_\L(\H)$) is given in Theorem~\ref{T:FusionNormalizerGenerate}. 

\smallskip

A weak version of Theorem~\ref{T:MainER} is crucial in proving Theorem~\ref{T:FusionNormalizerNew}(c). If $\E$ is fully $\F$-normalized, then notice that $N_\F(\E)$ is by parts (c) and (d) of Theorem~\ref{T:FusionNormalizerNew} the largest saturated subsystem of $\F$ containing $\E$ as a normal subsystem.  In Lemma~\ref{L:EfullyNormalized} we give some equivalent conditions for $\E$ to be fully $\F$-normalized.

\smallskip

With the above results about the normalizer $N_\F(\E)$ and some related properties in place, it is relatively easy to study the centralizer $C_\F(\E)$. First of all, we can show that there is a reasonable notion of a centralizer of $\E$ in $S$.

\begin{Th}\label{T:CentralizerEinS}
The set $\{R\leq S\colon \E\subseteq C_\F(R)\}$ has (with respect to inclusion) a unique maximal member $C_S(\E)$. %Moreover, $C_\F(\E)$ is a fusion system over $C_S(\E)$.
\end{Th}

Theorem~\ref{T:CentralizerEinS} is restated and proved in Proposition~\ref{P:CSE}. Next we want to list some of the properties of the centralizer of $\E$ in $\F$, but we need one more definition. As mentioned before, every $\F$-conjugate of a subnormal subsystem is subnormal. Hence, the first part of the following definition makes sense.

\begin{defn}\label{D:CommuteSubsystems}~
\begin{itemize}
 \item The subsystem $\E$ is called \emph{fully $\F$-centralized} if
\[|C_S(\E)|\geq |C_S(\tE)|\mbox{ for all }\tE\in\E^\F.\]
\item Given two subsystems $\F_1$ and $\F_2$ of $\F$ over $S_1$ and $S_2$ respectively, we say that $\F_1$ \emph{commutes} with $\F_2$, if $\F_i\subseteq C_\F(S_{3-i})$ for each $i=1,2$ and $S_1\cap S_2\leq Z(\F_1)\cap Z(\F_2)$.
\end{itemize}
\end{defn}

Assuming that the subsystems $\F_1$ and $\F_2$ above are saturated, it follows from \cite[Lemma~3.2]{Henke:2018} that they commute if and only if $\F_i\subseteq C_\F(S_{3-i})$ for each $i=1,2$.

\smallskip

When looking at Example~\ref{Ex:Easy} notice that a subsystem of the form $\F_R(R)$ with $R\leq O_2(G)$ of order $2$ is fully $\F$-centralized if and only if $R=Z(S)$ (i.e. if and only if $R$ is fully centralized). In particular, $\F_R(R)$ is fully normalized precisely when it is fully centralized. This example is a very special one of course, as firstly $\F_R(R)$ is the fusion system of a $p$-group, and secondly, with $R$ being of order $2$, we have $C_S(R)=N_S(R)$. In general, the following implications are true by Lemma~\ref{L:EfullyNormalized} and Lemma~\ref{L:CentralizerTFullyCentralized}(a).
\begin{equation}\label{E:FullyNormFullyCent}
\E\mbox{ fully normalized }\Longrightarrow T\mbox{ fully centralized }\Longrightarrow \E \mbox{ fully centralized}.
\end{equation}
We give some examples in Section~\ref{S:Examples} to show firstly that the converse of the above implications does not hold, and secondly, that $T$ being fully normalized is neither necessary nor sufficient for $\E$ being fully normalized.

\smallskip

We now summarize the main properties of the centralizer subsystem in the following theorem.

\begin{Th}\label{T:FusionCentralizer}
 \begin{itemize}
  \item [(a)] $C_\F(\E)$ is a subsystem of $N_\F(\E)$ over $C_S(\E)$. Moreover, $C_\F(\E)$ commutes with $\E$, and every saturated subsystem of $\F$ which commutes with $\E$ is contained in $C_\F(\E)$.
  \item [(b)] Suppose $\E$ is fully $\F$-centralized. Then $C_\F(\E)$ is saturated and thus the largest saturated subsystem of $\F$ that commutes with $\E$. Moreover,
\[C_\F(\E)=\<O^p(\Aut_{C_\F(T)}(P))\colon P\leq C_S(\E)\>_{C_S(\E)}.\]
 \item [(c)] $C_\F(\E)$ is weakly $N_\F(\E)$-invariant in the sense of Definition~\ref{D:WeaklyInvariant} below. If $\E$ is fully $\F$-normalized, then $\E$ is fully $\F$-centralized and $C_\F(\E)=C_{N_\F(\E)}(\E)$ is normal in $N_\F(\E)$.
 \end{itemize}
\end{Th}

Theorem~\ref{T:FusionCentralizer} is proved after Lemma~\ref{L:CommuteWithE}. Centralizers of normal subsystems were defined before by Aschbacher \cite[Chapter~6]{Aschbacher:2011} with an alternative construction given in \cite{Henke:2018}. It follows from \cite[Proposition~6.7]{Chermak/Henke} that the centralizer $C_\F(\E)$ defined above coincides with the centralizer subsystem defined by Aschbacher if $\E$ is normal in $\F$. 

\smallskip

For more information on the centralizer $C_\F(\E)$, the reader is referred to Subsection~\ref{SS:Centralizer}. It seems difficult to obtain a neat description of this subsystem if $\E$ is not fully centralized. However, the subsystem $C_\F(\E)$ is uniquely determined by the properties stated in Theorem~\ref{T:FusionCentralizer}(b) and Lemma~\ref{L:MoveToFullyCentralized}. 

\smallskip

%\subsection{The relationship between normalizers and centralizers in regular localities and normalizers and centralizers in fusion systems}

We end this introduction by clarifying some aspects of the relationship between normalizers and centralizers in regular localities and normalizers and centralizers in fusion systems.

\begin{Remark}\label{R:Main}
Let $(\L,\Delta,S)$ be a regular locality over $\F$, $\H\subn\L$, $T:=S\cap \H$ and $\E:=\F_T(\H)$. We show in Lemma~\ref{L:finNLtT} and Proposition~\ref{P:CSE} that $N_S(\E)=N_S(\H)$ and $C_S(\E)=C_S(\H)$. 

\smallskip

Choose now $S_0\in N_G(\H)$. It follows from Theorem~\ref{T:MainbNLH}(a) that $\F_0:=\F_{S_0}(\bN_\L(\H))$ is well-defined and saturated. Moreover Theorem~\ref{T:MainbNLH}(a),(f) in combination with Background Theorem~\ref{bT:1} gives that $\C_0:=\F_{C_{S_0}(\H)}(C_\L(\H))$ is saturated. The reader might wonder how $\F_0$ and $\C_0$ relate to $N_\F(\E)$ and $C_\F(\E)$. 

\smallskip

It turns out that $\F_0\cong N_\F(\E)$ if and only if $\E$ is fully $\F$-normalized, and that $\C_0\cong C_\F(\E)$ if and only if $\E$ is fully $\F$-centralized. As $N_S(\H)$ is a $p$-subgroup of $N_G(\H)$, it is indeed possible and natural to choose $S_0$ such that $N_S(\H)\leq S_0$. Assuming this from now on, the situation is actually far more transparent. Namely, using Notation~\ref{N:restriction}, it follows from Lemma~\ref{L:NFEbNLH}, Lemma~\ref{L:CF0ECollect} and Lemma~\ref{L:CFEWeaklyInvariant} that
\[N_\F(\E)=\F\cap \F_0=\F_0|_{C_S(\H)}\mbox{ and }C_\F(\E)=\F\cap \C_0=\C_0|_{C_S(\H)}.\]
Moreover, it is a consequence of the above equations, Lemma~\ref{L:EfullyNormalized} and Lemma~\ref{L:EFullyCentralized} that we have the following equivalences:
\[\E\mbox{ is fully $\F$-normalized}\Longleftrightarrow N_S(\H)=S_0\Longleftrightarrow \F_0=N_\F(\E)\]
and
\[\E\mbox{ is fully $\F$-centralized}\Longleftrightarrow C_S(\H)=C_{S_0}(\H)\Longleftrightarrow \C_0=C_\F(\E).\]
Indeed, for the statements on the centralizers to hold it would be enough to assume $C_S(\H)\leq S_0$. 
\end{Remark}

\subsubsection*{Organization of the paper} In Section~\ref{S:Preliminaries} we summarize some more background and show auxiliary results. This is used in Section~\ref{S:Main} to prove the theorems stated in the introduction as well as some related properties. In particular, in Subsection~\ref{SS:ConjugateE}, Lemma~\ref{L:EFullyCentralized} and Lemma~\ref{L:CentralizerTFullyCentralized}(a) we study in more detail what it means for a subnormal subsystem to be fully normalized or fully centralized. These findings are illustrated by examples in Section~\ref{S:Examples}, which show also that some of our results cannot be improved in certain ways.

%The following results are proved along the way: If $(\L,\Delta,S)$ is a regular locality over $\F$ and $\H\subn\L$ with $S\cap\H=T$ and $\E=\F_T(\H)$, then the subgroup $C_S(\E)$ from above is actually equal to $C_S(\H)$. Moreover, we show in Lemma~\ref{L:EFullyCentralized} that $\E$ is fully centralized if and only if $C_S(\H)$ is a maximal $p$-subgroup of $C_\L(\H)$.

\section{Preliminaries}\label{S:Preliminaries}

\textbf{Throughout this section let $\F$ be a fusion system over a $p$-group $S$.}

\smallskip

We adapt the terminology and notation from \cite[Chapter~I]{Aschbacher/Kessar/Oliver:2011}, except that we write homomorphisms on the right hand side. We use moreover the following notation.

\begin{notation}\label{N:restriction}
Given $T\leq S$, we write $\F|_T$ for the full subcategory of $\F$ whose objects are all the subgroups of $T$.
\end{notation}

\subsection{Invariant subsystems} As in \cite[Definition~I.6.1]{Aschbacher/Kessar/Oliver:2011}, we say that a subsystem $\E$ of $\F$ over $T$ is \emph{$\F$-invariant} if $T$ is strongly closed, $\E^\alpha=\E$ for every $\alpha\in\Aut_\F(T)$ and the Frattini condition holds, i.e. for every $P\leq T$ and every $\phi\in\Hom_\F(P,T)$, there exists $\phi_0\in\Hom_\E(P,T)$ and $\alpha\in\Aut_\F(T)$ with $\phi=\phi_0\alpha$. There is also a different notion of $\F$-invariant subsystems in the literature. To distinguish the two, we introduce the following definition.

\begin{definition}\label{D:WeaklyInvariant}
 We call a subsystem $\E$ of $\F$ over $T$ \emph{weakly $\F$-invariant}, if $T$ is strongly closed and, for all $P\leq Q\leq T$, all $\phi\in\Hom_\E(P,Q)$ and all $\alpha\in\Hom_\F(Q,T)$, we have
\[\phi^\alpha:=(\alpha|_P)^{-1}\phi\alpha \in\Hom_\E(P\alpha,Q\alpha).\]
\end{definition}

If $\F$ is saturated, it is shown in \cite[Proposition~I.6.4]{Aschbacher/Kessar/Oliver:2011} that $\E$ is $\F$-invariant if and only if it is weakly $\F$-invariant. In general, the two notions are different, but we have the following lemma.

\begin{lemma}\label{L:Finvariant}
Let $\F$ be a fusion system over $S$ and $\E$ a subsystem of $\F$ over $T\leq S$.
\begin{itemize}
 \item [(a)] If $\E$ is $\F$-invariant, then $\E$ is weakly $\F$-invariant.
 \item [(b)] Suppose $T$ is strongly $\F$-closed and $\E^\alpha=\E$ for every $\alpha\in \Aut_\F(T)$. Then $T$ is $\F$-invariant if and only if $\F|_T=\<\E,\Aut_\F(T)\>_T$. 
\end{itemize}
\end{lemma}

\begin{proof}
This is shown in the proof of \cite[Proposition~I.6.4]{Aschbacher/Kessar/Oliver:2011} if $\F$ is saturated. The arguments do  not use the assumption that $\F$ is saturated.
\end{proof}

\subsection{Commuting Subsystems}

As in Definition~\ref{D:CommuteSubsystems}, we say that two subsystems $\F_i$ over $S_i$ ($i=1,2$) commute in $\F$ if $S_1\cap S_2\leq Z(\F_1)\cap Z(\F_2)$ and $\F_i\subseteq C_\F(S_{3-i})$ for each $i=1,2$. If $\F_1$ and $\F_2$ commute then $[S_1,S_2]=1$ and thus $S_1S_2$ is a subgroup of $S$. We set in this case
\[\F_1*\F_2:=\<\phi\in\Hom_\F(P_1P_2,S_1S_2)\colon P_i\leq S_i,\;\phi|_{P_i}\in\Hom_{\F_i}(P_i,S_i)\mbox{ for each }i=1,2\>_{S_1S_2}.\]
The following lemma could be concluded from the results in \cite{Chermak/Henke}, but for simplicity we write out part of the argument.

\begin{lemma}\label{L:CentralProduct}
Let $\F_1$ and $\F_2$ be saturated subsystems of $\F$ over $S_1$ and $S_2$ respectively such that $\F_1$ and $\F_2$ commute in $\F$. Then $\mD:=\F_1*\F_2$ is a saturated subsystem of $\F$ with $\mD|_{S_i}=\F_i$, $\F_i\unlhd \mD$  and $\F_i\subseteq C_\mD(S_{3-i})$ for each $i=1,2$.
\end{lemma}

\begin{proof}
By \cite[Proposition~3.3]{Henke:2018}, $\mD=\F_1*\F_2$ is a saturated subsystem of $\F$. %Moreover, by \cite[Lemma~3.2]{Henke:2018}, $S_1\cap S_2\leq Z(\F_1)\cap Z(\F_2)$.

\smallskip

Let $P_i\leq S_i$ and $\phi\in\Hom_\F(P_1P_2,S_1S_2)$ with $\phi|_{P_i}\in\Hom_{\F_i}(P_i,S_i)$ for $i=1,2$. Fix now $i\in\{1,2\}$. Notice that $\hat{P}_i:=(P_1P_2)\cap S_i=P_iZ$, where $Z:=P_{3-i}\cap S_i\leq S_1\cap S_2\leq Z(\F_1)\cap Z(\F_2)$. As $Z\leq P_{3-i}\cap Z(\F_{3-i})$ and $\phi|_{P_{3-i}}$ is a morphism in $\F_{3-i}$, we have in particular $\phi|_Z=\id_Z$. As $Z\leq Z(\F_i)$ and $\phi|_{P_i}$ is a morphism in $\F_i$, it follows that $\phi|_{\hat{P}_i}$ is a morphism in $\F_i$. This property implies that $S_i$ is strongly closed in $\mD$ and 
\[\mD|_{S_i}=\F_i.\]
In particular, $\F_i$ is $\mD$-invariant. By assumption, $\F_i$ is saturated and thus weakly normal in $\mD$.

\smallskip

As $\F_i\subseteq C_\F(S_{3-i})$, every morphism $\phi_i\in\Hom_{\F_i}(P_i,S_i)$ extends to a morphism 
\[\phi\in \Hom_\F(P_iS_{3-i},S_1S_2)\]
with $\phi|_{S_{3-i}}=\id_{S_{3-i}}$. By definition of $\mD$, $\phi$ is then a morphism in $\mD$. This shows that $\F_i\subseteq C_{\mD}(S_{3-i})$. Using this property for $P_i:=S_i$, it follows moreover that the pair $(\F_i,\mD)$ satisfies the extension condition for normal subsystems (cf. \cite[Definition~I.6.1]{Aschbacher/Kessar/Oliver:2011}). Hence, $\F_i$ is a normal subsystem of $\mD$.
\end{proof}

\subsection{Two lemmas on fusion systems}

\begin{lemma}\label{L:CentricRadicalIntersection}
If $\E$ is a subnormal subsystem of $\F$ over $T\leq S$ and $R\in\F^{cr}$, then $R\cap T\in\E^{cr}$.
\end{lemma}

\begin{proof}
If $\E$ is normal, this is shown in \cite[Lemma~1.20(d)]{AOV1} (or alternatively in \cite[Lemma~4.6]{Grazian/Henke}). The result follows thus by induction on the subnormal length of $\E$.
\end{proof}

\begin{lemma}\label{L:EE=EF}
Let $\E$ be a normal subsystem of $\F$ over $T\leq S$ with $E(\E)=E(\F)$. If $P\leq S$ with $P\cap T\in\delta(\E)$, then $P\in\delta(\F)$.
\end{lemma}

\begin{proof}
Since the set $\delta(\E)$ is overgroup-closed in $T$, we have $(PO_p(\F))\cap T\in\delta(\E)$. Hence, by \cite[Lemma~4.9(a)]{Grazian/Henke}, $PO_p(\F)\in\delta(\F)$. Using Background Theorem~\ref{bT:0}, it follows now from \cite[Lemma~10.6]{Henke:Regular} that $P\in\delta(\F)$.
\end{proof}

\subsection{Products of subnormal subsystems with $p$-groups}

If $\F$ is saturated and $\E$ is a normal subsystem of $\F$ over $T\leq S$, then a product subsystem $\E R$ was defined by Aschbacher \cite[Chapter~8]{Aschbacher:2011} for every subgroup $R\leq S$. A new construction of this product subsystems was given in \cite{Henke:2013}, which we generalize in the following definition.

\begin{definition}\label{D:ProductER}
Let $\E$ be a subsystem of $\F$ over $T\leq S$. 
\begin{itemize}
\item For $P\leq S$ set
\[\Ac_{\F,\E}(P):=\<\phi\in\Aut_\F(P)\colon \phi\mbox{ $p^\prime$-element },\;[P,\phi]\leq P\cap T,\;\phi|_{P\cap T}\in\Aut_\E(P\cap T)\>.\]
\item If $\F$ is saturated and $\E$ is subnormal in $\F$, then for every $R\leq N_S(\E)$ set 
\[\E R:=(\E R)_\F:=\<\Ac_{\F,\E}(P)\colon P\leq TR,\;P\cap T\in\E^{cr}\>_{TR}.\]
\end{itemize}
\end{definition}

\textbf{For the remainder of this section let $\F$ be a saturated fusion system over a $p$-group $S$.} 

\begin{prop}\label{P:ERnormal}
Suppose $\E$ is a normal subsystem of $\F$ over $T\leq S$. Then $N_S(\E)=S$. Moreover, for every $R\leq S$, the subsystem $\E R=(\E R)_\F$ is saturated and
\[\E R=\<\Ac_{\F,\E}(P)\colon P\leq TR,\;P\cap T\in\E^c\>_{TR}.\]
Moreover, $\E\unlhd \E R$ and $\E R$ is the unique saturated subsystem $\mD$ of $\F$ over $TR$ with $O^p(\mD)=O^p(\E)$. 
\end{prop}

\begin{proof}
As $\E\unlhd\F$, we have $T\unlhd S$ and $\Aut_S(T)\leq \Aut_\F(T)\leq \Aut(\E)$. Hence, $N_S(\E)=S$. Set now 
\[\G:=\<\Ac_{\F,\E}(P)\colon P\leq TR,\;P\cap T\in\E^c\>_{TR}.\]
Then by \cite[Theorem~1]{Henke:2013}, $\G$ is the unique saturated subsystem of $\F$ over $TR$ with $O^p(\G)=O^p(\E)$. By \cite[Lemma~2.33]{Chermak/Henke}, $\E\unlhd\G$. Hence, it remains to argue that $\G$ equals $\E R$. Clearly, $\E R$ is contained in $\G$. As $\G$ is saturated, it follows from Alperin's Fusion Theorem (cf. \cite[Theorem~I.3.6]{Aschbacher/Kessar/Oliver:2011}) that $\G$ is generated by all the automorphism groups $\Aut_{\G}(P)$, where $P\in \G^{cr}$ is fully $\G$-normalized. For every $P\in\G^{cr}$, Lemma~\ref{L:CentricRadicalIntersection} gives $P\cap T\in\E^{cr}$, and \cite[Lemma~4.7]{Henke:2013} implies therefore that $O^p(\Aut_\G(P))=\Ac_{\F,\E}(P)$ is contained in $\E R$. If $P\in\G^{cr}$ is fully $\G$-normalized, then $P$ is fully $\G$-automized and hence $\Aut_\G(P)=\Aut_{TR}(P)O^p(\Aut_\G(P))$ is contained in $\E R$. This shows that $\G\subseteq\E R$ and thus $\G=\E R$ as required.
\end{proof}

The following lemma can be seen as a Frattini argument for fusion systems. 

\begin{lemma}\label{L:FusionSystemFrattini}
Let $\E$ be a normal subsystem of $\F$ over $T\leq S$. Then 
\[\F=\<\E S,N_\F(T)\>.\]
More precisely, for every subgroup $P\leq S$ and every morphism $\phi\in\Hom_\F(P,S)$, there exist $\psi\in\Hom_{\E P}(P,TP)$ and $\alpha\in\Hom_{N_\F(T)}(P\psi,S)$ such that $\phi=\psi\alpha$.
\end{lemma}

\begin{proof}
We use throughout that, for every subgroup $R\leq S$, the subsystem $\E R$ defined above coincides by Proposition~\ref{P:ERnormal} with the equally denoted subsystem defined in \cite{Henke:2013}. It was first shown by Aschbacher \cite[1.3.2]{AschbacherFSCT} that $\F=\<\E S,N_\F(T)\>$; a proof using localities is given in \cite[Lemma~4.10]{Grazian/Henke} based on Lemma~\ref{L:FrattiniSplitting} below. We use the property $\F=\<\E S,N_\F(T)\>$ now to prove the second part of the claim. Fix $P\leq S$ and $\phi\in\Hom_\F(P,S)$.

\smallskip

By the definition of $\E S=(\E S)_\F$, every morphism in $\E S$ is a product of morphisms each of which is either a restriction of a morphisms in $\Ac_{\F,\E}(Q)$ for some $Q\leq S$ with $Q\cap T\in\E^{cr}$, or a morphism induced by conjugation with an element of $S$. Notice that morphisms induced by conjugation with elements of $S$ are in particular morphisms in $N_\F(T)$. Since $\F=\<\E S,N_\F(T)\>$, it follows that $\phi$ can be factorized as the product of morphisms, each of which is either in $N_\F(T)$ or a restriction of an automorphism in $\Ac_{\F,\E}(Q)$ for some $Q\leq S$ with $Q\cap T\in\E^{cr}$. By induction on the length of such a product, we may assume that $\phi$ can be factorized as
\[\phi=\psi\alpha\gamma_0\]
where $\psi\in\Hom_{\E P}(P,TP)$, $\alpha\in\Hom_{N_\F(T)}(P\psi,S)$, $\gamma_0\in\Hom_\F(P\psi\alpha,S)$ and either
\begin{itemize}
 \item [(i)] $\gamma_0$ is a morphism in $N_\F(T)$; or 
 \item [(ii)] $\gamma_0$ is a restriction of a morphism $\gamma\in\Ac_{\F,\E}(Q)$ for some $Q\leq S$ with $Q\cap T\in\E^{cr}$ and $P\psi\alpha\leq Q$. 
\end{itemize}
In case (i), $\alpha\gamma_0$ is a morphism in $N_\F(T)$ and the claim holds. Hence, we may assume that (ii) holds. It follows from the definition of $\E P=(\E P)_\F$ that $TP=T(P\psi)$. Moreover, $\alpha$ extends to $\hat{\alpha}\in\Hom_\F(T(P\psi),T(P\psi\alpha))$ with $T\hat{\alpha}=T$. Replacing $Q$ by $Q\cap T(P\psi\alpha)$ and $\gamma$ by $\gamma|_{Q\cap T(P\psi\alpha)}$, we may assume that $Q\leq T(P\psi\alpha)$. Notice that
\[\delta:=\hat{\alpha}\gamma\hat{\alpha}^{-1}\in\Ac_{\F,\E}(Q\hat{\alpha}^{-1})\]
with
\[\phi=\psi(\delta|_{P\psi})\hat{\alpha}.\]
By \cite[Lemma~2.6]{Chermak/Henke}, we have $(Q\hat{\alpha}^{-1})\cap T=(Q\cap T)\hat{\alpha}^{-1}\in\E^{cr}$. Moreover, $Q\hat{\alpha}^{-1}\leq T(P\psi)=TP$. Hence, $\delta$ is a morphism in $\E P$ and thus $\psi(\delta|_{P\psi})$ is a morphism in $\E P$. As $\hat{\alpha}$ is a morphism in $N_\F(T)$ the assertion follows.
\end{proof}

\subsection{Some lemmas on partial groups and localities}

The reader might want to recall the definitions and results introduced in Subsection~\ref{SS:BackgroundIntro}.

\begin{lemma}\label{L:PartialSubnormalInduce}
Let $\L$ be a partial group, $\H\subn\L$ and $\M$ a partial subgroup of $\L$. Then $\H\cap \M$ is a partial subnormal subgroup of $\M$. In particular, if $\M$ is a subgroup of $\L$, then $\H\cap \M$ is a subnormal subgroup of the group $\M$.
\end{lemma}

\begin{proof}
 If $\H=\H_0\unlhd\H_1\unlhd\cdots \unlhd \H_n=\L$ is a subnormal series for $\H$ in $\L$, then
 \[\H\cap \M=\H_0\cap \M\unlhd \H_1\cap \M\unlhd\cdots \unlhd\H_n\cap \M=\M\]
 is a subnormal series for $\H\cap \M$ in $\M$.
\end{proof}

The next lemma can be seen as a Frattini argument for localities and can be used to prove the first part of Lemma~\ref{L:FusionSystemFrattini} above.

\begin{lemma}\label{L:FrattiniSplitting}
 Let $(\L,\Delta,S)$ be a locality, $\N\unlhd\L$ and $T:=\N\cap S$. Then $\L=\N N_\L(T)$. Indeed,  for every $g\in\L$ there exist $n\in\N$ and $f\in N_\L(T)$ such that $(n,f)\in\D$, $g=nf$ and $S_g=S_{(n,f)}$.
\end{lemma}

\begin{proof}
This is \cite[Lemma~3.13]{Grazian/Henke}, which is proved based on Corollary~3.11 and Lemma~3.12 in \cite{Chermak:2015}.
\end{proof}

\begin{lemma}\label{L:AutLFusion}
 Let $(\L,\Delta,S)$ and $(\tL,\tDelta,\tS)$ be localities, and suppose $\alpha\colon \L\rightarrow \tL$ is an isomorphism of partial groups. Let moreover $\H$ be a partial subgroup of $\L$. Then the following hold:
\begin{itemize}
\item [(a)] If $S\alpha=\tS$, then $(S\cap \H)\alpha=\tS\cap (\H\alpha)$. 
\item [(b)] If $(S\cap \H)\alpha=\tS\cap (\H\alpha)$, then $\alpha|_{S\cap \H}$ induces an isomorphism from $\F_{S\cap \H}(\H)$ to $\F_{\tS\cap (\H\alpha)}(\H\alpha)$.
\end{itemize}
\end{lemma}

\begin{proof}
\textbf{(a)} As $\alpha$ is injective, $S\alpha=\tS$ implies $(S\cap \H)\alpha=(S\alpha)\cap (\H\alpha)=\tS\cap (\H\alpha)$.

\smallskip

\textbf{(b)} Assume now that $(S\cap \H)\alpha=\tS\cap (\H\alpha)$. Let $h\in \H$. It is sufficient to show that
\begin{equation}\label{E:Conjbyalpha}
\alpha^{-1}(c_h|_{S_h\cap \H})\alpha=c_{h\alpha}|_{\tS_{h\alpha}\cap (\H\alpha)} 
\end{equation}
To prove this let $x\in S_h\cap \H$. Then $x\in\D(h)$ and $x^h\in S$. As $\alpha$ is a homomorphism of partial groups, we have then $x\alpha\in\tilde{\D}(h\alpha)$ (where $\tilde{\D}$ denotes the domain of the product on $\tL$). Moreover, since $x^h\in S\cap \H$, we have
\[(x\alpha)^{h\alpha}=(x^h)\alpha\in (S\cap \H)\alpha=\tS\cap (\H\alpha).\]
In particular, as $x\alpha\in (S\cap \H)\alpha=\tS\cap (\H\alpha)$, it follows $x\alpha\in \tS_{h\alpha}$, which proves 
\[(S_h\cap \H)\alpha\leq \tS_{h\alpha}\cap (\H\alpha).\]
It follows from \cite[Lemma~2.18]{Henke:2020} that $\alpha^{-1}$ is also an isomorphism of partial groups. Hence, using the property we just proved with $\alpha^{-1}$ in place of $\alpha$, we can conclude that $(\tS_{h\alpha}\cap (\H\alpha))\alpha^{-1}\leq S_h\cap \H$ and thus $\tS_{h\alpha}\cap (\H\alpha)=(S_h\cap \H)\alpha$. As $(x\alpha)^{h\alpha}=(x^h)\alpha=(x\alpha)(\alpha^{-1}c_h\alpha)$, property \eqref{E:Conjbyalpha} follows.
\end{proof}

\begin{corollary}\label{C:AutLFusion}
 Let $(\L,\Delta,S)$ be a regular locality, $\H\subn\L$ and $\alpha\in\Aut(\L)$ with $(\H\cap S)\alpha\leq S$. Then $(\H\cap S)\alpha=(\H\alpha)\cap S$ and $\alpha$ induces an isomorphism from $\F_{S\cap \H}(\H)$ to $\F_{S\cap (\H\alpha)}(\H\alpha)$. 
\end{corollary}

\begin{proof}
It is a consequence of Background Theorem~\ref{bT:1} and the definition of a locality that $S\cap \H$ is a maximal $p$-subgroup of $\H$. Notice that $\alpha\in\Aut(\L)$ induces an isomorphism from $\H$ to $\H\alpha$. Hence, $(S\cap \H)\alpha$ is a maximal $p$-subgroup of $\H\alpha$. As $(S\cap \H)\alpha\leq S\cap (\H\alpha)$ by assumption, it follows $(S\cap \H)\alpha=S\cap (\H\alpha)$. The assertion follows now from Lemma~\ref{L:AutLFusion}(b).
\end{proof}

A linking locality is a locality $(\L,\Delta,S)$ such that $\F_S(\L)$ is saturated, $\F_S(\L)^{cr}\subseteq\Delta$ and $N_\L(P)$ is of characteristic $p$ for every $P\in\Delta$. If $\F$ is saturated, then the set $\F^s$ of \emph{subcentric subgroups} is the largest possible object set of a linking locality. It consists of the subgroups $P\leq S$ such that, for every fully $\F$-normalized $\F$-conjugate $Q$ of $P$, the normalizer $N_\F(Q)$ is constrained (cf. \cite[Proposition~3.1]{Henke:2015}). 

\smallskip

For every saturated fusion system $\F$, the set $\delta(\F)$ contains $\F^{cr}$ (see e.g. \cite[Lemma~10.4]{Henke:Regular}). Hence, every regular locality is a linking locality. For every linking locality $(\L,\Delta,S)$, there is the \emph{generalized Fitting subgroup} $F^*(\L)$ defined (cf. \cite[Definition~3]{Henke:Regular}). It is shown in \cite[Lemma~10.2]{Henke:Regular}) that, for any linking locality $(\L,\Delta,S)$, we have 
\begin{equation}\label{E:deltaF0}
 \delta(\F)=\{P\leq S\colon P\cap F^*(\L)\in\F^s\}.
\end{equation}
In particular, for every $P\leq S$, we have
\begin{equation}\label{E:deltaF}
 P\in\delta(\F)\Longleftrightarrow P\cap F^*(\L)\in\delta(\F).
\end{equation}
We will use this property for a regular locality $(\L,\Delta,S)$.

\smallskip

We set $Z(\L)=C_\L(\L)$ for every partial group $\L$. A linking locality $(\L,\Delta,S)$ is called \emph{quasisimple} if $\L=O^p(\L)$ (i.e. $\L\neq \N S$ for every proper partial normal subgroup $\N$ of $\L$) and $\L/Z(\L)$ is simple (i.e. $\L/Z(\L)$ has precisely two partial normal subgroups).  

\smallskip

Suppose now that $(\L,\Delta,S)$ is a regular locality. By Background Theorem~\ref{bT:1}, a partial subnormal subgroup $\K$ of $\L$ can be regarded as a regular locality. It is called a \emph{component} if it is quasisimple. By \cite[Proposition~11.7, Lemma~11.9]{Henke:Regular}, the product $E(\L)$ of components of $\L$ does not depend on the order of the factors and forms a partial normal subgroup of $\L$. Moreover, $F^*(\L)=E(\L)O_p(\L)$.% (cf. \cite[Theorem~4(a),(b),(c)]{Henke:Regular}).

\begin{lemma}\label{L:BasisNLtT}
Let $(\L,\Delta,S)$ be a regular locality. Set
\[\tT:=E(\L)\cap S\mbox{ and }T^*:=F^*(\L)\cap S.\]
Then $\tT$ and $T^*$ are elements of $\Delta$, and $G=N_\L(\tT)=N_\L(T^*)$ is a group of characteristic $p$. Moreover, $\L\subseteq \D(f)$ and $c_f\in\Aut(\L)$ for every $f\in G$, and $G$ acts on $\L$ via conjugation. 
\end{lemma}

\begin{proof}
By \cite[Lemma~11.9]{Henke:Regular}, we have $E(\L)\unlhd\L$ and $T^*=\tT O_p(\L)$. Therefore $G=N_\L(\tT)=N_\L(T^*)$. It follows from \eqref{E:deltaF} that $T^*\in\Delta$, which in turn implies by \cite[Lemma~10.6]{Henke:Regular} that $\tT\in\Delta$. In particular, $G$ is a group of characteristic $p$, as $(\L,\Delta,S)$ is a linking locality. By \cite[Theorem~2]{Henke:Regular}, $F^*(\L)$ is centric in $\L$. Hence, by \cite[Lemma~10.11(c)]{Henke:Regular}, we have
\[\L\subseteq\D(f)\mbox{ and }c_f\in\Aut(\L)\]
for all $f\in G=N_\L(T^*)$. Notice that $x^{\One}=x$ for all $x\in\L$. Moreover, if $f,g\in G$, we have $u=(g^{-1},f^{-1},x,f,g)\in\D$ via $(S_x\cap T^*)^{fg}$, as $S_x\cap T^*\in\delta(\F)=\Delta$ by \eqref{E:deltaF}. Hence, $x^{fg}=\Pi((fg)^{-1},x,(fg))=\Pi(u)=(x^f)^g$. Thus, $G=N_\L(T^*)$ acts on $\L$ via conjugation.
\end{proof}

\subsection{Some group theoretic lemmas}

Throughout this section let $G$ be a finite group.

\begin{lemma}\label{L:OpGinNormalizer}
If $H\subn G$, then $O_p(G)\leq N_G(O^p(H))$. 
\end{lemma}

\begin{proof}
Let $G$ be a minimal counterexample. Notice that $H\neq G$ as otherwise the assertion clearly holds. As $H\subn G$, there exists thus $N\unlhd G$ with $H\subn N$ and $N\neq G$. Then $N$ is not a counterexample and hence $[O_p(N),O^p(H)]\leq O^p(H)$. Notice that
\[[O_p(G),H]\leq [O_p(G),N]\leq O_p(G)\cap N=O_p(N).\]
Hence, by a property of coprime action (cf. \cite[8.2.7(b)]{KS}), we have 
\[[O_p(G),O^p(H)]=[O_p(G),O^p(H),O^p(H)]\leq [O_p(N),O^p(H)]\leq O^p(H)\]
contradicting the assumption that $G$ is a counterexample.
\end{proof}

Recall from the introduction that $G$ is of \emph{characteristic $p$} if $C_G(O_p(G))\leq O_p(G)$.

\begin{lemma}\label{L:MS}
Let $G$ be a group of characteristic $p$. Then the following hold:
\begin{itemize}
\item[(a)] $N_G(P)$ is of characteristic $p$ for all non-trivial $p$-subgroups $P \leq G$.
\item[(b)] Every subnormal subgroup of $G$ is of characteristic $p$.
\item[(c)] If $H\leq G$ with $O_p(G)\leq H$, then $H$ is of characteristic $p$.
\end{itemize}
\end{lemma}

\begin{proof}
 See \cite[Lemma 1.2(a),(c)]{MS:2012b} for parts (a) and (b). If $H$ is as in (c), then $O_p(G)\leq O_p(H)$ and hence $C_H(O_p(H))\leq C_G(O_p(G))\leq O_p(G)\leq O_p(H)$. 
\end{proof}

\begin{lemma}\label{L:CharpNGH}
Let $G$ be of characteristic $p$ and $H\subn G$. Then $N_G(H)$ is of characteristic $p$.
\end{lemma}

\begin{proof}
By Lemma~\ref{L:OpGinNormalizer}, we have $O_p(G)\leq N_G(O^p(H))$. Hence, $N_G(O^p(H))$ is of characteristic $p$ by Lemma~\ref{L:MS}(c). Notice that $N_G(H)\leq N_G(O^p(H))$ as $O^p(H)$ is characteristic in $H$. Moreover, $H\subn N_G(O^p(H))$. Hence, replacing $G$ by $N_G(O^p(H))$, we can and will assume from now on that $O^p(H)\unlhd G$.

\smallskip

Let $T\in\Syl_p(H)$. Since $H\unlhd N_G(H)$ and $H$ is by Lemma~\ref{L:MS}(b) of characteristic $p$, it is a consequence of \cite[Lemma~4.3]{Grazian/Henke} that $N_G(H)$ is of characteristic $p$ if and only if $N_{N_G(H)}(T)$ is of characteristic $p$. However, as $O^p(H)\unlhd G$ and $H=O^p(H)T$, we have $N_{N_G(H)}(T)=N_G(T)$. By Lemma~\ref{L:MS}(a), $N_G(T)$ is of characteristic $p$.
\end{proof}

Recall for the next lemma that a model for a saturated fusion system $\F$ is a finite group of characteristic $p$ which realizes $\F$. By \cite[Theorem~III.5.10]{Aschbacher/Kessar/Oliver:2011}, such a model for $\F$ exists if and only if $\F$ is constrained. Here $\F$ is called constrained if it is saturated and $C_S(O_p(\F))\leq O_p(\F)$.

\begin{lemma}\label{L:NormalizerConstrained}
Let $\F$ be a constrained fusion system. Suppose $G$ and $\tG$ are models for $\F$ so that $S$ is a Sylow $p$-subgroup of $G$ and of $\tG$ (as $\F$ is a fusion system over $S$). Let moreover $H\subn G$ and $\tH\subn \tG$ be such that $\F_{S\cap H}(H)=\F_{S\cap \tH}(\tH)$. Then 
\[\F_{N_S(H)}(N_G(H))=\F_{N_S(\tH)}(N_G(\tH)).\]
\end{lemma}

\begin{proof}
As $G$ and $\tG$ are models, it follows from the \cite[Theorem~III.5.10]{Aschbacher/Kessar/Oliver:2011} that there exists an isomorphism $\alpha\colon G\rightarrow \tG$ which restricts to the identity on $S$. If $\Delta$ is the set of all subgroups of $S$, then $(G,\Delta,S)$ and $(\tG,\Delta,S)$ are localities. Hence, using $\alpha|_S=\id_S$, it is a special case of Lemma~\ref{L:AutLFusion}(a) that
\begin{equation}\label{E:FX}
S\cap X=S\cap (X\alpha)\mbox{ and }\F_{S\cap X}(X)=\F_{S\cap X}(X)^\alpha=\F_{S\cap (X\alpha)}(X\alpha)\mbox{ for all }X\leq G.
\end{equation}
Notice that $H\alpha\subn\tG$. Moreover \eqref{E:FX} together with our assumption gives $\F_{S\cap \tH}(\tH)=\F_{S\cap H}(H)=\F_{S\cap (H\alpha)}(H\alpha)$. By \cite[Lemma~2.9]{Chermak/Henke}, $\tH$ is the unique subnormal subgroup of $\tG$ realizing $\F_{S\cap \tH}(\tH)$. Hence, $\tH=H\alpha$. It follows $N_G(H)\alpha=N_{\tG}(\tH)$. Hence, using \eqref{E:FX} now for $X=N_G(H)$, we obtain $\F_{N_S(H)}(N_G(H))=\F_{N_S(\tH)}(N_G(\tH))$.
\end{proof}

\section{Normalizers and centralizers}\label{S:Main}

Throughout this section we assume the following hypothesis.

\begin{hypo}\label{H:main}
Let $(\L,\Delta,S)$ be a regular locality and $\H\subn\L$. Set 
\[\tT:=E(\L)\cap S,\;G:=N_\L(\tT),\]
\[T:=\H\cap S,\;\E:=\F_T(\H),\;T_0:=E(\H)\cap S,\]
\[\M:=\left(\prod_{\K\in\Comp(\L)\backslash\Comp(\H)}\K\right)\mbox{ and }T_0^\perp:=\M\cap S.\]
\end{hypo}
We use above and throughout this section that, by Background Theorem~\ref{bT:1}, $\E$ is saturated and $(\H,\delta(\E),T)$ is a regular locality. In particular, $\Comp(\H)$ and $E(\H)\unlhd\H$ are well-defined. We use also that the order of the factors in a product of components does not matter and thus $\M$ is well-defined (cf. \cite[Proposition~11.7]{Henke:Regular}). Consequently, $T_0$ and $T_0^\perp$ are well-defined.

\subsection{Some basic results}

\begin{lemma}\label{L:BasicT0perp}
\begin{itemize}
 \item [(a)] $E(\H)$ and $\M$ are partial normal subgroups of $E(\L)$ with $E(\L)=E(\H)\M$. Moreover, $\tT=T_0T_0^\perp$, $\H\subseteq C_\L(\M)$ and $\M\subseteq C_\L(\H)$. 
 \item [(b)] We have $\Comp(\M)=\Comp(\L)\backslash\Comp(\H)$, i.e. $\Comp(\L)$ is the disjoint union of $\Comp(\M)$ and $\Comp(\H)$. 
 \item [(c)] $\H\M$ is a partial subnormal subgroup of $\L$ with $\H\unlhd \H\M$ and $\M\unlhd \H \M$. 
\end{itemize}
\end{lemma}

\begin{proof}
\textbf{(a,c)} We refer in this proof to the notion of an (internal) central product, which was defined in \cite[Definition~4.1]{Henke:Regular}. However, we will not need to go into the details of the definition.

\smallskip

By \cite[Remark~11.2(b)]{Henke:Regular}, we have $\Comp(\H)\subseteq\Comp(\L)$. By \cite[Lemma~4.5]{Henke:Regular}, forming the internal central product is associative. Hence, it follows from \cite[Theorem~11.18(a)]{Henke:Regular} that $E(\L)$ is a central product of $E(\H)$ and $\M$, and from \cite[Theorem~11.18(c)]{Henke:Regular} that $\H\M$ is subnormal in $\L$ and a central product of $\H$ and $\M$. In particular, by \cite[Lemma~4.9]{Henke:Regular}, $E(\H)$ and $\M$ are partial normal subgroups of $E(\L)=E(\H)\M$, and $\H$ and $\M$ are partial normal subgroups of $\H\M$. It is moreover a consequence of \cite[Lemma~4.8]{Henke:Regular} that $\H\subseteq C_\L(\M)$ and $\M\subseteq C_\L(\H)$.

\smallskip

By Background Theorem~\ref{bT:1}, $E(\L)$ can be regarded as a regular locality with Sylow subgroup $\tT$. Hence, it follows from \cite[Theorem~5.1]{Chermak:2015} (applied with $E(\L)$ in place of $\L$) that
\[\tT=(E(\H)\cap \tT)(\M\cap\tT)=T_0T_0^\perp.\]

\smallskip

\textbf{(b)} Part (a) yields that $\M\unlhd E(\L)\unlhd\L$ is subnormal in $\L$. In particular,  by Background Theorem~\ref{bT:1}, $\M$ can be given the structure of a regular locality and thus $\Comp(\M)$ is well-defined. Using \cite[Remark~11.2]{Henke:Regular}, one sees that $\Comp(\L)\backslash\Comp(\H)\subseteq\Comp(\M)$ and $\Comp(\M)\subseteq\Comp(\L)$. If $\K\in\Comp(\M)\cap \Comp(\H)$, then part (a) yields that $\K=Z(\K)$, which contradicts $\K$ being quasisimple. Hence, $\Comp(\M)\subseteq\Comp(\L)\backslash\Comp(\H)$ and the assertion holds. 
\end{proof}

By Lemma~\ref{L:BasisNLtT}, $G$ is a subgroup of $\L$ which acts on $\L$ via conjugation. Thus, for every subset $\X$ of $\L$, the set-wise stabilizer 
\[N_{G}(\X):=\{f\in G\colon \X^f=\X\}\]
is a subgroup of $G$. In particular, $N_G(\H)$ is a subgroup of $G$. As $S\leq G$,  it follows that $N_S(\H):=S\cap N_G(\H)$ is a subgroup of $S$.

\begin{lemma}\label{L:NinNLtTofH}
The following hold:
\begin{itemize}
\item [(a)] $N_\H(\tT)=N_\H(T_0)\subn G$.
\item [(b)] $N_\F(\tT)$ is constrained and $G$ is a model for $N_\F(\tT)$. Moreover, $N_\H(\tT)\subn G$ is a model for $N_\E(T_0)$. In particular, 
\[N_\E(T_0)\subn N_\F(\tT).\]
\item [(c)] $N_G(\H)=N_G(N_\H(\tT))=N_G(N_\H(T_0))$ is of characteristic $p$. 
\item [(d)] $N_{E(\L)}(\tT)\subseteq N_{G}(\H)$.
\item [(e)] $N_{G}(\H)\subseteq N_{G}(E(\H))=N_{G}(T_0)=N_{G}(\M)=N_{G}(T_0^\perp)$.
\item [(f)] If $g\in G$ with $T_0^g\leq T$, then $g\in N_G(E(\H))$.
\end{itemize}
\end{lemma}

\begin{proof}
\textbf{(a)} Using Lemma~\ref{L:BasicT0perp}(a), one can observe that $\H\subseteq C_\L(\M)\subseteq C_\L(T_0^\perp)$ and then $N_\H(T_0)\subseteq N_\H(T_0T_0^\perp)=N_\H(\tT)$. As $T_0\leq \tT$, we have $T_0=E(\H)\cap \tT$. Thus, $E(\H)\unlhd\H$ implies that $T_0$ is normalized by $N_\H(\tT)$. This proves $N_\H(\tT)=N_\H(T_0)$. It follows from Lemma~\ref{L:PartialSubnormalInduce} that $N_\H(\tT)=\H\cap G$ is a subnormal subgroup of $G$. This shows (a).

\smallskip

\textbf{(b)} By Lemma~\ref{L:BasisNLtT}, $\tT\in\Delta$ and the group $G$ is of characteristic $p$. By \cite[Lemma~3.10(b)]{Grazian/Henke}, we have $N_\F(\tT)=\F_S(G)$, i.e. $N_\F(\tT)$ is constrained and $G$ is a model for $N_\F(\tT)$. The same argument with $(\H,\delta(\E),T)$ and $\E$ in place of $(\L,\Delta,S)$ and $\F$ yields that $N_\H(T_0)$ is a model for $N_\E(T_0)$. By (a), $N_\H(T_0)=N_\H(\tT)\subn G$. In particular, $N_\E(T_0)\subn N_\F(\tT)$ by \cite[Proposition~I.6.2]{Aschbacher/Kessar/Oliver:2011}. This proves (b).

\smallskip

\textbf{(f)} Let $g\in G$ with $T_0^g\leq T$ and $\K\in \Comp(\H)$. Recall that $c_g\in\Aut(\L)$ by Lemma~\ref{L:BasisNLtT} 
and so $\K^g\in\Comp(\L)$ by \cite[Lemma~11.12]{Henke:Regular}. Notice moreover that $\K\cap S\leq T_0$ and thus $(\K\cap S)^g\leq T_0^g\leq T$. If $\K^g$ were not a component of $\H$, then the definition of $\M$ and Lemma~\ref{L:BasicT0perp}(a) would imply that $\K^g\subseteq \M\subseteq C_\L(\H)\subseteq C_\L(T)$ and thus $(\K\cap S)^g$ were abelian. As $\K$ is quasisimple and $\K\cap S\cong (\K\cap S)^g$, this would contradict \cite[Lemma~7.10]{Henke:Regular}. Hence, $\K^g\in\Comp(\H)$. As $\K$ was arbitrary, this yields $E(\H)^g=E(\H)$.

\smallskip

\textbf{(c)} Using parts (a) and (b), Lemma~\ref{L:CharpNGH} yields that $N_{G}(N_\H(\tT))=N_G(N_\H(T_0))$ is of characteristic $p$. Hence, it remains only to show that $N_{G}(\H)=N_{G}(N_\H(\tT))$.

\smallskip

Clearly, $N_G(\H)\leq N_G(\H\cap G)=N_G(N_\H(\tT))$. Fix now $g\in N_G(N_\H(\tT))$. Then $T_0^g\leq N_\H(\tT)\cap \tT\leq \H\cap S=T$. Hence, by (f), we have $E(\H)^g=E(\H)$. It follows from the Frattini Lemma \cite[Corollary~3.11]{Chermak:2015} applied with $(\H,\delta(\E),T)$ in place of $(\L,\Delta,S)$ that $\H=E(\H)N_\H(T_0)=E(\H)N_\H(\tT)$. Using again that $c_g\in\Aut(\L)$ by Lemma~\ref{L:BasisNLtT}, we can conclude that $\H^g=E(\H)^gN_\H(\tT)^g=E(\H)N_\H(\tT)=\H$. Hence, $N_{G}(N_\H(\tT))\leq N_{G}(\H)$ and so equality holds.  

\smallskip

\textbf{(d)}  By Lemma~\ref{L:BasicT0perp}(a), $E(\H)\unlhd E(\L)$. In particular, $N_{E(\L)}(\tT)$ normalizes $T_0=E(\H)\cap \tT$. Lemma~\ref{L:BasicT0perp}(a) gives also that $E(\L)=E(\H)\M$ and $\M\subseteq C_\L(\H)\subseteq C_\L(T_0)\subseteq N_\L(T_0)$. Hence, using the Dedekind argument for localities \cite[Lemma~1.10]{Chermak:2015} and part (a), we conclude that
\[N_{E(\L)}(\tT)\subseteq N_{E(\L)}(T_0)=N_{E(\H)}(T_0)\M=N_{E(\H)}(\tT)\M.\]
Another application of the Dedekind argument gives then $N_{E(\L)}(\tT)=N_{E(\H)}(\tT)N_\M(\tT)$. As $N_{E(\H)}(\tT)\subseteq N_\H(\tT)\subseteq N_G(N_\H(\tT))$ and $\M\subseteq C_\L(\H)$, this implies $N_{E(\L)}(\tT)\subseteq N_{G}(N_\H(\tT))$. Now (d) follows from (c).

\smallskip

\textbf{(e)} By Lemma~\ref{L:BasisNLtT}, we have $c_g\in\Aut(\L)$ for every $g\in G$. Hence, for $g\in N_G(\H)$, $c_g$ restricts to an automorphism of $\H$. It follows now from \cite[Lemma~11.12]{Henke:Regular} that $G$ acts on $\Comp(\L)$, and that $N_G(\H)$ acts   on $\Comp(\H)$. In particular, $N_G(\H)\leq N_G(E(\H))$. Similarly, $N_G(E(\H))$ acts on $\Comp(\H)=\Comp(E(\H))$ (cf. \cite[Remark~11.2]{Henke:Regular}). Hence, $N_G(E(\H))$ acts also on $\Comp(\L)\backslash \Comp(\H)$, which implies $N_G(E(\H))\leq N_G(\M)$. Notice that $N_G(E(\H))\leq N_G(E(\H)\cap \tT)=N_G(T_0)$. By (f), we have $N_G(T_0)\leq N_G(E(\H))$. We know thus that 
\[N_G(E(\H))=N_G(T_0)\leq N_G(\M).\]
By Lemma~\ref{L:BasicT0perp}(a),(b), $\M\unlhd E(\L)\unlhd\L$ is subnormal in $\L$, $\Comp(\H)=\Comp(\L)\backslash\Comp(\M)$ and $\M=E(\M)$. Hence, applying the same property with $\M$ in place of $\H$, we obtain $N_G(\M)=N_G(T_0^\perp)\leq N_G(E(\H))$. This shows (e).
\end{proof}

\begin{lemma}\label{L:NLT}
Let $f\in N_\L(T)$. Then $\H\subseteq \D(f)$. Moreover, there exist $n\in E(\L)$ and $g\in G$ such that $(n,g)\in\D$, $f=ng$ and $\H^f=\H^g$.
\end{lemma}

\begin{proof}
By Lemma~\ref{L:FrattiniSplitting}, there exist $n\in E(\L)$ and $g\in G$ such that $(n,g)\in\D$, $f=ng$ and $S_f=S_{(n,g)}$. As $f\in N_\L(T)$, we have then $T\leq S_f\leq S_n$. By Lemma~\ref{L:BasicT0perp}(c), $\H\M$ is a partial subnormal subgroup of $\L$ with $\H\unlhd\H\M$. Notice that $n\in E(\L)=E(\H)\M\subseteq \H\M$ by Lemma~\ref{L:BasicT0perp}(a). So $T^n\subseteq S\cap\H=T$ and thus $n\in N_{E(\L)}(T)\subseteq N_{\H\M}(T)$. As $\H\M\subn\L$, Background Theorem~\ref{bT:1} yields that $\H\M$ can be given the structure of a regular locality with Sylow subgroup $S\cap(\H\M)$. Hence, by \cite[Theorem~10.16(f)]{Henke:Regular} applied with $\H\M$ in place of $\L$, we have $N_{\H\M}(T)\subseteq N_{\H\M}(\H)$. In particular, $n\in N_\L(\H)$. Fixing $h\in\H$, it follows $(n^{-1},h,n)\in\D$ and thus $Q:=S_{(n^{-1},h,n)}\in\Delta$ by \cite[Corollary~2.7]{Chermak:2015}. As $(\L,\Delta,S)$ is regular, we have $\Delta=\delta(\F)$ and \eqref{E:deltaF} yields $Q\cap F^*(\L)\in\Delta$. By Lemma~\ref{L:BasisNLtT}, $g\in N_\L(F^*(\L)\cap S)$ and then $v=(g^{-1},n^{-1},h,n,g)\in\D$ via $(Q\cap F^*(\L))^g$. By \cite[Lemma~1.4(f)]{Chermak:2015}, $g^{-1}n^{-1}=f^{-1}$. Hence, using the ``associativity axiom'' of partial groups, we can conclude that $(f^{-1},h,f)\in\D$ and $h^f=\Pi(v)=(h^n)^g$. Since this holds for an arbitrary element $h\in\H$, this shows $\H\subseteq \D(f)$ and $\H^f=(\H^n)^g=\H^g$, where the last equality uses $n\in N_\L(\H)$. This proves the assertion.
\end{proof}

\subsection{The regular normalizer and the centralizer of $\H$ in $\L$}\label{SS:NLH}

Recall from the introduction that
\[\bN_\L(\H):=E(\L)N_{G}(\H)\] 
is called the \emph{regular normalizer} of $\H$ in $\L$.  

\smallskip

In this subsection we will prove Theorem~\ref{T:MainbNLH} and some related results. Before we do so let us consider the special case that $\F$ is constrained.

\begin{remark}\label{R:Constrained}
Let $\F$ be constrained. Then by \cite[Lemma~11.6]{Henke:Regular}, $E(\L)=1$ and $\L$ is a model for $\F$. In particular, $\H$ is a subnormal subgroup of $\L$, $\tT=1$, $G=\L$  and $\bN_\L(\H)=N_\L(\H)$.
\end{remark}

We are now going back to the general case, for which we summarize some important results in the next theorem.

\begin{theorem}\label{T:bNLH}
Fix $S_0\in\Syl_p(N_{G}(\H))$ and set
\[\Delta_0:=\{P\leq S_0\colon P\cap\tT\in E(\F)^s\}.\]
Then the following hold:
\begin{itemize}
 \item [(a)] The regular normalizer $\bN_\L(\H)$ is a partial subgroup of $\L$ and $\tT\leq S_0$. Moreover, the triple $(\bN_\L(\H),\Delta_0,S_0)$ is a linking locality and the fusion system $\F_0:=\F_{S_0}(\bN_\L(\H))$ is saturated. Indeed, $(\bN_\L(\H),\delta(\F_0),S_0)$ is a regular locality over $\F_0$ and
\[\delta(\F_0)=\{P\leq S_0\colon PO_p(\bN_\L(\H))\in\Delta_0\}.\]
 \item [(b)] $\H\unlhd\bN_\L(\H)$. As $T$ is a $p$-subgroup of $N_G(\H)$, we may choose $S_0$ such that $T\leq S_0$. For such a choice of $S_0$, we have $T=\H\cap S_0$ and $\E\unlhd\F_0$.
 \item [(c)] $\Comp(\bN_\L(\H))=\Comp(\L)$, $E(\bN_\L(\H))=E(\L)$ and $\tT=E(\L)\cap S_0$. In particular, $\Comp(\F)=\Comp(\F_0)$ and  $E(\F)=E(\F_0)$.
 \item [(d)] $E(\H)\unlhd\bN_\L(\H)$, and $\M\unlhd\bN_\L(\H)$. In particular, $E(\E)\unlhd \F_0$ and $\F_{T_0^\perp}(\M)\unlhd \F_0$.
 \item [(e)] $N_{\bN_\L(\H)}(\tT)=N_{G}(\H)$. In particular, $S\cap \bN_\L(\H)=N_S(\H)$.
\end{itemize}
\end{theorem}

\begin{proof}
\textbf{(a)} Set $T^*:=F^*(\L)\cap S$ and recall from Lemma~\ref{L:BasisNLtT} that $G=N_\L(T^*)$. So $H:=N_{G}(\H)\leq N_\L(T^*)$. By Lemma~\ref{L:NinNLtTofH}(d), we have
\[N_{E(\L)}(T^*)=N_{E(\L)}(\tT)\leq H.\]
In particular, $\tT$ is contained in $H$ and then a normal $p$-subgroup of $H$. Therefore $\tT\leq S_0\in\Syl_p(H)$ and so $\tT\leq S_0\cap E(\L)$. As $\tT=S\cap E(\L)$ is a maximal $p$-subgroup of $E(\L)$ by \cite[Lemma~3.1(c)]{Chermak:2015}, this yields
\begin{equation}\label{E:tT}
 \tT=S_0\cap E(\L).
\end{equation}
By Lemma~\ref{L:NinNLtTofH}(c), $H$ is of characteristic $p$. By Background Theorem~\ref{bT:2}, we have $\F_{S\cap E(\L)}(E(\L))=E(\F)$ and by \cite[Lemma~7.22]{Chermak/Henke}, $\delta(E(\F))=E(\F)^s$. Hence,
\[\Delta_0=\{P\leq S_0\colon P\cap E(\L)\in\delta(\F_{S\cap E(\L)}(E(\L)))\}.\] 
It follows now from \cite[Theorem~6.1]{Grazian/Henke} (applied with $E(\L)$ in place of $\N$)\footnote{The reader might want to observe that the arguments to prove \cite[Theorem~6.1]{Grazian/Henke} become actually easier in the case that $\N$ equals $E(\L)$.} that $\bN_\L(\H)$ is a partial subgroup of $\L$, $(\bN_\L(\H),\Delta_0,S_0)$ is a linking locality, $\F_0:=\F_{S_0}(\bN_\L(\H))$ is saturated and $(\bN_\L(\H),\delta(\F_0),S_0)$ is a regular locality over $\F_0$, where
\[\delta(\F_0)=\{P\leq S_0\colon PO_p(\bN_\L(\H))\in\Delta_0\}.\]
This shows (a). 

\smallskip

We will use now throughout that $\F_0$ is well-defined and saturated, and that $(\bN_\L(\H),\delta(\F_0),S_0)$ is a regular locality over $\F_0$.

\smallskip

\textbf{(e)} By the Dedekind Lemma \cite[Lemma~1.10]{Chermak:2015}, $N_{\bN_\L(\H)}(\tT)=N_{E(\L)}(\tT)N_{G}(\H)$. Hence, Lemma~\ref{L:NinNLtTofH}(d) implies that  $N_{\bN_\L(\H)}(\tT)=N_{G}(\H)$.  As $\tT\unlhd S$, it follows in particular that $S\cap \bN_\L(\H)=S\cap N_{\bN_\L(\H)}(\tT)=S\cap N_G(\H)=N_S(\H)$.

\smallskip

\textbf{(b)} We prove first that $\H$ is contained in $\bN_\L(\H)$. It follows now from the Frattini argument for localities (Lemma~\ref{L:FrattiniSplitting}) applied with $(\H,\delta(\E),T)$ in place of $(\L,\Delta,S)$ that $\H=E(\H)N_\H(T_0)$. As $E(\H)\subseteq E(\L)$ by  Lemma~\ref{L:BasicT0perp}(a) and $N_\H(T_0)=N_\H(\tT)\subseteq N_{G}(\H)$ by Lemma~\ref{L:NinNLtTofH}(a), it follows $\H\subseteq\bN_\L(\H)$.

\smallskip

As $\H$ is a partial subgroup of $\L$, it is also a partial subgroup of $\bN_\L(\H)$. We prove now that $\H$ is a partial normal subgroup. Let $f\in\bN_\L(\H)$ and $h\in \H$ such that $h\in\D(f)$. We need to show that $h^f\in\H$.

\smallskip

Clearly $E(\L)\unlhd \bN_\L(\H)$ as $E(\L)\unlhd\L$ is contained in $\bN_\L(\H)$. Moreover, \eqref{E:tT} gives $\tT=E(\L)\cap S_0$. Hence, by Lemma~\ref{L:FrattiniSplitting} applied with $(\bN_\L(\H),\delta(\F_0),S_0)$ in place of $(\L,\Delta,S)$, there exists $x\in E(\L)$ and $g\in N_{\bN_\L(\H)}(\tT)$ such that $(x,g)\in\D$, $f=xg$ and
\[(S_0)_f=(S_0)_{(x,g)}.\]
It follows from \cite[Lemma~1.4(f), Proposition~2.6(b)]{Chermak:2015} that $(g^{-1},x^{-1})\in\D\cap \W(\bN_\L(\H))$ via $(S_0)_{f^{-1}}=(S_0)_f^f=(S_0)_{(x,g)}^{xg}$ and $f^{-1}=g^{-1}x^{-1}$. So $u:=(g^{-1},x^{-1},h,x,g)\in\D\cap \W(\bN_\L(\H))$ via $(S_0)_{(f^{-1},x,f)}$ and, by the axioms of a partial group, $h^f=\Pi(f^{-1},h,f)=\Pi(u)=(h^x)^g$.

\smallskip

By Lemma~\ref{L:BasicT0perp}(a),(b), $E(\L)=E(\H)\M\subseteq \H\M$ and $\H\unlhd \H\M$. In particular, $h^x\in\H$. By (e), $g\in N_{G}(\H)$. Thus, it follows $h^f=(h^x)^g\in\H$ as required. Thus, $\H\unlhd \bN_\L(\H)$.

\smallskip

Suppose now $T\leq S_0$. Then $T\leq \H\cap S_0$. As $T=S\cap \H$ is by Background Theorem~\ref{bT:1} a maximal $p$-subgroup of $\H$, it follows $T=\H\cap S_0$. As $\H$ is a partial normal subgroup of $\bN_\L(\H)$, Background Theorem~\ref{bT:1} together with part (a) implies now $\E=\F_T(\H)=\F_{S_0\cap\H}(\H)\unlhd \F_0$.

\smallskip

\textbf{(c)} As every component of $\L$ is subnormal in $\L$ and contained in the partial subgroup $\bN_\L(\H)$ of $\L$, it is a special case of Lemma~\ref{L:PartialSubnormalInduce} that every component of $\L$ is subnormal in $\bN_\L(\H)$. Thus $\Comp(\L)\subseteq\Comp(\bN_\L(\H))$.

\smallskip

Suppose now that there is a component $\K\in\Comp(\bN_\L(\H))\backslash\Comp(\L)$. As $E(\L)\unlhd\bN_\L(\H)$, it follows from \cite[Lemma~11.16]{Henke:Regular} (applied with $(\bN_\L(\H),E(\L))$ in place of $(\L,\N)$ and combined with \cite[Theorem~10.16(e)]{Henke:Regular}) that 
\[\K\subseteq C_{\bN_\L(\H)}(E(\L))\subseteq C_{\bN_\L(\H)}(\tT)\subseteq N_{\bN_\L(\H)}(\tT).\]
By (e) and Lemma~\ref{L:NinNLtTofH}(c), $N_{\bN_\L(\H)}(\tT)=N_{G}(\H)$ is a group of characteristic $p$. As $\K$ is a partial subnormal subgroup of $\bN_\L(\H)$, it follows from Lemma~\ref{L:PartialSubnormalInduce} that $\K=\K\cap N_{\bN_\L(\H)}(\tT)$ is a subnormal subgroup of the group $N_{\bN_\L(\H)}(\tT)$. Thus $\K$ is by Lemma~\ref{L:MS}(b) of characteristic $p$. However, by \cite[Lemma~7.10]{Henke:Regular}, the fact that $\K$ is a quasisimple locality implies $O_p(\K)=Z(\K)$. This contradicts $\K$ being of characteristic $p$ and shows thus that
\[\Comp(\L)=\Comp(\bN_\L(\H)).\]
In particular, $E(\L)=E(\bN_\L(\H))$. We have seen in \eqref{E:tT} already that $E(\L)\cap S_0=\tT$. It follows thus from Background Theorem~\ref{bT:2} that $\Comp(\F)=\Comp(\F_0)$ and  $E(\F)=E(\F_0)$.

\smallskip

\textbf{(d)} As $\H\unlhd\bN_\L(\H)$ by part (b), it follows from \cite[Lemma~11.13]{Henke:Regular} that $E(\H)\unlhd\bN_\L(\H)$. By Lemma~\ref{L:BasicT0perp}(a), $\M\unlhd E(\L)$. Moreover, by (e) and Lemma~\ref{L:NinNLtTofH}(e), $N_{\bN_\L(\tT)}(\tT)=N_G(\H)\leq N_G(\M)$. As $E(\L)\unlhd\bN_\L(\H)$ and $S_0\cap E(\L)=\tT$, it follows now from \cite[Corollary~3.13]{Chermak:2015} applied with $\bN_\L(\H)$ in place of $\L$ that $\M\unlhd \bN_\L(\H)$.

\smallskip

By Background Theorem~\ref{bT:1}, $T_0$ is a maximal $p$-subgroup of $E(\H)$ and $T_0^\perp$ is a maximal $p$-subgroup of $\M$. At the same time, as $\tT\leq S_0$ by \eqref{E:tT}, we have $T_0\leq E(\H)\cap S_0$ and $T_0^\perp\leq \M\cap S_0$. Hence, it follows
\[T_0=E(\H)\cap S_0\mbox{ and }T_0^\perp=\M\cap S_0.\]
It is a consequence of Background Theorem~\ref{bT:2} (applied with $(\H,\delta(\E),T)$ in place of $(\L,\Delta,S)$) that $E(\E)=\F_{T_0}(E(\H))$. Background Theorem~\ref{bT:2} (applied with $(\bN_\L(\H),\delta(\F_0),S_0)$ in place of $(\L,\Delta,S)$) gives now  $E(\E)\unlhd \F_0$ and $\F_{T_0^\perp}(\M)\unlhd\F_0$.
\end{proof}

\begin{lemma}\label{L:X}
 Let $\X$ be a partial subgroup of $\L$ such that $\H\unlhd\X$ and $(\X,\Gamma_\X,S_\X)$ is a locality for some $p$-subgroup $S_\X$ of $\X$ with $T\leq S_\X$ and some set $\Gamma_\X$ of subgroups of $S_\X$. Then $\X\subseteq \bN_\L(\H)$.
\end{lemma}

\begin{proof}
It is a consequence of Background Theorem~\ref{bT:1} (and the definition of a locality) that $T$ is a maximal $p$-subgroup of $\H$. As $T\leq S_\X\cap\H$, it follows thus that $T=S_\X\cap \H$. The Frattini argument for localities (Lemma~\ref{L:FrattiniSplitting}), applied with $(\X,S_\X,\Delta_\X)$ in place of $(\L,\Delta,S)$, gives therefore $\X=\H N_\X(T)$. By Theorem~\ref{T:bNLH} $\bN_\L(\H)$ is a partial subgroup of $\L$ with $\H\subseteq \bN_\L(\H)$. Thus it remains to show that $N_\X(T)\subseteq \bN_\L(\H)$. 

\smallskip

Let $f\in N_\X(T)$. By Lemma~\ref{L:NLT}, we have $\H\subseteq\D(f)$ and there exist $n\in E(\L)$ and $g\in G$ such that $(n,g)\in\D$, $f=ng$ and $\H^f=\H^g$. As $\H\unlhd \X$, it follows $\H^g=\H^f=\H$ and thus $g\in N_G(\H)$. Hence, $f=ng\in E(\L)N_G(\H)=\bN_\L(\H)$. This completes the proof.
\end{proof}

For the next lemma, the reader might want to recall the definition of $N_S(\E)$ from Notation~\ref{N:FusionBasic}. We recall furthermore that $N_\E(T_0)$ is by Lemma~\ref{L:NinNLtTofH}(b) a subnormal subsystem of $N_\F(\tT)$, which is a saturated fusion system on $S$. Thus, $N_S(N_\E(T_0))$ is also defined.

\begin{lemma}\label{L:finNLtT}
Let $f\in N_{G}(T)$. Then $f\in N_{G}(\H)$ if and only if $c_f|_T\in\Aut(\E)$. In particular, \[N_S(\E)=N_S(\H)=N_S(N_\H(\tT))=N_S(N_\E(T_0)).\]
\end{lemma}

\begin{proof}
We show first that 
\begin{equation}\label{E:NSE}
N_S(\E)=N_S(\H). 
\end{equation}
For the proof recall that $c_f\in\Aut(\L)$ by Lemma~\ref{L:BasisNLtT}. In particular, as $\H\subn\L$, we have $\H^f\subn\L$. Moreover, setting $\alpha:=c_f|_T$, it follows from Lemma~\ref{C:AutLFusion} that $S\cap \H^f=T^f=T$ and 
\[\F_T(\H^f)=\E^\alpha.\]
In particular, if $\H^f=\H$, then $\alpha\in\Aut(\E)$. On the other hand, if $\alpha\in\Aut(\E)$, then $\F_T(\H^f)=\E^\alpha=\E=\F_T(\H)$. Since the map $\hat{\Psi}$ from Background Theorem~\ref{bT:2} is injective, we see thus that $\alpha\in\Aut(\E)$ implies $\H^f=\H$. This shows $f\in N_{G}(\H)$ if and only if $c_f|_T=\alpha\in\Aut(\E)$. Hence, we proved that \eqref{E:NSE} holds.

\smallskip

By Lemma~\ref{L:NinNLtTofH}(b), $G$ is a model for $N_\F(\tT)$ and $N_\H(\tT)\subn G$ is a model for $N_\E(T_0)$. It follows from Lemma~\ref{L:MS}(a) that $(G,\delta(N_\F(\tT)),S)$ is a regular locality over $N_\F(\tT)$. Thus, applying \eqref{E:NSE} with $(G,N_\H(\tT),N_\E(T_0))$ in place of $(\L,\H,\E)$, we obtain $N_S(N_\H(\tT))=N_S(N_\E(T_0))$. Lemma~\ref{L:NinNLtTofH}(c) implies that $N_S(\H)=N_S(N_\H(\tT))$. Hence the assertion holds.
\end{proof}

\begin{lemma}\label{L:CLTinbNLH}
\begin{itemize}
 \item [(a)] $C_\L(T)\subseteq \bN_\L(\H)$.
 \item [(b)] $C_\L(\H)=C_{\bN_\L(\H)}(\H)\unlhd \bN_\L(\H)$.
\end{itemize}
\end{lemma}

\begin{proof}
\textbf{(a)} As $T$ is a $p$-subgroup of $N_G(\H)$, we may choose $S_0\in\Syl_p(N_G(\H))$ with $T\leq S_0$. Set $\F_0:=\F_{S_0}(\bN_\L(\H))$. By Theorem~\ref{T:bNLH}(a),(b), $\F_0$ is saturated and $\E\unlhd\F_0$. Let $f\in C_\L(T)$. By Lemma~\ref{L:FrattiniSplitting}, there exists $x\in E(\L)$ and $g\in G$ such that $(x,g)\in\D$, $f=xg$ and $S_{(x,g)}=S_f$. Notice that $E(\L)\subseteq \bN_\L(\H)$, so $c_x|_T\in\Aut_{\F_0}(T)\leq \Aut(\E)$. As $(c_x|_T)\circ (c_g|_T)=c_f|_T=\id\in \Aut(\E)$, it follows $c_g|_T\in\Aut(\E)$. Hence, Lemma~\ref{L:finNLtT} yields $g\in N_{G}(\H)$ and thus $f=xg\in E(\L)N_{G}(\H)=\bN_\L(\H)$.

\smallskip

\textbf{(b)} Using (a), we see that $C_\L(\H)\subseteq C_\L(T)\subseteq \bN_\L(\H)$ and thus $C_\L(\H)=C_{\bN_\L(\H)}(\H)$. Recall that $\bN_\L(\H)$ forms a regular locality by Theorem~\ref{T:bNLH}(a). Hence, part (b) follows from Background Theorem~\ref{bT:1} applied with $\bN_\L(\H)$ in place of $\L$.   
\end{proof}

\begin{proof}[\textbf{Proof of Theorem~\ref{T:MainbNLH}}]
Parts (a),(b),(c) follow from Theorem~\ref{T:bNLH}(a),(b),(c), and part (d) is stated in Theorem~\ref{T:bNLH}(e). Part (e) is proven in Lemma~\ref{L:X},  and part (f) in Lemma~\ref{L:CLTinbNLH}. 
\end{proof}

\subsection{The normalizer of $\E$}

We continue to assume Hypothesis~\ref{H:main} and assume in addition for the remainder of this section the following hypothesis.

\begin{hypo}\label{H:S0F0}
Let $S_0$ be a Sylow $p$-subgroup of $N_{G}(\H)$ with $T\leq S_0$. Set 
\[\F_0:=\F_{S_0}(\bN_\L(\H)).\]
\end{hypo}

Note here that $T$ is a $p$-subgroup of $N_G(\H)$ and thus there exists always a Sylow $p$-subgroup $S_0$ of $N_G(\H)$ with $T\leq S_0$.

\smallskip

We will use throughout that, by Theorem~\ref{T:bNLH}, $\F_0$ is well-defined and saturated, $(\bN_\L(\H),\delta(\F_0),S_0)$ is a regular locality over $\F_0$, $E(\L)=E(\bN_\L(\H))\unlhd \bN_\L(\H)$, $\tT=E(\L)\cap S_0$, $E(\F)=E(\F_0)\unlhd \F_0$, $\H\unlhd\bN_\L(\H)$, $T=S_0\cap \H$ and $\E\unlhd\F_0$. We use moreover throughout that
\[N_S(\E)=N_S(\H)\]
by Lemma~\ref{L:finNLtT}. Recall that
\[N_\F(\E)=\F_{N_S(\H)}(\bN_\L(\H))\]
by definition, so in particular, $N_\F(\E)$ is a fusion system over $N_S(\H)=N_S(\E)$.

\smallskip

In this subsection, we will prove Theorem~\ref{T:MainER},  the properties stated in Theorem~\ref{T:FusionNormalizerNew}(a),(b),(c),(e), and some related results concerning the connection between $N_\F(\E)$ and $\F_0$. Moreover, we give in Theorem~\ref{T:FusionNormalizerGenerate} another concrete description of $N_\F(\E)$. It somehow resembles the definiton of $\bN_\L(\H)$.

\smallskip

For the next lemma, the reader might want to recall the definition of $(\E R)_\F$ from Definition~\ref{D:ProductER}.

\begin{lemma}\label{L:EFR}
For every $R\leq N_S(\E)\cap S_0$, we have 
\[E(\F)R:=(E(\F)R)_\F=\F_{\tT R}(E(\L)R)=(E(\F)R)_{\F_0}\] 
and $(E(\L)R,\delta(E(\F)R),\tT R)$ is a regular locality over $E(\F)R$. In particular, $E(\F)R\subseteq N_\F(\E)$.
\end{lemma}

\begin{proof}
By Background Theorem~\ref{bT:2}, we have $E(\F)=\F_{S\cap E(\L)}(E(\L))=\F_{\tT}(E(\L))=\F_{S_0\cap E(\L)}(E(\L))$. Hence, by \cite[Theorem~D(a)]{Chermak/Henke} applied twice (once with $\L$ and once with $\bN_\L(\H)$ in place of $\L$), we have 
\[(E(\F)R)_\F=\F_{\tT R}(E(\L)R)=(E(\F)R)_{\F_0}\]
and $(E(\L)R,\delta(E(\F)R),\tT R)$ is a regular locality over $E(\F)R$. As $\tT R\leq N_S(\H)$ and $E(\L)R\subseteq \bN_\L(\H)$, it follows that $E(\F)R\subseteq N_\F(\E)$.
\end{proof}

Given $R\leq N_S(\E)$, we will write from now on usually $E(\F)R$ for $(E(\F)R)_\F$.

\begin{lemma}\label{L:FrattiniApply}
Let $X\leq N_S(\E)$ and $\phi\in\Hom_\F(X,S)$. Then there exist 
\[\psi\in\Hom_{E(\F)X}(X,X\tT)\subseteq \Hom_{N_\F(\E)}(X,X\tT)\]
and $\alpha\in\Hom_{N_\F(\tT)}(X\psi,S)$ such that $\phi=\psi\alpha$. If $X\leq S_0$, then $\psi$ is a morphism in $\F_0$.
\end{lemma}

\begin{proof}
The existence of $\psi\in\Hom_{E(\F)X}(X,\tT X)$ and $\alpha\in\Hom_{N_\F(\tT)}(X\psi,S)$ with $\phi=\psi\alpha$ follows from Lemma~\ref{L:FusionSystemFrattini} applied with $E(\F)$ in place of $\E$. The assertion follows now from Lemma~\ref{L:EFR}.
\end{proof}

\begin{lemma}\label{L:NFEbNLH}
Set $\B:=\F_{N_S(\H)}(N_G(\H))$. Then the following hold:
\begin{itemize}
 \item [(a)] We have 
\[N_\F(\E)=\<E(\F)N_S(\E), \B\>.\]
If $N_S(\E)\leq S_0$, then also 
\[N_\F(\E)=\F\cap \F_0=\F_0|_{N_S(\E)}.\]
 \item [(b)] If $P\leq N_S(\E)$ and $\phi\in\Hom_{N_\F(\E)}(P,N_S(\E))$, then there exist 
\[\psi\in\Hom_{E(\F)P}(P,\tT P)\mbox{ and }\alpha\in\Hom_{\B}(P\psi,N_S(\E))\]
such that $\phi=\psi\alpha$. 
 \item [(c)] $\E$ is a $N_\F(\E)$-invariant subsystem of $N_\F(\E)$. 
\end{itemize}
\end{lemma}

\begin{proof}
As $N_S(\E)=N_S(\H)$ is a $p$-subgroup of $N_G(\H)$, there exists always a Sylow $p$-subgroup of $N_G(\H)$ containing $N_S(\E)$. Hence, we may assume throughout the proof that $N_S(\E)\leq S_0$.

\smallskip

\textbf{(a,b)} As $N_S(\E)\leq S_0$ and $\E\unlhd\F_0$, we have $S_0\cap S=N_S(\E)=N_S(\H)$. Notice that $N_\F(\E)=\F_{N_S(\H)}(\bN_\L(\H))$ is contained in $\F=\F_S(\L)$ and in $\F_0=\F_{S_0}(\bN_\L(\H))$. Hence,
\begin{equation}\label{E:Chain1}
N_\F(\E)\subseteq \F\cap \F_0\subseteq \F_0|_{N_S(\E)}.
\end{equation}
It follows moreover from Lemma~\ref{L:EFR} that $E(\F)N_S(\E)=\F_{N_S(\E)}(E(\L)N_S(\E))\subseteq \F_0$. Clearly, $\B$ is contained in $\F_0$. Hence, 
\begin{equation}\label{E:Chain2}
\G:=\<E(\F)N_S(\E),\B\>\subseteq \F_0\cap \F\subseteq \F_0|_{N_S(\E)}.
\end{equation}
We note furthermore that Lemma~\ref{L:EFR} gives
\begin{equation}\label{E:EFP1}
(E(\F)P)_{\F_0}=(E(\F)P)_{\F}\subseteq (E(\F)N_S(\E))_\F\subseteq \G \mbox{ for all }P\leq N_S(\E)
\end{equation}
and
\begin{equation}\label{E:EFP2}
(E(\F)P)_{\F_0}\subseteq N_\F(\E) \mbox{ for all }P\leq N_S(\E).
\end{equation}

\smallskip

We show now that (b) holds and 
\begin{equation}\label{E:Chain3}
\F_0|_{N_S(\E)}\subseteq \G\cap N_\F(\E)
\end{equation}
For the proof let $P,Q\leq N_S(\E)$ and $\phi\in\Hom_{\F_0}(P,Q)$. Notice that any morphism in $\F_0|_{N_S(\E)}$ is of this form. As $E(\F)\unlhd\F_0$, it follows from Lemma~\ref{L:FusionSystemFrattini} that there exists $\psi\in\Hom_{(E(\F)P)_{\F_0}}(P,\tT P)$ and $\alpha\in \Hom_{N_{\F_0}(\tT)}(P\psi,N_S(\E))$ such that $\phi=\psi\alpha$. By Lemma~\ref{L:BasisNLtT} applied with $\bN_\L(\H)$ in place of $\L$, we have $\tT\in\delta(\F_0)$. Hence, by \cite[Lemma~3.10(b)]{Grazian/Henke} and Lemma~\ref{T:bNLH}(e),
\[N_{\F_0}(\tT)=\F_{S_0}(N_{\bN_\L(\H)}(\tT))=\F_{S_0}(N_{G}(\H)).\]
As $P\psi$ and $P\psi\alpha=P\phi$ are contained in $N_S(\E)=N_S(\H)$, it follows that $\alpha$ is a morphism in $\B=\F_{N_S(\H)}(N_G(\H))$. Since $N_\F(\E)\subseteq \F_0|_{N_S(\E)}$ by \eqref{E:Chain1} and as $\phi$ was an arbitrary morphism in $\F_0|_{N_S(\E)}$, this proves (b) and
\[\F_0|_{N_S(\E)}\subseteq\<(E(\F)P)_{\F_0},\B\>.\]
Observe now that $\B\subseteq \G\cap N_\F(\E)$ and that $(E(\F)P)_{\F_0}\subseteq \G\cap N_\F(\E)$ by \eqref{E:EFP1} and \eqref{E:EFP2}. Hence, \eqref{E:Chain3} holds. 

\smallskip

Notice now that (a) follows from \eqref{E:Chain1},  \eqref{E:Chain2} and \eqref{E:Chain2}.

\smallskip

\textbf{(c)} Using (a) one sees that $\E\subseteq \F\cap \F_0=N_\F(\E)$. As $\E$ is normal in $\F_0$, $T$ is strongly $N_\F(\E)$-closed. Moreover, $\E^\alpha=\E$ for every $\alpha\in \Aut_{\F_0}(T)$ and thus also for every $\alpha\in\Aut_{N_\F(\E)}(T)$. By Lemma~\ref{L:Finvariant}(b), it remains to show that every $N_\F(\E)$-morphism between subgroups of $T$ is in $\<\E,\Aut_{N_\F(\E)}(T)\>_T$. Fix $f\in \bN_\L(\H)$. We need to show that $c_f|_{T\cap S_f}\in \<\E,\Aut_{N_\F(\E)}(T)\>_T$.

\smallskip

By Lemma~\ref{L:FrattiniSplitting} applied with $(\bN_\L(\H),\delta(\F_0), S_0)$ and $\H$ in place of $(\L,\Delta,S)$ and $\N$, there exist $h\in\H$ and $g\in N_{\bN_\L(\H)}(T)$ such that $(h,g)\in\D$, $f=hg$ and $(S_0)_{(h,g)}=(S_0)_f$. As $T\leq S_0$ and $T$ is strongly closed in $\F_0$, we have in particular $T\cap S_{(h,g)}=T\cap S_f$. Hence, $c_f|_{T\cap S_f}$ is the composition of $c_h|_{T\cap S_f}$ and $c_g|_{(T\cap S_f)^h}$. Here $c_h|_{T\cap S_f}$ is a morphism in $\E$, and $c_g|_{(T\cap S_f)^h}$ extends to $c_g|_T$, which is an element of $\Aut_{N_\F(\E)}(T)$. This proves (c).
\end{proof}

As mentioned before, if $\F$ is constrained, then $\F$ has a model, i.e. a group of characteristic $p$ realizing $\F$ (cf. \cite[Theorem~III.5.10]{Aschbacher/Kessar/Oliver:2011}). Moreover, if $\F$ is constrained and $\hat{G}$ is a model for $\F$, then it is shown in \cite[Lemma~2.9]{Chermak/Henke} that there exists a unique subnormal subgroup $\hat{H}$ of $\hat{G}$ with $T=\hat{H}\cap S$ and $\E=\F_T(\hat{H})$. This is used to formulate the following theorem, which gives a concrete description of $N_\F(\E)$ that somehow resembles the definition of $\bN_\L(\H)$. %It gives a characterization of $N_\F(\E)$, which one could take as an alternative definition of this subsystem.

\begin{theorem}\label{T:FusionNormalizerGenerate}~
Fix $T_0,\tT\leq S$ such that
\[E(\E)\mbox{ is a fusion system over $T_0$ and $E(\F)$ is a fusion system over $\tT$.}\]
Then $N_\F(\tT)$ is constrained and $N_\E(T_0)$ is subnormal in $N_\F(\tT)$. If we choose a model $\tG$ for $N_\F(\tT)$ and $\tH\subn \tG$ with $\tH\cap S=T$ and $\F_T(\tH)=N_\E(T_0)$, then $N_S(\tH)=N_S(\E)$ and
\[N_\F(\E)=\<\,E(\F)N_S(\E)\, ,\,\F_{N_S(\tH)}(N_{\tG}(\tH))\,\>_{N_S(\E)}.\]
\end{theorem}

\begin{proof}
We have introduced $T_0$ and $\tT$ in Hypothesis~\ref{H:main}, but Background Theorem~\ref{bT:2} gives that $E(\E)$ is a fusion system over $T_0$ and $E(\F)$ is a fusion system over $\tT$. By Lemma~\ref{L:NinNLtTofH}(b), $N_\F(\tT)$ is constrained, $G$ is a model for $N_\F(\tT)$, and $N_\H(\tT)\subn G$ is a model for $N_\E(T_0)\subn N_\F(\tT)$. 
%In particular, $N_\F(\tT)$ is constrained.

\smallskip

If $\tG$ is now an arbitrary model for $N_\F(\tT)$ and $\tH\subn \tG$ with $\F_{S\cap \tH}(\tH)=N_\E(T_0)$, then it follows from Lemma~\ref{L:NormalizerConstrained} that $\F_{N_S(\tH)}(N_{\tG}(\tH))=\F_{N_S(N_\H(\tT))}(N_G(N_\H(\tT)))$. As $N_G(\H)=N_G(N_\H(\tT))$ by Lemma~\ref{L:NinNLtTofH}(c), we can conclude for such $\tG$ and $\tH$ that $N_S(\tH)=N_S(N_\H(\tT))=N_S(\H)=N_S(\E)$ and
\[\F_{N_S(\tH)}(N_{\tG}(\tH))=\F_{N_S(\H)}(N_G(\H)).\]
Hence, the last part of the assertion follows from the first statement in Lemma~\ref{L:NFEbNLH}(a).
\end{proof}

The reader might want to recall the definition of $N_\F(T,\E)$ from Definition~\ref{D:6}.

\begin{lemma}\label{L:inAutEinNFE}
The following hold:
\begin{itemize}
 \item [(a)] $N_{N_\F(\E)}(T)=N_\F(T,\E)$.
 \item [(b)] $C_{N_\F(\E)}(T)=C_\F(T)$.
 \item [(c)] If $N_S(\E)\leq S_0$, then $N_\F(T,\E)=N_{\F_0}(T)|_{N_S(\E)}$ and $C_\F(T)=C_{\F_0}(T)|_{C_S(T)}$.
\end{itemize}
\end{lemma}

\begin{proof}
\textbf{(a)} Since $\E$ is $N_\F(\E)$-invariant by Lemma~\ref{L:NFEbNLH}(c), it follows that $N_{N_\F(\E)}(T)\subseteq N_\F(T,\E)$. It remains thus to prove the opposite inclusion. For that let $X,Y\leq N_S(\E)$ and $\phi\in\Hom_\F(X,Y)$ with $T\leq X\cap Y$ and $\phi|_T\in\Aut(\E)$. By definition of $N_\F(T,\E)$, it is sufficient to prove that $\phi$ is a morphism in $N_\F(\E)$ and thus in $N_{N_\F(\E)}(T)$.

\smallskip

By Lemma~\ref{L:FrattiniApply}, there exists $\psi\in\Hom_{N_\F(\E)}(X,\tT X)$ and $\alpha\in\Hom_{N_\F(\tT)}(X\psi,S)$ such that $\phi=\psi\alpha$. As $\E$ is $N_\F(\E)$-invariant by Lemma~\ref{L:NFEbNLH}(c), we have $T\psi=T$ and $\E^\psi=\E$. Hence, $T=T\psi\leq X\psi$, $T\alpha=T\psi\alpha=T\phi=T$ and $\E^\alpha=(\E^\psi)^\alpha=\E^\phi=\E$. By Lemma~\ref{L:NinNLtTofH}(b), $G$ is a model for $N_\F(\tT)$, i.e. there exists $g\in G$ with $\alpha=c_g|_{X\psi}$. As $c_g|_T=\alpha|_T\in\Aut(\E)$, it follows from Lemma~\ref{L:finNLtT} that $g\in N_G(\H)$. As $X\psi\leq \tT X\leq N_S(\E)=N_S(\H)$, it follows $X\psi\alpha=(X\psi)^g\leq S\cap N_G(\H)=N_S(\H)$. So $\alpha$ is a morphism in $N_\F(\E)=\F_{N_S(\H)}(\bN_\L(\H))$. Hence, $\phi=\psi\alpha$ is in $N_\F(\E)$. 

\smallskip

\textbf{(b)} Clearly $C_{N_\F(\E)}(T)\subseteq C_\F(T)$. If $P,Q\leq C_S(T)$ and $\phi\in\Hom_{C_\F(T)}(P,Q)$, then $\phi$ extends to $\hat{\phi}\in \Hom_\F(PT,QT)$ with $\hat{\phi}|_T=\id_T$. As $\id_T\in\Aut(\E)$, it follows that $\hat{\phi}$ is a morphism in $N_\F(T,\E)$. Thus, part (a) yields that $\hat{\phi}$ is a morphism in $N_\F(\E)$. Hence, $\phi=\hat{\phi}|_P$ is a morphism in $C_{N_\F(\E)}(T)$.

\smallskip

\textbf{(c)} Suppose $N_S(\E)\leq S_0$.  Lemma~\ref{L:NFEbNLH}(a) gives then that $N_\F(\E)=\F_0|_{N_S(\E)}$. In particular, since $TC_S(T)\leq N_S(\E)$, it follows that
\[N_{N_\F(\E)}(T)=N_{\F_0}(T)|_{N_S(\E)}\mbox{ and }C_{N_\F(\E)}(T)=C_{\F_0}(T)|_{C_S(T)}.\]
The assertion follows thus from parts (a) and (b).
\end{proof}

Our next goal is to show that Theorem~\ref{T:MainER} holds and that every saturated subsystem $\mD$ of $\F$ with $\E\unlhd\mD$ is contained in $N_\F(\E)$. We first state some technical lemmas which are used in the proof.

\begin{lemma}\label{L:QcapT0phi}
Let $T_0^\perp\leq Q\leq N_S(\E)$ and $\phi\in\Hom_\F(Q,S)$ such that $T_0^\perp\phi=T_0^\perp$. Then 
\[(Q\cap T_0)\phi\leq T_0.\] 
\end{lemma}

\begin{proof}
We may choose $S_0$ such that $N_S(\E)\leq S_0$. By Lemma~\ref{L:FrattiniApply}, there exist then $\psi\in\Hom_{\F\cap\F_0}(Q,\tT Q)$ and $\alpha\in\Hom_{N_\F(\tT)}(Q\tT,S)$ with $\phi=\psi\alpha$. Theorem~\ref{T:bNLH}(d) implies that $T_0$ and $T_0^\perp$ are strongly closed in $\F_0$. Thus, 
\[(Q\cap T_0)\psi\leq T_0\]
and $T_0^\perp\psi=T_0^\perp$. This yields $T_0^\perp\alpha=T_0^\perp\psi\alpha=T_0^\perp\phi=T_0^\perp$. By Lemma~\ref{L:NinNLtTofH}(b), $G$ realizes $N_\F(\tT)$. Hence, $\alpha=c_g|_{Q\tT}$ for some $g\in G$, which is then an element of $N_G(T_0^\perp)$ as $T_0^\perp\alpha=T_0^\perp$. Using Lemma~\ref{L:NinNLtTofH}(e) we see now that $g\in N_{G}(T_0^\perp)=N_{G}(T_0)$ and thus $T_0\alpha=T_0$. It follows that $(Q\cap T_0)\phi=(Q\cap T_0)\psi\alpha\leq T_0\alpha=T_0$ as required.
\end{proof}

\begin{lemma}\label{L:NormalizeT0perp}
 Let $Q\in \E^{cr}$ and $\phi\in\Hom_\F(Q\tT,S)$ with $Q\phi=Q$. Then $T_0\phi=T_0$ and $T_0^\perp\phi=T_0^\perp$.
\end{lemma}

\begin{proof}
If $\K\in\Comp(\H)$, then $\F_{S\cap \K}(\K)\subn \E$ by Background Theorem~\ref{bT:2}. Hence, by Lemma~\ref{L:CentricRadicalIntersection}, $Q\cap \K\in \F_{S\cap \K}(\K)^{cr}$. In particular, $Q\cap \K\not\leq Z(\K)$ as $Z(\K)\neq \K\cap S$ by \cite[Lemma~7.10]{Henke:Regular}. On the other hand, if $\K\in \Comp(\L)\backslash\Comp(\H)$, then $\H\subseteq C_\L(\K)$ by \cite[Lemma~3.5, Lemma~11.17]{Henke:Regular}. Hence, as $Q\leq T\subseteq\H$, we have in this case $Q\cap \K\subseteq \H\cap\K\subseteq C_\K(\K)=Z(\K)$. This shows 
\[\Comp(\H)=\{\K\in\Comp(\L)\colon \K\cap Q\not\leq Z(\K)\}.\] 

\smallskip

As $\tT$ is strongly $\F$-closed, $\phi$ is a morphism in $N_\F(\tT)=\F_S(G)$. Hence, there exists $f\in G$ with $\phi=c_f|_{Q\tT}$. Then  conjugation by $f$ induces by Lemma~\ref{L:BasisNLtT} an automorphism of $\L$ and thus permutes by \cite[Lemma~11.12]{Henke:Regular} the components of $\L$. Moreover, for every $\K\in\Comp(\L)$, $c_f$ induces an isomorphism from $\K$ to $\K^f$ so that $Z(\K)^f=Z(\K^f)$. As $Q^f=Q\phi=Q$, we have also $(\K\cap Q)^f=\K^f\cap Q$. So for every $\K\in\Comp(\L)$, we have $\K\cap Q\not\leq Z(\K)$ if and only if $\K^f\cap Q\not\leq Z(\K^f)$. Therefore, $f$ permutes the elements of $\Comp(\H)$ via conjugation and thus also the elements of $\Comp(\L)\backslash \Comp(\H)$. This implies $E(\H)^f=E(\H)$ and $\M^f=\M$. As $\phi=c_f|_{Q\tT}$, the assertion follows.
\end{proof}

\begin{lemma}\label{L:GCollect}
 Let $R\leq S_0$ and set 
\[\G:=(\E R)_{\F_0}\]
Then the following hold:
\begin{itemize}
 \item [(a)] $\G$ is saturated, $\E\unlhd \G$, $O^p(\G)=O^p(\E)$ and $(\H R,\delta(\G),TR)$ is a regular locality over $\G$ with $E(\H R)=E(\H)$. In particular, $E(\G)=E(\E)$. Moreover, if $R\leq N_S(\E)$, then $\G=\F_{T R}(\H R)\subseteq \F$.
 \item [(b)] If $P\leq TR$ with $P\cap T\in\delta(\E)$, then $P\in\delta(\G)$. 
 \item [(c)] Suppose $P\in\delta(\G)$. Then $N_{\H R}(P)$ and $N_\H(P)$ are subgroups of $\L$. In particular,  
\[\Aut_\H(P):=\{c_h\colon h\in N_\H(P)\}\]
is a subgroup of $\Aut(P)$. Moreover, $O^p(\Aut_\G(P))=O^p(\Aut_\H(P))$.
 \item [(d)] If $P\leq TR\cap N_S(\E)$ is such that $P\cap T\in\E^{cr}$, then $P\cap T\in\delta(\E)$ and thus $P\in\delta(\G)$. Moreover, \[O^p(\Aut_\G(P))=O^p(\Aut_\H(P))=\Ac_{\F_0,\E}(P)=\Ac_{N_\F(\E),\E}(P).\]
\end{itemize}
\end{lemma}

\begin{proof}
\textbf{(a)} We use as before that $\F_0$ is well-defined and saturated, $(\bN_\L(\H),\delta(\F_0),S_0)$ is a regular locality over $\F_0$ with $T=S_0\cap \H$, $\H\unlhd\bN_\L(\H)$ and $\E=\F_T(\H)\unlhd\F_0$. By Proposition~\ref{P:ERnormal} applied with $\F_0$ in place of $\F$, $\G$ is saturated, $\E\unlhd \G$, $O^p(\G)=O^p(\E)$ and our notion of a product of a normal subsystem with a $p$-subgroup coincides with the one used in \cite{Chermak/Henke}. Thus, it follows from \cite[Theorem~D(a)]{Chermak/Henke} that $\G$ is saturated and $(\H R,\delta(\G),TR)$ is a regular locality over $\G$. In particular, if $R\leq N_S(\E)$, then 
\[\G=\F_{T R}(\H R)\subseteq \F_S(\L)=\F.\]
For an arbitrary subgroup $R\leq S_0$, $\H$ is a partial normal subgroup of $\H R$ of $p$-power index. Thus, by \cite[Lemma~11.14]{Henke:Regular}, $E(\H R)=E(\H)$. It follows now from Background Theorem~\ref{bT:2} applied twice (once with $(\H,T,\E)$ and once with $(\H R,TR,\G)$ in place of $(\L,S,\F)$) that $E(\E)=\F_{T_0}(E(\H))=E(\G)$.  This proves (a).

\smallskip

\textbf{(b)} As $E(\E)=E(\G)$, part (b) follows from Lemma~\ref{L:EE=EF} applied with $\G$ in place of $\F$.

\smallskip

\textbf{(c)} Let $P\in\delta(\G)$. As $(\H R,\delta(\G),TR)$ is a regular locality over $\G$, it follows that $N_{\H R}(P)$ is a subgroup of $\H R$ and thus of $\L$. Moreover, 
\[\Aut_\G(P)=\{c_g\colon g\in N_{\H R}(P)\}.\]
Since $\H\unlhd \H R$, $N_\H(P)=\H\cap N_{\H R}(P)$ is a normal subgroup of $N_{\H R}(P)$. In particular, $\Aut_\H(P)$ is a subgroup of $\Aut(P)$. By \cite[Lemma~6.1]{Henke:2020} (applied with $(\H P,\delta(\G),TP)$ in place of $(\L,\Delta,S)$), we have $O^p(N_{\H P}(P))=O^p(N_\H(P))$. Observe that every $p^\prime$-element of $\Aut_\H(P)$ can be realized by conjugation with a $p^\prime$-element of $N_\H(P)$, and similarly every $p^\prime$-automorphism of $\Aut_\G(P)$ can be realized by conjugation with a $p^\prime$-element of $N_{\H P}(P)$. This implies
\[O^p(\Aut_\H(P))=\{c_h\colon h\in O^p(N_\H(P))\}=\{c_h\colon h\in O^p(N_{\H R}(P))\}=O^p(\Aut_\G(P)).\]
Hence (c) holds.

\smallskip

\textbf{(d)} By \cite[Lemma~10.4]{Henke:Regular}, $\E^{cr}\subseteq \delta(\E)$. Moreover, if $P\leq TR$ with $P\cap T=\E^c$, then \cite[Lemma~4.7]{Henke:2013} gives $O^p(\Aut_\G(P))=\Ac_{\F_0,\E}(P)$. Hence, supposing $P$ is a subgroup of $TR$ with $P\cap T\in\E^{cr}$, it follows from (b) and (c) that $O^p(\Aut_\H(P))=O^p(\Aut_\G(P))=\Ac_{\F_0,\E}(P)$. If in addition $P\leq N_S(\E)$, then $\Ac_{\F_0,\E}(P)\leq \Aut_\H(P)\leq \Aut_{N_\F(\E)}(P)$ and so $\Ac_{\F_0,\E}(P)=\Ac_{N_\F(\E),\E}(P)$. This proves (d).
\end{proof}

We use from now on without further reference that, for all $P\leq N_S(\E)$ with $P\cap T\in\E^{cr}$, the normalizer $N_\H(P)$ is a subgroup of $\L$ and  
\[\Aut_\H(P):=\{c_h\colon h\in N_\H(P)\}\] 
is a subgroup of $\Aut(P)$. Notice that this is true by Lemma~\ref{L:GCollect}(b),(c) (applied e.g. for a Sylow subgroup $S_0$ of $N_G(\H)$ containing $N_S(\H)=N_S(\E)$, and with $R=N_S(\E)$).

\begin{lemma}\label{L:AutHPsubnormal}
Let $P\leq N_S(\E)$ be such that $P\cap T\in\E^{cr}$. Then $\Aut_\H(P)\subn \Aut_{N_\F(T_0^\perp)}(P)$. 
\end{lemma}

\begin{proof}
By Lemma~\ref{L:BasicT0perp}(a), $\H\subseteq C_\L(\M)\subseteq C_\L(T_0^\perp)$. In particular, 
\[\Aut_\H(P)\leq \Aut_{N_\F(T_0^\perp)}(P).\] 
As $\F_{T_0}(E(\H))=E(\E)\unlhd \E$ by Background Theorem~\ref{bT:2}, it follows from Lemma~\ref{L:CentricRadicalIntersection} and \cite[Lemma~10.4]{Henke:Regular} that 
\[P_0:=P\cap T_0\in E(\E)^{cr}\subseteq \delta(E(\E)).\]
By Lemma~\ref{L:BasicT0perp}(a), $E(\H)$ and $\M$ are partial normal subgroups of $E(\L)$ with $E(\L)=E(\H)\M$ and $E(\H)\subseteq \H\subseteq C_\L(\M)$. The latter property implies by \cite[Lemma~3.5]{Henke:Regular} that $E(\H)$ commutes with $\M$ in the sense of \cite[Definition~2]{Henke:Regular}. Thus, it follows from \cite[Lemma~10.15(d)]{Henke:Regular} that 
\[X:=P_0T_0^\perp\in\delta(\F_{S\cap E(\L)}(E(\L))).\]
By Background Theorem~\ref{bT:2}, we have $\F_{S\cap E(\L)}(E(\L))=E(\F)$ and by \cite[7.22]{Chermak/Henke}, $\delta(E(\F))\subseteq\delta(\F)$. \footnote{Alternatively, a direct argument involving Lemma~10.2, Lemma~10.6, Lemma~10.15(d) and Lemma~11.5 in \cite{Henke:Regular} could be given to show $\delta(\F_{S\cap E(\L)}(E(\L)))\subseteq\delta(\F)$.} We can conclude that \[X\in\Delta=\delta(\F).\]
In particular, $N_\L(X)$ is a group. 

\smallskip

It follows from  Lemma~\ref{L:NinNLtTofH}(e) that $P\leq N_S(\E)=N_S(\H)\leq N_S(T_0^\perp)=N_S(T_0)$. Thus, $P\leq N_S(P_0)$ and 
\[P\leq N_S(X).\]
Notice moreover that 
\[N_\H(P)\leq N_\H(X)\]
as $P_0=P\cap E(\H)$ is normalized by $N_\H(P)$ and $\H\subseteq C_\L(T_0^\perp)$.

\smallskip

Let now $\H=\H_0\unlhd\H_1\unlhd\cdots\unlhd\H_n=\L$ be a subnormal series of $\H$ in $\L$. Then 
\[N_\H(P)=N_{N_\H(X)}(P)=N_{N_{\H_0}(X)}(P)\unlhd N_{N_{\H_1}(X)}(P)\unlhd\cdots \unlhd N_{N_{\H_n}(X)}(P)=N_{N_\L(X)}(P)\]
is a subnormal series of $N_\H(P)$ in $N_{N_\L(X)}(P)$. This leads to a subnormal series of $\Aut_\H(P)$ in 
\[\Aut_{N_\L(X)}(P)=\{c_g|_P\colon g\in N_{N_\L(X)}(P)\}.\]
Thus, it remains only to argue that $\Aut_{N_\F(T_0^\perp)}(P)\leq \Aut_{N_\L(X)}(P)$. By \cite[Lemma~3.10(b)]{Grazian/Henke}, $N_\F(X)=\F_{N_S(X)}(N_\L(X))$ and thus 
\[\Aut_{N_\L(X)}(P)=\Aut_{N_\F(X)}(P).\]
Let $\phi\in\Aut_{N_\F(T_0^\perp)}(P)$. Then $\phi$ extends to $\hat{\phi}\in\Hom_\F(PT_0^\perp,S)$ with $T_0^\perp\hat{\phi}=T_0^\perp$. By Lemma~\ref{L:QcapT0phi}, we have $P_0\hat{\phi}\leq T_0$ and thus $P_0\hat{\phi}\leq P\cap T_0=P_0$. This implies $X\hat{\phi}=X$ and proves thus $\Aut_{N_\F(T_0^\perp)}(P)\leq \Aut_{N_\F(X)}(P)=\Aut_{N_\L(X)}(P)$ as required.
\end{proof}

\begin{lemma}\label{L:ConjAcbyNFE}
Let $P\leq N_S(\E)$ and $\chi\in\Hom_{N_\F(\E)}(P,N_S(\E))$. Then 
\[\chi^{-1}\Ac_{\F,\E}(P)\chi=\Ac_{\F,\E}(P\chi)\mbox{ and }\chi^{-1}\Ac_{N_\F(\E),\E}(P)\chi=\Ac_{N_\F(\E),\E}(P\chi).\]
In particular, $\Ac_{\F,\E}(P)=\Ac_{N_\F(\E),\E}(P)$ if and only if $\Ac_{\F,\E}(P\chi)=\Ac_{N_\F(\E),\E}(P\chi)$. 
\end{lemma}

\begin{proof}
Recall that $\E$ is $N_\F(\E)$-invariant by Lemma~\ref{L:NFEbNLH}(c). In particular, $\E$ is weakly $N_\F(\E)$-invariant by Lemma~\ref{L:Finvariant}(a). We can conclude that $(P\cap T)\chi=(P\chi)\cap T$ and 
\[\chi^{-1}\Aut_\E(P\cap T)\chi=\Aut_\E((P\chi)\cap T).\]
Since the map $\Aut_\F(P)\rightarrow \Aut_\F(P\chi),\delta\mapsto \chi^{-1}\delta\chi$ is an isomorphism which restricts to an isomorphism $\Aut_{N_\F(\E)}(P)\rightarrow \Aut_{N_\F(\E)}(P\chi)$, the assertion is now easy to check.  
\end{proof}

\begin{lemma}\label{L:Findchi}
 Let $P\leq N_S(\E)\cap S_0$ be such that $P\cap T\in\E^{cr}$. Then there exists $\chi\in\Hom_{N_\F(\E)}(P,N_S(\E))$ such that $P\chi\leq S_0$ and the following hold:  
\begin{itemize}
 \item[(i)] $O_p(N_{\Aut_\F(P\chi)}((P\chi)\cap T))\leq \Aut_{N_\F(T_0^\perp)}(P\chi)$; 
 \item[(ii)] $(P\chi)\cap T$ is fully normalized in $(\E(P\chi))_{\F_0}$ and an element of $\E^{cr}$.  
\end{itemize}
\end{lemma}

\begin{proof}
To ease notation set $Q:=P\cap T$. We show first
\begin{eqnarray}\label{E:Getphi}
\mbox{There exists $\phi\in\Hom_\F(P,S)$ such that $N_{\Aut_S(P\phi)}(Q\phi)\in\Syl_p(N_{\Aut_\F(P\phi)}(Q\phi))$}.
\end{eqnarray}
For the proof fix $\beta\in\Hom_\F(P,S)$ such that $P\beta$ is fully $\F$-normalized. Then by the Sylow axiom \cite[Proposition~I.2.5]{Aschbacher/Kessar/Oliver:2011}, we have $\Aut_S(P\beta)\in\Syl_p(\Aut_\F(P\beta))$. As $H:=N_{\Aut_\F(P\beta)}(Q\beta)$ is a subgroup of $\Aut_\F(P\beta)$, there exists $\gamma\in\Aut_\F(P\beta)$ such that $H^\gamma\cap \Aut_S(P\beta)\in\Syl_p(H^\gamma)$. Note that $P\beta\gamma=P\beta$ and $H^\gamma=N_{\Aut_\F(P\beta)}(Q\beta\gamma)$. Hence \eqref{E:Getphi} holds for $\phi:=\beta\gamma$. 

\smallskip

Let now $\phi$ be as in \eqref{E:Getphi}. Then in particular, 
\[O_p(N_{\Aut_\F(P\phi)}(Q\phi))\leq \Aut_S(P\phi)\leq \Aut_{N_\F(\tT)}(P\phi).\]
By Lemma~\ref{L:FrattiniApply}, there exist $\psi\in\Hom_{N_\F(\E)}(P,\tT P)\cap \Hom_{\F_0}(P,\tT P)$ and $\alpha\in\Hom_{N_\F(\tT)}(P\psi,S)$ such that $\phi=\psi\alpha$. Notice that we have an isomorphism from $N_{\Aut_\F(P\phi)}(Q\phi)$ to $N_{\Aut_\F(P\psi)}(Q\psi)$ defined by $\delta\mapsto \alpha\delta\alpha^{-1}$. Hence,
\begin{equation}\label{E:OpinNtT}
O_p(N_{\Aut_\F(P\psi)}(Q\psi))=\alpha\,O_p(N_{\Aut_\F(P\phi)}(Q\phi))\alpha^{-1}\leq \Aut_{N_\F(\tT)}(P\psi).
\end{equation}
Since $\E\unlhd\F_0$ and $\psi$ is a morphism in $\F_0$, it follows from $Q\in\E^{cr}$ and \cite[Lemma~2.6]{Chermak/Henke} that $Q\psi\in\E^{cr}$. Consider now $\xi\in O_p(N_{\Aut_\F(P\psi)}(Q\psi))$ and notice that  \eqref{E:OpinNtT} allows us to extend $\xi$ to $\hat{\xi}\in \Aut_\F((P\psi)\tT)$ with $\tT\hat{\xi}=\tT$. Applying Lemma~\ref{L:NormalizeT0perp} with $Q\psi$ in place of $Q$ and $\hat{\xi}|_{(Q\chi)\tT}$ in place of $\phi$, we can conclude that $T_0^\perp\hat{\xi}=T_0^\perp$ and thus $\xi$ is a morphism in $N_\F(T_0^\perp)$. Hence, we have shown that
\begin{equation}\label{E:OpinNT0perp}
O_p(N_{\Aut_\F(P\psi)}(Q\psi))\leq \Aut_{N_\F(T_0^\perp)}(P\psi).
\end{equation}

\smallskip

Observe now that $P\psi\leq N_S(\E)=N_S(\H)\leq N_S(T_0^\perp)$ by Lemma~\ref{L:finNLtT} and Lemma~\ref{L:NinNLtTofH}(e). Moreover, by Lemma~\ref{L:BasicT0perp}(a), $\H\subseteq C_\L(\M)\subseteq C_\L(T_0^\perp)$. As $N_\L(T_0^\perp)$ is by \cite[Lemma~5.5]{Henke:2015} a partial subgroup of $\L$, it follows $\H (P\psi)\subseteq N_\L(T_0^\perp)$. Using Lemma~\ref{L:GCollect}(a), we see now that 
\[\G:=(\E (P\psi))_{\F_0}=\F_{T(P\psi)}(\H(P\psi))\subseteq N_\F(T_0^\perp)\]
is saturated. Thus, by \cite[Lemma~I.2.6(c)]{Aschbacher/Kessar/Oliver:2011}, there exists $\rho\in\Hom_\G(N_{T(P\psi)}(Q\psi),T(P\psi))$ such that $Q\psi\rho$ is fully $\G$-normalized. Notice that $Q\psi\unlhd P\psi$ and thus $P\psi\rho$ is defined. Moreover, by definition of $\G$, we have $T(P\psi)=T(P\psi\rho)$ and thus $\G=(\E (P\psi\rho))_{\F_0}$. Since $\E\unlhd \F_0$ and $\psi$ and $\rho$ are morphisms in $\F_0$, we have  that $(P\psi\rho)\cap T=Q\psi\rho$ is fully $\G$-normalized and moreover an element of $\E^{cr}$ by \cite[Lemma~2.6]{Chermak/Henke}. Since $\rho$ is a morphism in $\G$ and $\G\subseteq N_\F(T_0^\perp)$, it follows from \eqref{E:OpinNT0perp} that 
\[O_p(N_{\Aut_\F(P\psi\rho)}((P\psi\rho)\cap T))=\rho^{-1}O_p(N_{\Aut_\F(P\psi)}(Q\psi))\rho\leq \Aut_{N_\F(T_0^\perp)}(P\psi\rho).\]
Hence, the assertion follows for $\chi=\psi\rho$.  
\end{proof}

\begin{lemma}\label{L:CrucialLemma}
Let $P\leq N_S(\E)\cap S_0$ be such that $P\cap T\in\E^{cr}$. Then 
\[\Ac_{\F,\E}(P)=\Ac_{N_\F(\E),\E}(P)=\Ac_{\F_0,\E}(P)=O^p(\Aut_\H(P)).\]
In particular, for every $R\leq N_S(\E)\cap S_0$, we have $(\E R)_\F=(\E R)_{\F_0}\subseteq N_\F(\E)$.
\end{lemma}

\begin{proof}
By definition of $(\E R)_\F$ and $(\E R)_{\F_0}$ it is sufficient to prove the first part of the assertion.
Set
\[\G:=(\E P)_{\F_0}.\]
By Lemma~\ref{L:GCollect}, $\G$ is saturated, $(\H P,\delta(\G),TP)$ is a regular locality over $\G$, $P\in\delta(\G)$ and
\begin{equation}\label{E:AcEqualities}
O^p(\Aut_\G(P))=\Ac_{N_\F(\E),\E}(P)=\Ac_{\F_0,\E}(P)=O^p(\Aut_\H(P)).
\end{equation}
Thus, it remains to show $\Ac_{N_\F(\E),\E}(P)=\Ac_{\F,\E}(P)$. Set
\[A:=N_{\Aut_\F(P)}(P\cap T).\]
Replacing $P$ by a suitable $N_\F(\E)$-conjugate, Lemma~\ref{L:ConjAcbyNFE} and Lemma~\ref{L:Findchi} allow us to assume that 
\begin{equation}\label{E:Assumption}
P\cap T\mbox{ is fully $\G$-normalized and }O_p(A)\leq \Aut_{N_\F(T_0^\perp)}(P). 
\end{equation}
Clearly, $\Ac_{N_\F(\E),\E}(P)\leq \Ac_{\F,\E}(P)$, so it remains to show the converse inclusion. For the proof let $\phi\in\Ac_{\F,\E}(P)$ be a $p^\prime$-automorphism. By \eqref{E:AcEqualities}, it remains to show that $\phi\in O^p(\Aut_\H(P))$. 

\smallskip

Observe that $\phi_0:=\phi|_{P\cap T}\in\Aut_\E(P\cap T)$ is a $p^\prime$-automorphism. Set $\hat{P}:=PC_{TP}(P\cap T)$. We show first that
\begin{equation}\label{E:psihat}
\mbox{there exists $\beta\in\Aut_\G(\hat{P})$ such that $\beta|_{P\cap T}=\phi_0$, and}
\end{equation}
\vspace{-0.3 cm}
\begin{equation}\label{E:betaCommutator}
\mbox{for any such $\beta$, we have $[C_P(\phi),\beta]\leq C_{TP}(P\cap T)$}
\end{equation}
%where
%\[[C_P(\phi),\beta]:=\<x^{-1}(x\beta)\colon x\in C_P(\phi)\>.\]

\smallskip

For the proof of \eqref{E:psihat} notice that $\phi_0$ is a morphism in $\G$ and $C_{TP}(P\cap T)\leq N_{\phi_0}^\G$ (where we use the notation introduced in \cite[Definition~I.2.2]{Aschbacher/Kessar/Oliver:2011}). Using e.g. \cite[Remark~2.1]{Henke:2013}), observe furthermore that 
\[\phi_0^{-1}\Aut_P(P\cap T)\phi_0=\Aut_{P\phi}(P\cap T)=\Aut_P(P\cap T)\]
and hence $P\leq N_{\phi_0}^{\G}$. Therefore, $\hat{P}\leq N_{\phi_0}^\G$. Since $\G$ is saturated and $P\cap T$ is by \eqref{E:Assumption} fully $\G$-normalized, it follows that $\phi_0$ extends to $\beta\in\Hom_\G(\hat{P},TP)$ (cf. \cite[Proposition~I.2.5]{Aschbacher/Kessar/Oliver:2011}). 

\smallskip

We argue now that for any such $\beta$, we have $[C_P(\phi),\beta]:=\<x^{-1}(x\beta)\colon x\in C_P(\phi)\>\leq C_{TP}(P\cap T)$ and $\hat{P}\beta=\hat{P}$, thus proving \eqref{E:psihat} and \eqref{E:betaCommutator}.
As $\phi$ is a $p^\prime$-automorphism and $[P,\phi]\leq P\cap T$, using a property of coprime action (see e.g. \cite[8.2.7(a)]{KS}), we see that $P=(P\cap T)C_P(\phi)$. As $\phi|_{P\cap T}=\phi_0=\beta|_{P\cap T}$ and $P\cap T\unlhd P$, it is a consequence of \cite[Lemma~2.4]{Henke:2013} that 
\[[C_P(\phi),\beta]\leq C_{TP}((P\cap T)\beta)=C_{TP}(P\cap T)\leq \hat{P}.\]
In particular, \eqref{E:betaCommutator} holds. Notice that $\beta$ normalizes $P\cap T$ and thus also $C_{TP}(P\cap T)$. As $\hat{P}=(P\cap T)C_P(\phi)C_{TP}(P\cap T)$, we can conclude that  $\hat{P}\beta=\hat{P}$. Hence, \eqref{E:psihat} holds.

\smallskip

We choose now $\beta$ as in \eqref{E:psihat}. Since $\phi_0$ is a $p^\prime$-element, replacing $\beta$ by a suitable power of itself, we may assume that $\beta$ is a $p^\prime$-element of $\Aut_\G(\hat{P})$. As $\G=(\E P)_{\F_0}$, it follows from \cite[Lemma~4.7]{Henke:2013} that $\beta\in O^p(\Aut_\G(\hat{P}))\leq \Ac_{\F_0,\E}(\hat{P})$.  In particular, 
\[[\hat{P},\beta]\leq \hat{P}\cap T\mbox{ and so }[C_P(\phi),\beta]\leq [\hat{P},\beta]\leq T.\]
Together with \eqref{E:betaCommutator}, we obtain thus
\[[C_P(\phi),\beta]\leq C_{TP}(P\cap T)\cap T=C_T(P\cap T)\leq P\cap T,\]
where the last inclusion uses $P\cap T\in\E^c$. As $P=(P\cap T)C_P(\phi)$ and $(P\cap T)\beta=P\cap T$, this yields $P\beta=P$ and thus, by \eqref{E:AcEqualities},
\[\psi:=\beta|_P\in O^p(\Aut_\G(P))=O^p(\Aut_\H(P))=\Ac_{N_\F(\E),\E}(P).\]

\smallskip

Notice that $\psi|_{P\cap T}=\phi_0=\phi|_{P\cap T}$ and thus $(\phi\psi^{-1})|_{P\cap T}=\id_{P\cap T}$. Moreover, $[P,\psi]\leq P\cap T$ and $[P,\phi]\leq P\cap T$, so $[P,\phi\psi^{-1}]\leq P\cap T$. By a property of coprime action (see e.g. \cite[8.2.7(b)]{KS}), $C_A(P/P\cap T)\cap C_A(P\cap T)$ is a normal $p$-subgroup of $A$ and thus contained in $O_p(A)$. Hence, $\phi\psi^{-1}\in O_p(A)$ and thus 
\[\phi\in X:=O_p(A)O^p(\Aut_\H(P)).\]
Notice that $\Aut_\H(P)\leq A$ normalizes $O_p(A)$ and therefore $X$ is a group with $O_p(A)\leq O_p(X)$. By Lemma~\ref{L:AutHPsubnormal}, $\Aut_\H(P)\subn \Aut_{N_\F(T_0^\perp)}(P)$ and thus $O^p(\Aut_\H(P))\subn \Aut_{N_\F(T_0^\perp)}(P)$. By assumption \eqref{E:Assumption}, we have $O_p(A)\leq \Aut_{N_\F(T_0^\perp)}(P)$. Hence, $O^p(\Aut_\H(P))$ is subnormal in $X$. By Lemma~\ref{L:OpGinNormalizer}, this implies that $O_p(A)\leq O_p(X)$ normalizes $O^p(\Aut_\H(P))$. In particular, $O^p(\Aut_\H(P))$ is normal in $X$. As $O^p(\Aut_\H(P))$ has $p$-power index in $X$, this implies $O^p(X)=O^p(\Aut_\H(P))$. As $\phi$ is a $p^\prime$-automorphism of $P$, it follows $\phi\in O^p(X)=O^p(\Aut_\H(P))$ as required.
\end{proof}

\begin{prop}\label{P:DinNFE}
Let $\mD$ be a saturated subsystem of $\F$ with $\E\unlhd\mD$. Then $\mD\subseteq N_\F(\E)$. 
\end{prop}

\begin{proof}
Let $R\leq S$ be such that $\mD$ is a fusion system over $R$. As $\E\unlhd\mD$, we have $R\leq N_S(\E)$. Lemma~\ref{L:FusionSystemFrattini} gives 
\[\mD=\<(\E R)_{\mD}, N_{\mD}(T)\>.\]
Notice that we may choose $S_0$ always such that $N_S(\E)\leq S_0$. For $P\leq TR$ with $P\cap T\in\E^{cr}$, it follows thus from  Lemma~\ref{L:CrucialLemma} that $\Ac_{\mD,\E}(P)\leq \Ac_{\F,\E}(P)\leq \Aut_{N_\F(\E)}(P)$. Hence, by definition of $(\E R)_{\mD}$, we have $(\E R)_{\mD}\subseteq N_\F(\E)$. 

\smallskip

A morphism in $N_{\mD}(T)$ is the restriction of a morphism $\phi\in\Hom_{\mD}(X,Y)$ where $X,Y\leq R$ with $T\leq X\cap Y$ and $T\phi=T$. As $\E\unlhd \mD$, for any such $\phi$, we have $\phi|_T\in\Aut(\E)$. Thus $\phi$ is a morphism in $N_\F(T,\E)\subseteq N_\F(\E)$ by Lemma~\ref{L:inAutEinNFE}(a). This shows $N_{\mD}(T)\subseteq N_\F(\E)$ and therefore $\mD\subseteq N_\F(\E)$.
\end{proof}

\begin{proof}[\textbf{Proof of Theorem~\ref{T:MainER}}]
Choose $S_0$ such that $N_S(\E)=N_S(\H)\leq S_0$. By Lemma~\ref{L:CrucialLemma}, we have $\E R:=(\E R)_\F=(\E R)_{\F_0}$. Hence, by Lemma~\ref{L:GCollect}(a), $\E R$ is saturated, $\E\unlhd \E R$, $O^p(\E R)=O^p(\E)$ and $(\H R,\delta(\E R),TR)$ is a regular locality over $\E R$. In particular, (b) holds.

\smallskip

Fix now a saturated subsystem $\mD$ of $\F$ over $T R$ with $O^p(\mD)=O^p(\E)$. It remains to show that $\mD=\E R$. If $\E\unlhd \mD$, then it follows from Proposition~\ref{P:DinNFE} that $\mD$ is contained in $N_\F(\E)\subseteq \F_0$ and is thus equal to $(\E R)_{\F_0}=\E R$ by Proposition~\ref{P:ERnormal} applied with $\F_0$ in place of $\F$. Hence, 
\begin{equation}\label{E:ReductionEunlhdD}
\mD=\E R\mbox{ if }\E\unlhd\mD.
\end{equation}
We have $O^p(\E)=O^p(O^p(\E))$ by \cite[Lemma~2.29]{Chermak/Henke}. This allows us to apply \eqref{E:ReductionEunlhdD} with $O^p(\E)$ in place of $\E$. Indeed, using property \eqref{E:ReductionEunlhdD} twice, once with $O^p(\E)\unlhd \mD$ and once with $O^p(\E)\unlhd \E R$ in place of $\E\unlhd\mD$, and both times with $TR$ in place of $R$, we can conclude that
\[\mD=O^p(\E)TR=\E R\]
as required.
\end{proof}

\begin{lemma}\label{L:GenNFE}
Let $P,Q\leq N_S(\E)$ and $\phi\in\Hom_{N_\F(\E)}(P,Q)$. Then there exist $\psi\in\Hom_{(\E P)_\F}(P,TP)$ and $\alpha\in\Hom_{N_\F(T,\E)}((P\psi),Q)$ such that $\phi=\psi\alpha$. In particular,
\[N_\F(\E)=\<(\E N_S(\E))_\F,\;N_\F(T,\E)\>_{N_S(\E)}.\]
\end{lemma}

\begin{proof}
Choose $S_0$ such that $N_S(\E)\leq S_0$. By Lemma~\ref{L:CrucialLemma}, $(\E R)_\F=(\E R)_{\F_0}\subseteq N_\F(\E)$ for every $R\leq N_S(\E)$. Moreover, $N_\F(T,\E)=N_{\F_0}(T)|_{N_S(\E)}\subseteq N_\F(\E)$ by Lemma~\ref{L:inAutEinNFE}(a),(c). In particular,
\[\<(\E N_S(\E))_\F,\;N_\F(T,\E)\>\subseteq N_\F(\E)\] and it suffices to prove the first part of the assertion. By Theorem~\ref{T:bNLH}(a),(b), $\F_0$ is saturated and $\E\unlhd\F_0$. Hence, by Lemma~\ref{L:FusionSystemFrattini} applied with $\F_0$ in place of $\F$, there exist $\psi\in\Hom_{(\E P)_{\F_0}}(P,TP)$ and $\alpha\in\Hom_{N_{\F_0}(T)}(P\psi,S)$ such that $\phi=\psi\alpha$. Since $P\psi\leq TP\leq N_S(\E)$ and $N_{\F_0}(T)|_{N_S(\E)}=N_\F(T,\E)$, we can observe that $\alpha$ is a morphism in $N_\F(T,\E)$. As $(\E P)_{\F_0}=(\E P)_\F\subseteq (\E N_S(\E))_\F$, this shows the assertion.
\end{proof}

\subsection{Conjugates of $\E$} \label{SS:ConjugateE}
In this subsection we will look at properties of $\F$-conjugates of $\E$ and complete the proof of Theorem~\ref{T:FusionNormalizerNew}. Some of the results we prove below are not needed in the proofs of the main theorems, but seem interesting on their own.

\smallskip

We continue to assume Hypotheses~\ref{H:main} and \ref{H:S0F0}. Recall from the introduction that 
\[\E^\F:=\{\E^\phi\colon \phi\in\Hom_\F(T,S)\}\]
denotes the set of $\F$-conjugates of $\E$.

\begin{lemma}\label{L:ConjEandT}
We have $\E^\F=\E^{N_\F(\tT)}$ and $T^\F=T^{N_\F(\tT)}$.
\end{lemma}

\begin{proof}
Let $\phi\in\Hom_\F(T,S)$. As $T\leq N_S(\E)$, it follows from Lemma~\ref{L:FrattiniApply} that there exists $\psi\in\Hom_{N_\F(\E)}(T,\tT T)$ and $\alpha\in\Hom_{N_\F(\tT)}(T\psi,S)$ such that $\phi=\psi\alpha$. By  Lemma~\ref{L:NFEbNLH}(c), $\E$ is $N_\F(\E)$-invariant. Hence, $T\psi=T$ and $\E^\psi=\E$, which yields $T\phi=T\psi\alpha=T\alpha\in T^{N_\F(\tT)}$ and similarly $\E^\phi=(\E^\psi)^\alpha=\E^\alpha\in\E^{N_\F(\tT)}$. The assertion follows. 
\end{proof}

\begin{lemma}\label{L:ConjH}~
\begin{itemize}
 \item[(a)] Let $g\in G$ be such that $T^g\leq S$. Then $\H^g\subn \L$ with $\H^g\cap S=T^g$. Moreover, setting $\phi:=c_g|_T$, we have $\E^\phi=\F_{T^g}(\H^g)\subn \F$.
 \item[(b)] For every $g\in G$, the conjugation automorphism $c_g\in\Aut(\L)$ restricts to isomorphisms from $N_G(\H)$ to $N_G(\H^g)$ and  from $\bN_\L(\H)$ to $\bN_\L(\H^g)$. For $S_0\in\Syl_p(N_G(\H))$ we have $S_0^g\in\Syl_p(N_G(\H^g))$. 
\end{itemize}
\end{lemma}

\begin{proof}
\textbf{(a)} By Lemma~\ref{L:BasisNLtT}, $c_g\in\Aut(\L)$ and thus $\H^g\subn\L$. It follows moreover from Corollary~\ref{C:AutLFusion} that $\E^\phi=\F_{T^g}(\H^g)$. By Background Theorem~\ref{bT:2}, we have $\F_{T^g}(\H^g)\subn\F$.

\smallskip

\textbf{(b)} Using again $c_g\in\Aut(\L)$, one observes easily that $N_G(\H)c_g=N_G(\H^g)$, i.e. $c_g$ induces an isomorphism from $N_G(\H)$ to $N_G(\H^g)$, which then takes a Sylow $p$-subgroup of $N_G(\H)$ to a Sylow $p$-subgroup of $N_G(\H^g)$. By \cite[Lemma~11.12]{Henke:Regular}, we have also $E(\L)c_g=E(\L)$. As $\bN_\L(\H)=E(\L)N_G(\H)$ and $\bN_\L(\H^g)=E(\L)N_G(\H^g)$ the assertion follows.
\end{proof}

\begin{lemma}\label{L:ConjugatesSubnormal}
Every $\F$-conjugate of $\E$ is a subnormal subsystem of $\F$.
\end{lemma}

\begin{proof}
As $N_\F(\tT)=\F_S(G)$, this follows from Lemma~\ref{L:ConjEandT} and Lemma~\ref{L:ConjH}(a).
\end{proof}

Recall that we call $\E$ \emph{fully $\F$-normalized}, if 
\[|N_S(\E)|\geq |N_S(\tE)|\mbox{ for all }\tE\in\E^\F.\]
Notice that $\E^\F=(\E')^\F$ for every $\E'\in\E^\F$. Hence, we can always choose a fully $\F$-normalized $\F$-conjugate of $\E$. 

\begin{definition}
The subsystem $\E$ is called \emph{fully $\F$-automized} if
\[\Aut_S(T)\cap \Aut(\E)\in\Syl_p(\Aut_\F(T)\cap \Aut(\E)).\]
\end{definition}

We will prove below that $\E$ is fully $\F$-normalized if and only if $T$ is fully $\F$-centralized and $\E$ is fully $\F$-automized. 

\begin{lemma}\label{L:EFullyAutomized} Suppose $N_S(\E)\leq S_0$. Then the following hold.
\begin{itemize}
\item [(a)] $\Aut_{S_0}(T)$ is a Sylow $p$-subgroup of $\Aut_{\F_0}(T)$ and of 
\[\Aut_{\F_0}(T)\cap \Aut_\F(T)=\Aut_{N_\F(\E)}(T)=\Aut_\F(T)\cap\Aut(\E).\]
\item [(b)] The following properties are equivalent.  
\begin{itemize}
 \item [(i)] $\E$ is fully $\F$-automized;
 \item [(ii)] $T$ is fully $N_\F(\E)$-automized. 
\item [(iii)] $\Aut_{N_S(\E)}(T)=\Aut_{S_0}(T)$.
\end{itemize}
\end{itemize}
\end{lemma}

\begin{proof}
\textbf{(a)} As $\F_0:=\F_{S_0}(\bN_\L(\H))$ is saturated and $T\unlhd S_0$, we have $\Aut_{S_0}(T)\in\Syl_p(\Aut_{\F_0}(T))$. Notice moreover that $\Aut_{S_0}(T)\leq \Aut_\F(T)$ since $\F=\F_S(\L)$. Hence, $\Aut_{S_0}(T)$ is also a Sylow $p$-subgroup of $\Aut_{\F_0}(T)\cap \Aut_\F(T)$. It follows from Lemma~\ref{L:NFEbNLH}(a) that
\[\Aut_\F(T)\cap \Aut_{\F_0}(T)=\Aut_{\F\cap\F_0}(T)=\Aut_{N_\F(\E)}(T).\]
By Lemma~\ref{L:NFEbNLH}(c), $\E$ is $N_\F(\E)$ invariant and so $\Aut_{N_\F(\E)}(T)\leq \Aut_\F(T)\cap \Aut(\E)$. The converse inclusion follows from Lemma~\ref{L:inAutEinNFE}(a). So (a) holds.

\smallskip

\textbf{(b)} Notice that $\Aut_S(T)\cap \Aut(\E)=\Aut_{N_S(\E)}(T)$. Hence, by (a), property (i) is equivalent to $\Aut_{N_S(\E)}(T)\in\Syl_p(\Aut_{N_\F(\E)}(T))$, i.e. to (ii). As $N_S(\E)\leq S_0$ by assumption, we have $\Aut_{N_S(\E)}(T)\leq \Aut_{S_0}(T)$. So (ii) is by (a) equivalent to $\Aut_{N_S(\E)}(T)=\Aut_{S_0}(T)$.
\end{proof}

\begin{lemma}\label{L:Tfullycentralized}~
\begin{itemize}
 \item [(a)] Suppose $C_S(T)\leq S_0$. Then $T$ is fully $\F$-centralized if and only if $C_S(T)=C_{S_0}(T)$.
 \item [(b)] If $T$ is fully $\F$-centralized and $N_S(\E)\leq S_0$, then $C_\F(T)=C_{\F_0}(T)$.
\end{itemize}
\end{lemma}

\begin{proof}
\textbf{(a)} It follows from Lemma~\ref{L:finNLtT} that $C_G(T)\leq X:=N_G(\H)$ and thus $C_G(T)=C_X(T)\unlhd N_X(T)$. Notice that $S_0$ is a Sylow $p$-subgroup of $X$ which is contained in $N_X(T)$. Thus, $S_0$ is also a Sylow $p$-subgroup of $N_X(T)$. As $C_G(T)\unlhd N_X(T)$, it follows that $C_{S_0}(T)=S_0\cap C_G(T)\in\Syl_p(C_G(T))$. Since $C_S(T)\leq C_{S_0}(T)$ by assumption, we can conclude that $C_S(T)\in\Syl_p(C_G(T))$ if and only if $C_S(T)=C_{S_0}(T)$. Since, $N_\F(\tT)=\F_S(G)$, the condition $C_S(T)\in\Syl_p(C_G(T))$ is by \cite[Proposition~I.5.4]{Aschbacher/Kessar/Oliver:2011} equivalent to $T$ being fully $N_\F(\tT)$-centralized. It follows from Lemma~\ref{L:ConjEandT} that this is in turn equivalent to $T$ being fully $\F$-centralized.

\smallskip

\textbf{(b)} Note that $C_S(T)\leq N_S(\E)\leq S_0$. Hence, (b) follows from part (a) and Lemma~\ref{L:inAutEinNFE}(c).
\end{proof}

\begin{lemma}\label{L:EfullyNormalized}
The following properties are equivalent:
\begin{itemize}
 \item [(i)] $\E$ is fully $\F$-normalized;
 \item [(ii)] $N_S(\H)\in\Syl_p(N_{G}(\H))$;
 \item [(iii)] $\E$ is fully $\F$-automized and $T$ is fully $\F$-centralized.
\end{itemize}
If so, then $N_\F(\E)$ is saturated, $\E\unlhd N_\F(\E)$, $E(N_\F(\E))=E(\F)$ and $(\bN_\L(\H),\delta(N_\F(\E)),N_S(\E))$ is a regular locality over $N_\F(\E)$.
\end{lemma}

\begin{proof}
Recall that $N_S(\E)=N_S(\H)$ and $N_\F(\E)=\F_{N_S(\H)}(\bN_\L(\H))$. In particular, if (ii) holds, then Theorem~\ref{T:bNLH}(a),(b),(c) applied with $S_0=N_S(\H)$ gives that $N_\F(\E)$ is saturated, $\E\unlhd N_\F(\E)$, $E(N_\F(\E))=E(\F)$ and $(\bN_\L(\H),\delta(N_\F(\E)),N_S(\E))$ is a regular locality over $N_\F(\E)$. Thus, it remains only to prove the equivalence of properties (i)-(iii). For the proof choose $S_0\in\Syl_p(N_G(\H))$ such that $N_S(\E)=N_S(\H)\leq S_0$.

\smallskip

\textbf{(ii)$\Longleftrightarrow$(iii)} Observe that
\begin{equation}\label{E:Inequality}
|\Aut_{N_S(\E)}(T)|\cdot |C_S(T)|=|N_S(\E)|\leq |S_0|=|\Aut_{S_0}(T)|\cdot |C_{S_0}(T)|.
\end{equation}
The inequality here is an equality if and only if $N_S(\E)=S_0$, i.e. if and only if (ii) holds. We have moreover $\Aut_{N_S(\E)}(T)\leq \Aut_{S_0}(T)$ with equality if and only if $\E$ is fully automized (see Lemma~\ref{L:EFullyAutomized}(b)). As $C_S(T)\leq N_S(\E)\leq S_0$, we have moreover $C_S(T)\leq C_{S_0}(T)$. Lemma~\ref{L:Tfullycentralized}(a) gives that $C_S(T)=C_{S_0}(T)$ if and only if $T$ is fully centralized. Hence, the inequality in \eqref{E:Inequality} is also an equality if and only if (iii) holds.

\smallskip

\textbf{(i)$\Longrightarrow$(ii)} To prove this implication by contraposition, assume (ii) is false, i.e. $N_S(\E)<S_0$. As $S\in\Syl_p(G)$ and $S_0$ is a $p$-subgroup of $G$, it follows from Sylow's Theorem that there exists $g\in G$ with $S_0^g\leq S$. Notice that $T\leq N_S(\E)\leq S_0$. So in particular, $T^g\leq S$ and $\phi=c_g|_T$ is a morphism in $\F$. By Lemma~\ref{L:ConjH}(a), we have $\E^\phi=\F_{T^g}(\H^g)$. By Lemma~\ref{L:ConjH}(b), $S_0^g\leq N_G(\H^g)$ and thus $S_0^g\leq N_S(\H^g)=N_S(\E^\phi)$ by Lemma~\ref{L:finNLtT}. Hence, $|N_S(\E)|<|S_0|=|S_0^g|\leq |N_S(\E^\phi)|$ and therefore $\E$ is not fully $\F$-normalized. 

\smallskip

\textbf{(ii)$\Longrightarrow$(i)} Assume (ii) and fix $\E'\in \E^\F$ such that $\E'$ is fully $\F$-normalized. It is sufficient to show that $|N_S(\E')|\leq |N_S(\E)|$. By Lemma~\ref{L:ConjEandT} we have $\E'\in\E^{N_\F(\tT)}$. As $N_\F(\tT)=\F_S(G)$, there exists thus $g\in G$ with $T^g\leq S$ such that, setting $\phi=c_g|_T$, we have $\E'=\E^\phi$. Then Lemma~\ref{L:ConjH}(a) gives $\E'=\F_{T^g}(\H^g)$. By Lemma~\ref{L:ConjH}(b), $c_g$ induces an isomorphism from $N_G(\H)$ to $N_G(\H^g)$. By Lemma~\ref{L:finNLtT}, $N_S(\E')=N_S(\H^g)$ is a $p$-subgroup of $N_G(\H^g)$. So (ii) implies $|N_S(\E')|\leq |N_S(\H)|=|N_S(\E)|$ as required, 
\end{proof}

\begin{proof}[\textbf{Proof of Theorem~\ref{T:FusionNormalizerNew}}]
By Lemma~\ref{L:finNLtT}, $N_\F(\E)$ is a fusion system over $N_S(\E)$. Part (a) follows thus from Lemma~\ref{L:GenNFE}.
Part (b) is a restatement of Lemma~\ref{L:NFEbNLH}(c), part (c) was shown as Proposition~\ref{P:DinNFE}, part (d) follows from Lemma~\ref{L:EfullyNormalized}, and part (e) was proved as Lemma~\ref{L:inAutEinNFE}(a),(b). If $\F$ is constrained, then by Remark~\ref{R:Constrained}, $\L$ is a model for $\F$, $\H$ is a subnormal subgroup of $\L$ realizing $\E$, $\bN_\L(\H)=N_\L(\H)$ and thus $N_\F(\E)=\F_{N_S(\H)}(N_\L(\H))$. Part (f) follows thus from the uniqueness of models (cf. \cite[Theorem~III.5.10]{Aschbacher/Kessar/Oliver:2011}) and  Lemma~\ref{L:NormalizerConstrained}.
\end{proof}

We conclude this section with two lemmas which make it easier to move to fully $\F$-normalized conjugates of $\E$.

\begin{lemma}\label{L:MapNSE}
 Let $\beta\in\Hom_\F(N_S(\E),S)$. Then $N_S(\E)\beta\leq N_S(\E^\beta)$. 
\end{lemma}

\begin{proof}
If $s\in N_S(\E)$, then $c_s|_T\in\Aut(\E)$. So using \cite[Remark~2.1]{Henke:2013}, one can observe that 
\[c_{s\beta}|_{T\beta}=\beta^{-1}(c_s|_T)\beta\in\Aut(\E^\beta)\]
and thus $s\beta\in N_S(\E^\beta)$.
\end{proof}

\begin{lemma}
 \begin{itemize}
 \item [(a)] Let $\hat{\E}\in\E^\F$ be such that $\hat{\E}$ is fully $\F$-normalized. Then there exists $\alpha\in\Hom_\F(N_S(\E),S)$ such that  $\E^\alpha=\hat{\E}$.
 \item [(b)] Fix $\alpha\in\Hom_\F(N_S(\E),S)$ such that $\E^\alpha$ is fully normalized. Then for all $P,Q\leq N_S(\E)$,  
\[\Hom_{N_\F(\E)}(P,Q)=\{\phi\in\Hom_\F(P,Q)\colon \alpha^{-1}\phi\alpha\in\Hom_{N_\F(\E^\alpha)}(P\alpha,Q\alpha)\},\]
i.e. $\alpha$ induces an isomorphism from $N_\F(\E)$ to $N_\F(\E^\alpha)|_{N_S(\E)\alpha}$. In particular, if $\E$ is also fully normalized, then $N_S(\E)\alpha=N_S(\E^\alpha)$ and $\alpha$ induces an isomorphism from $N_\F(\E)$ to $N_\F(\E^\alpha)$.
 \end{itemize}
\end{lemma}

\begin{proof}
\textbf{(a)} Let $\hat{\E}$ be a subsystem over $\hat{T}$. As $\hat{\E}\in\E^\F$, there exists $\phi\in\Hom_\F(T,S)$ such that $\hat{T}=T\phi$ and $\hat{\E}=\E^\phi$. Notice that $(\Aut_\F(T)\cap \Aut(\E))^\phi=\Aut_\F(\hat{T})\cap \Aut(\hat{\E})$. In particular, $\Aut_{N_S(\E)}(T)^\phi$ is a $p$-subgroup of $\Aut_\F(\hat{T})\cap \Aut(\hat{\E})$. As $\hat{\E}$ is fully normalized, Lemma~\ref{L:EfullyNormalized} gives that $\hat{\E}$ is fully automized, i.e. $\Aut_S(\hat{T})\cap \Aut(\hat{\E})\in\Syl_p(\Aut_\F(\hat{T})\cap \Aut(\hat{\E}))$. Hence, by Sylow's Theorem there exists $\gamma\in\Aut_\F(\hat{T})\cap\Aut(\hat{\E})$ such that $\Aut_{N_S(\E)}(T)^{\phi\gamma}\leq \Aut_S(\hat{T})$. This means that $N_S(\E)\leq N_{\phi\gamma}$. Hence, by the extension axiom \cite[Proposition~I.2.5]{Aschbacher/Kessar/Oliver:2011}, $\phi\gamma$ extends to $\alpha\in\Hom_\F(N_S(\E),S)$. As $\gamma\in\Aut(\hat{\E})$, we have $\E^\alpha=\E^{\phi\gamma}=\hat{\E}^\gamma=\hat{\E}$. This proves (a).

\smallskip

\textbf{(b)} Let $\alpha\in\Hom_\F(N_S(\E),S)$ be such that $\E^\alpha$ is fully normalized. By Lemma~\ref{L:MapNSE}, $N_S(\E)\alpha\leq N_S(\E^\alpha)$. If $\E$ is fully normalized, then $|N_S(\E)\alpha|=|N_S(\E)|\geq |N_S(\E^\alpha)|$ and so $N_S(\E)\alpha=N_S(\E^\alpha)$.

\smallskip

By Lemma~\ref{L:BasicT0perp}(a), we have $\tT=T_0T_0^\perp\subseteq TC_S(\H)\leq N_S(\H)=N_S(\E)$. As $\tT=E(\L)\cap S$ is strongly $\F$-closed, it follows that $\tT\alpha=\tT$ and $\alpha$ is a morphism in $N_\F(\tT)=\F_S(G)$. Thus, we may pick $g\in G$ such that $\alpha=c_g|_{N_S(\E)}$. 

\smallskip

Recall that $N_S(\E)=N_S(\H)$. By Lemma~\ref{L:ConjH}(a), we have $\E^\alpha=\F_{T^g}(\H^g)$ and $T^g=\H^g\cap S$. In particular, $N_S(\E^\alpha)=N_S(\H^g)$ by Lemma~\ref{L:finNLtT}. As stated in Lemma~\ref{L:ConjH}(b), $c_g$ induces an isomorphism from $N_G(\H)$ to $N_G(\H^g)$. Since $\E^\alpha$ is fully normalized, it follows from Lemma~\ref{L:EfullyNormalized} that $N_S(\E^\alpha)=N_S(\H^g)\in\Syl_p(N_G(\H^g))$. Hence, $N_S(\E^\alpha)^{g^{-1}}$ is a Sylow $p$-subgroup of $N_G(\H)$. This allows us to choose $S_0$ such that
\[S_0=N_S(\E^\alpha)^{g^{-1}}.\]
By Lemma~\ref{L:MapNSE}, we have $N_S(\E)^g=N_S(\E)\alpha\leq N_S(\E^\alpha)$ and thus $N_S(\E)\leq S_0$. Now by Lemma~\ref{L:ConjH}(b), $c_g$ induces an isomorphism from $\bN_\L(\H)$ to $\bN_\L(\H^g)$ with $S_0c_g=N_S(\E^\alpha)=N_S(\H^g)$. Hence, using Theorem~\ref{T:bNLH}(a), it follows from Lemma~\ref{L:AutLFusion} that $c_g$ induces an isomorphism from $\F_0:=\F_{S_0}(\bN_\L(\H))$ to $N_\F(\E^\alpha)=\F_{N_S(\H^g)}(\bN_\L(\H^g))$.

\smallskip

If $\E$ is fully normalized, then Lemma~\ref{L:EfullyNormalized} implies $N_S(\E)=S_0$ and thus $\F_0=N_\F(\E)$, i.e. $\alpha$ induces an isomorphism from $N_\F(\E)$ to $N_\F(\E^\alpha)$. In the general case, Lemma~\ref{L:NFEbNLH}(a) gives $N_\F(\E)=\F_0|_{N_S(\E)}$ and thus $\Hom_{N_\F(\E)}(P,Q)=\Hom_{\F_0}(P,Q)$ for $P,Q\leq N_S(\E)$. The assertion follows thus, as $c_g$ induces an isomorphism $\F_0\rightarrow N_\F(\E^\alpha)$ with $c_g|_{N_S(\E)}=\alpha$.
\end{proof}

\subsection{The centralizer of $\E$ and fully centralized conjugates of $\E$}\label{SS:Centralizer}

We continue to assume Hypotheses~\ref{H:main} and \ref{H:S0F0}. In particular, $S_0$ is a Sylow $p$-subgroup of $N_G(\H)$ containing $T$.

\smallskip

If $\E$ is normal in $\F$, then it was first shown by Aschbacher that there is a unique maximal element of the set 
\[\{R\leq S\colon \E\subseteq C_\F(R)\},\]
which we denote by $C_S(\E)=C_S(\E)_\F$ (cf. \cite[Chapter~6]{Aschbacher:2011} or \cite[Theorem~1(a)]{Henke:2018}). We extend this result now to the case that $\E$ is subnormal in $\F$. As $\F_0$ is saturated and $\E\unlhd \F_0$, we can use that the subgroup $C_{S_0}(\E)_{\F_0}$ is defined as the unique maximal element of the set $\{R\leq S_0\colon \E\subseteq C_{\F_0}(R)\}$. Notice that the following proposition implies Theorem~\ref{T:CentralizerEinS}.

\begin{prop}\label{P:CSE}
The subgroup $C_S(\H)$ is with respect to inclusion the unique maximal element of the set $\m{X}:=\{R\leq S\colon \E\subseteq C_\F(R)\}$. If $C_S(\H)\leq S_0$, then $C_S(\H)=S\cap C_{S_0}(\E)_{\F_0}$.
\end{prop}

\begin{proof}
By \cite[Lemma~3.5]{Henke:Regular}, we have $\H\subseteq C_\L(C_S(\H))$, which implies easily that $\E=\F_T(\H)\subseteq C_\F(C_S(\H))$. Hence, $C_S(\H)\in\X$. 

\smallskip

Choose now $S_0$ such that $C_S(\H)\leq S_0$. Recall from Theorem~\ref{T:bNLH}(a),(b) that $\F_0$ is saturated and $(\bN_\L(\H),\delta(\F_0),S_0)$ is a regular locality over $\F_0$ with $\H\unlhd\bN_\L(\H)$, $T=S_0\cap \H$ and $\E\unlhd\F_0$. Hence, it follows from \cite[Proposition~4]{Henke:2015} that $C_{S_0}(\E)_{\F_0}=C_{S_0}(\H)$. As $C_S(\H) \leq S_0$ by assumption, we can conclude that   $C_S(\H)=S\cap C_{S_0}(\H)=S\cap  C_{S_0}(\E)_{\F_0}$. 

\smallskip

Choose now $S_0$ such that $N_S(\E)\leq S_0$. We have then $N_\F(\E)\subseteq \F_0$ by Lemma~\ref{L:NFEbNLH}(a). Let $R\in\X$. Then $\E$ and $\F_R(R)$ commute in the sense of Definition~\ref{D:CommuteSubsystems}. Hence, by Lemma~\ref{L:CentralProduct}, $\mD:=\E*\F_R(R)$ is a saturated subsystem of $\F$ over $T R$ with $\E\unlhd\mD$ and $\E\subseteq C_{\mD}(R)$. Proposition~\ref{P:DinNFE} yields now $\mD\subseteq N_\F(\E)\subseteq \F_0$. In particular, $R\leq N_S(\E)\leq S_0$. As $\E\subseteq C_\mD(R)\subseteq C_{\F_0}(R)$, it follows $R\leq C_{S_0}(\E)_{\F_0}$. Hence, $R\leq S\cap C_{S_0}(\E)_{\F_0}=C_S(\H)$. This proves the assertion.
\end{proof}

\begin{definition}
 We write $C_S(\E)$ or $C_S(\E)_\F$ for $C_S(\H)$, i.e. for the unique maximal element of the set 
\[\{R\leq S\colon \E\subseteq C_\F(R)\}.\] 
We call $C_S(\E)=C_S(\E)_\F$ the \emph{centralizer of $\E$ in $S$}.
\end{definition}

The reader might want to observe that 
\[C_S(\E)=C_S(\H)\leq N_S(\H)=N_S(\E)\] 
by Lemma~\ref{L:finNLtT}. 

\smallskip

Recall from the introduction that $C_\F(\E)$ is defined as 
\[C_\F(\E):=\F_{C_S(\H)}(C_\L(\H)).\]
We call $C_\F(\E)$ the \emph{centralizer of $\E$ in $\F$}.

\smallskip

Note that $C_\F(\E)$ depends a priori not only on $\F$ and $\E$, but also on $\L$. However, the properties we show below imply that $C_\F(\E)$ is indeed independent of the choice of $\L$. Alternatively, this can be seen more directly using Background Theorems~\ref{bT:0} and \ref{bT:2} combined with Lemma~\ref{L:AutLFusion}. 

\smallskip

The reader is referred to \cite[Definition~I.7.1]{Aschbacher/Kessar/Oliver:2011} for the definition of the focal subgroup $\foc(\F)$ and the hyperfocal subgroup $\hyp(\F)$. We extend the definition of the focal subgroup to fusion systems which are not saturated. Namely, for every fusion system $\G$ over a $p$-group $X$, we set  
\[\foc(\G):=\<x^{-1}(x\phi)\colon x\in X,\;\phi\in\Hom_\F(\<x\>,X)\>.\]
It is always true that $\hyp(\F)\leq \foc(\F)$. For every subgroup $R\leq S$ with $\hyp(\F)\leq R$, it is shown in \cite[Theorem~I.7.4]{Aschbacher/Kessar/Oliver:2011} that
\begin{equation}\label{E:DefineFR}
 \F_R:=\<O^p(\Aut_\F(P))\colon P\leq R\>_R
\end{equation}
is the unique saturated subsystem of $\F$ over $R$ of $p$-power index. This concept will play a role in obtaining a description of $C_\F(\E)$ if $\E$ is fully centralized. 

\smallskip

If $\E$ is normal in $\F$, then a subsystem $C_\F(\E)$ was first defined by Aschbacher \cite[Chapter~6]{Aschbacher:2011}, and a new construction of $C_\F(\E)$ was given in \cite{Henke:2018}. However, if $\E$ is normal, then by Background Theorem~\ref{bT:1}, we have $\H\unlhd\L$. Hence, \cite[Proposition~6.7]{Chermak/Henke} states in this case that the subsystem $C_\F(\E)$ defined in \cite{Aschbacher:2011,Henke:2018} equals $\F_{C_S(\H)}(C_\L(\H))$ and coincides thus with the centralizer $C_\F(\E)$ defined in this paper. Using this, the following theorem follows from \cite[Proposition~1, Theorem~2]{Henke:2018} and the construction of $C_\F(\E)$ given in \cite{Henke:2018}.

\begin{theorem}\label{T:Henke2018}
Suppose $\E\unlhd\F$. Then $\hyp(C_\F(T))\leq \foc(C_\F(T))\leq C_S(\E)$. Moreover, 
\[C_\F(\E)=C_\F(T)_{C_S(\E)}\]
is normal in $\F$. Every saturated subsystem of $\F$ which commutes with $\E$ is contained in $C_\F(\E)$. \qed
\end{theorem} 

We use here the notation introduced in \eqref{E:DefineFR} for $C_\F(T)$ in place of $\F$. Notice that this is possible, as $T$ is in the situation above fully centralized and thus $C_\F(T)$ is saturated.

\begin{lemma}\label{L:CF0ECollect}
The centralizer $C_{\F_0}(\E)$ is normal in $\F_0$, $\foc(C_{\F_0}(T))\leq C_{S_0}(\E):=C_{S_0}(\E)_{\F_0}$ and
\[C_{\F_0}(\E)=\F_{C_{S_0}(\H)}(C_\L(\H))=C_{\F_0}(T)_{C_{S_0}(\E)}.\]
Moreover, $(C_\L(\H),\delta(C_{\F_0}(\E)),C_{S_0}(\H))$ is a regular locality over $C_{\F_0}(\E)$. In particular, $C_{S_0}(\H)=C_{S_0}(\E)_{\F_0}$ is a maximal $p$-subgroup of $C_\L(\H)$.
\end{lemma}

\begin{proof}
As $\F_0$ is saturated and $\E\unlhd\F_0$, it follows from Theorem~\ref{T:Henke2018} that $C_{\F_0}(\E)=C_{\F_0}(T)_{C_{S_0}(\E)}$ is normal in $\F_0$. Recall that $(\bN_\L(\H),\delta(\F_0),S_0)$ is a regular locality over $\F_0$ and $C_\L(\H)=C_{\bN_\L(\H)}(\H)$ by Lemma~\ref{L:CLTinbNLH}. Hence, the remaining part of the assertion follows from Background Theorem~\ref{bT:2} and the definition $C_{\F_0}(\E)$. 
\end{proof}

Recall that $\E$ is called \emph{fully $\F$-centralized} if
\[|C_S(\E)|\geq |C_S(\tE)|\mbox{ for all }\tE\in\E^\F.\]

\begin{lemma}\label{L:EFullyCentralized}~
\begin{itemize}
\item [(a)] Choose $S_0$ such that $C_S(\H)=C_S(\E)\leq S_0$. Then the following are equivalent:
\begin{itemize}
 \item [(i)] $\E$ is fully $\F$-centralized;
 \item [(ii)] $C_S(\H)=C_{S_0}(\H)$;
 \item [(iii)] $C_S(\H)$ is a maximal $p$-subgroup of $C_\L(\H)$.
\end{itemize}
\item [(b)] Let $\hat{\E}\in\E^\F$ be such that $\hat{\E}$ is fully $\F$-centralized. Then there exists $\alpha\in\Hom_\F(TC_S(\E),S)$ such that $\E^\alpha=\hat{\E}$.
\item [(c)] If $\alpha\in\Hom_\F(TC_S(\E),S)$, then $\alpha$ is a morphism in $N_\F(\tT)$ and $C_S(\E)\alpha\leq C_S(\E^\alpha)$.
\end{itemize}
\end{lemma}

\begin{proof}
\textbf{(a,b)} Choose $S_0$ such that $C_S(\H)\leq S_0$. By Lemma~\ref{L:CF0ECollect}, $C_{S_0}(\H)\geq C_S(\H)$ is a maximal $p$-subgroup of $C_\L(\H)$. Hence, (ii) and (iii) are equivalent. Thus, for this part of the proof, it remains only to show that (i) and (ii) are equivalent and that (b) holds.

\smallskip

\textit{Step~1:} We show that (i) implies (ii). Notice that $TC_{S_0}(\H)$ is a $p$-subgroup of $N_G(\H)$. As $S\in\Syl_p(G)$, there exists thus $g\in G$ with $(TC_{S_0}(\H))^g\leq S$. As $T^g\leq S$, $\phi=c_g|_T$ is an $\F$-morphism and $T\phi=T^g\in T^\F$. By Lemma~\ref{L:ConjH}(a), we have $T^g=\H^g\cap S$ and $\E^\phi=\F_{T^g}(\H^g)$. In particular, $C_S(\E^\phi)=C_S(\H^g)$. Using Lemma~\ref{L:ConjH}(b), one observes that $C_{S_0}(\H)^g\leq C_S(\H^g)$. Thus, if $C_S(\H)<C_{S_0}(\H)$, then $|C_S(\E)|=|C_S(\H)|<|C_S(\H^g)|=|C_S(\E^\phi)|$ and $\E$ is thus not fully centralized. So Step~1 is complete.

\smallskip

\emph{Step~2:} We show that (b) holds and (ii) implies (i).  For the proof fix $\hat{\E}\in \E^\F$ fully centralized (and notice that such $\hat{\E}$ always exists). By Lemma~\ref{L:ConjEandT}, there is $\phi\in\Hom_{N_\F(\tT)}(T,S)$ such that $\E^\phi=\hat{\E}$. As $N_\F(\tT)=\F_S(G)$, there exists $g\in G$ with $\phi=c_g|_T$. By Lemma~\ref{L:ConjH}(a), we have $\F_{T^g}(\H^g)=\hat{\E}$ and thus $C_S(\hat{\E})=C_S(\H^g)$.

\smallskip

Fix $\hat{S}_0\in \Syl_p(N_G(\H^g))$ such that 
\[T^gC_S(\H^g)\leq \hat{S}_0.\]
As $T$ is a maximal $p$-subgroup of $\H$, the conjugate $T^g=T\phi$ is a maximal $p$-subgroup of $\H^g$ and thus equal to $S\cap \H^g$. By Theorem~\ref{T:bNLH}(b) applied with $\H^g$ in place of $\H$, we have moreover $\H^g\cap \hat{S}_0=T^g=S\cap \H^g$.

Since $\hat{\E}$ is fully $\F$-centralized and we already showed in  Step~1 that (i) implies (ii), we can apply this implication for $(\hat{\E},\H^g)$ in place of $(\E,\H)$ to obtain
\[C_S(\H^g)=C_{\hat{S}_0}(\H^g).\]
Notice that $S_0^g$ is a Sylow $p$-subgroup of $N_G(\H^g)$ as $c_g$ induces an isomorphism from $N_G(\H)$ to $N_G(\H^g)$ by Lemma~\ref{L:ConjH}(b). Hence, there exists $x\in N_G(\H^g)$ such that $S_0^{gx}=\hat{S}_0$. Observe that $\H^g=\H^{gx}$. Moreover, $T^{gx}\leq S_0^{gx}\cap \H^g=\hat{S}_0\cap \H^g=T^g$ and thus $T^{gx}=T^g\leq S$. Thus, replacing $\phi$ and $g$ by $c_{gx}$ and $gx$, we may assume that \[S_0^g=\hat{S}_0.\] 
Hence, using $c_g\in\Aut(\L)$, we may conclude that 
\begin{equation}\label{E:CSHg}
C_S(\hat{\E})=C_S(\H^g)=C_{\hat{S}_0}(\H^g)=C_{S_0}(\H)^g\geq C_S(\H)^g=C_S(\E)^g.
\end{equation}
So $\alpha:=c_g|_{TC_S(\E)}$ is an element of $\Hom_{N_\F(\tT)}(TC_S(\E),S)$ with $\E^\alpha=\hat{\E}$. This proves (b).

\smallskip

If we assume in addition that (ii) holds (i.e. $C_S(\H)=C_{S_0}(\H)$), then we obtain from \eqref{E:CSHg} that $C_S(\E)^g=C_S(\H)^g=C_{S_0}(\H)^g=C_S(\hat{\E})$. As $\hat{\E}$ was chosen to be fully centralized, it follows then that $\E$ is fully centralized, i.e. (i) holds. This completes Step~2 and thus the proof of (a) and (b).

\smallskip

\textbf{(c)} By Lemma~\ref{L:BasicT0perp}(a), we have $\tT=T_0T_0^\perp\subseteq TC_S(\H)=TC_S(\E)$. As $\tT$ is strongly $\F$-closed, it follows that a morphism $\alpha\in\Hom_\F(TC_S(\E),S)$ is a morphism in $N_\F(\tT)$. As $N_\F(\tT)=\F_S(G)$, there exists thus $g\in G$ with $\alpha=c_g|_{TC_S(\E)}$. By Lemma~\ref{L:ConjH}(a), we have $S\cap \H^g=T^g$ and $\E^\alpha=\F_{T^g}(\H^g)$. In particular, $C_S(\E^\alpha)=C_S(\H^g)$. Using Lemma~\ref{L:ConjH}(b), we can conclude that $C_S(\E)\alpha=C_S(\E)^g=C_S(\H)^g\leq S\cap C_G(\H^g)=C_S(\H^g)=C_S(\E^\alpha)$.
\end{proof}

\begin{prop}\label{P:CFEDescribe}
If $\E$ is fully centralized, then $C_\F(\E)$ is saturated, $\foc(C_\F(T))\leq C_S(\E)$ and 
\[C_\F(\E)=\<O^p(\Aut_{C_\F(T)}(P))\colon P\leq C_S(\E)\>_{C_S(\E)}.\]
If in addition $S_0$ is chosen such that $C_S(\E)\leq S_0$, then $C_\F(\E)=C_{\F_0}(\E)$ is normal in $\F_0$.
\end{prop}

\begin{proof}
Choose first $S_0$ such that $C_S(\E)\leq S_0$. Assuming that $\E$ is fully $\F$-centralized, part (a) gives $C_{S_0}(\E)_{\F_0}=C_{S_0}(\H)=C_S(\H)=C_S(\E)$. Hence, using the definition of $C_\F(\E)$ and Lemma~\ref{L:CF0ECollect}, we can conclude that 
\begin{equation}\label{E:Foc}
\foc(C_{\F_0}(T))\leq C_S(\E),
\end{equation}
and that
\begin{equation}\label{E:CFE}
C_\F(\E)=\F_{C_S(\H)}(C_\L(\H))=\F_{C_{S_0}(\H)}(C_\L(\H))=C_{\F_0}(\E)=C_{\F_0}(T)_{C_S(\E)} 
\end{equation}
is saturated and normal in $\F_0$. It remains to show that $\foc(C_\F(T))\leq C_S(\E)$ and $C_\F(\E)=\<O^p(\Aut_{C_\F(T)}(P))\colon P\leq C_S(\E)\>_{C_S(\E)}$. For the proof we may choose $S_0$ such that $N_S(\E)\leq S_0$. By Lemma~\ref{L:inAutEinNFE}(c), we have then
\[C_\F(T)=C_{\F_0}(T)|_{C_S(T)}.\]
In particular, if $x\leq C_S(T)$ and $\phi\in\Hom_{C_\F(T)}(x,C_S(T))$, then $\phi$ is a morphism in $C_{\F_0}(T)$ and $x^{-1}(x\phi)\in\foc(C_{\F_0}(T))\leq C_S(\E)$ by \eqref{E:Foc}. This shows $\foc(C_\F(T))\leq C_S(\E)$.

\smallskip

As $C_S(\E)\leq C_S(T)$, it follows moreover that $\Aut_{C_{\F_0}(T)}(P)=\Aut_{C_\F(T)}(P)$ for every $P\leq C_S(\E)$. Hence, using \eqref{E:DefineFR} and \eqref{E:CFE}, we can conclude that  
\begin{eqnarray*}
C_\F(\E)&=&C_{\F_0}(T)_{C_S(\E)}=\<O^p(\Aut_{C_{\F_0}(T)}(P))\colon P\leq C_S(\E)\>_{C_S(\E)}\\
&=&\<O^p(\Aut_{C_\F(T)}(P))\colon P\leq C_S(\E)\>_{C_S(\E)}.
\end{eqnarray*}
This proves the assertion.
\end{proof}

\begin{lemma}\label{L:CentralizerTFullyCentralized}~
\begin{itemize}
\item [(a)] Suppose $T$ is fully $\F$-centralized. Then $\E$ is fully $\F$-centralized, $\foc(C_\F(T))\leq C_S(\E)$ and $C_\F(\E)=C_\F(T)_{C_S(\E)}$.
\item [(b)] If $\E$ is fully $\F$-normalized, then $C_\F(\E)=C_{N_\F(\E)}(\E)\unlhd N_\F(\E)$. 
\end{itemize}
\end{lemma}

\begin{proof}
\textbf{(a)} Assume $T$ is fully $\F$-centralized and choose $S_0$ such that $C_S(T)\leq S_0$. Lemma~\ref{L:Tfullycentralized}(a) gives then that $C_S(T)=C_{S_0}(T)$. In particular, $C_S(\H)=C_{C_S(T)}(\H)=C_{C_{S_0}(T)}(\H)=C_{S_0}(\H)$. Hence, by Lemma~\ref{L:EFullyCentralized}(a), $\E$ is fully centralized. Using that $C_\F(T)$ is saturated, part (a) follows now from Proposition~\ref{P:CFEDescribe} and \eqref{E:DefineFR}.

\smallskip

\textbf{(b)} If $\E$ is fully normalized, then Lemma~\ref{L:EfullyNormalized} yields $N_S(\E)=N_S(\H)\in\Syl_p(N_G(\H))$. Hence, we may assume that $S_0=N_S(\E)$ and thus $\F_0=N_\F(\E)$. The assertion follows then from Lemma~\ref{L:CF0ECollect}.
\end{proof}

Note that Lemma~\ref{L:EfullyNormalized} and Lemma~\ref{L:CentralizerTFullyCentralized}(a) show in particular the implications stated in \eqref{E:FullyNormFullyCent} in the introduction.

\smallskip

We now want to consider the properties of $C_\F(\E)$ if $\E$ is not necessarily fully centralized. The reader might want to recall Definition~\ref{D:WeaklyInvariant}.

\begin{lemma}\label{L:CFEWeaklyInvariant}
Suppose $S_0$ is chosen such that $C_S(\H)=C_S(\E)\leq S_0$. Then 
\[C_\F(\E)=\F\cap C_{\F_0}(\E)=C_{\F_0}(\E)|_{C_S(\E)}.\]
In particular, $C_\F(\E)$ is weakly $N_\F(\E)$-invariant.
\end{lemma}

\begin{proof}
By Lemma~\ref{L:CF0ECollect}, $\C:=C_{\F_0}(\E)=\F_{C_{S_0}(\H)}(C_\L(\H))$ is a normal subsystem of $\F_0$. In particular, $\C$ is weakly $\F_0$-invariant, which implies that $\F\cap \C$ is weakly $(\F\cap \F_0)$-invariant. By Lemma~\ref{L:NFEbNLH}(a), we have $N_\F(\E)=\F\cap \F_0$ if $S_0$ is chosen such that $N_S(\E)\leq S_0$. Hence, it is sufficient to prove that $C_\F(\E)=\F\cap \C=\C|_{C_S(\E)}$. 

\smallskip

By Proposition~\ref{P:CSE}, both $C_\F(\E)$ and $\F\cap \C$ are fusion systems over 
\[C_S(\H)=C_S(\E)=S\cap C_{S_0}(\E)_{\F_0}=S\cap C_{S_0}(\H).\]
Notice that $C_\F(\E)=\F_{C_S(\H)}(C_\L(\H))$ is contained in $\C=\F_{C_{S_0}(\H)}(C_\L(\H))$ and thus in $\F\cap \C$. Hence
\[C_\F(\E)\subseteq \C\cap \F\subseteq \C|_{C_S(\E)}\] 
and it remains to show that $\C|_{C_S(\E)}\subseteq C_\F(\E)$.

\smallskip

Recall from Lemma~\ref{T:bNLH}(a),(b),(c) that $(\bN_\L(\H),\delta(\F_0),S_0)$ is a regular locality over $\F_0$ with $\H\unlhd\bN_\L(\H)$ and $\Comp(\L)=\Comp(\bN_\L(\H))$. Since $C_\L(\H)=C_{\bN_\L(\H)}(\H)$ by Lemma~\ref{L:CLTinbNLH}, it follows from  \cite[Lemma~11.16, Theorem~10.16(e)]{Henke:Regular} (applied with $(\bN_\L(\H),\H)$ in place of $(\L,\N)$) that 
\[\Comp(C_\L(\H))=\Comp(\bN_\L(\H))\backslash\Comp(\H)=\Comp(\L)\backslash\Comp(\H).\]
Thus, $E(C_\L(\H))=\M$ by definition of $\M$. Here $\M$ is subnormal in $\L$ as $\M\unlhd E(\L)\unlhd\L$ by Lemma~\ref{L:BasicT0perp}(a). Thus, it follows from Background Theorem~\ref{bT:1} and the definition of a locality that $T_0^\perp=\M\cap S$ is a maximal $p$-subgroup of $\M$. As $T_0^\perp\leq C_S(\H)\leq C_{S_0}(\H)$ by Lemma~\ref{L:BasicT0perp}(a), it follows that $T_0^\perp=C_{S_0}(\H)\cap \M$. Hence, $E(\C)=\F_{T_0^\perp}(\M)$ by Background Theorem~\ref{bT:2} applied with $(C_\L(\H),C_{S_0}(\H),\C)$ in place of $(\L,S,\F)$. 

\smallskip

Let now $P,Q\leq C_S(\E)=C_S(\H)$ and $\phi\in\Hom_{\C}(P,Q)$. We need to show that $\phi$ is a morphism in $C_\F(\E)$. By Lemma~\ref{L:FusionSystemFrattini} (applied with $\C$ and $E(\C)$ in place of $\F$ and $\E$), there exist $\psi\in\Hom_{E(\C)P}(P,PT_0^\perp)$ and a morphism $\alpha\colon P\psi\rightarrow Q$ in $N_{\C}(T_0^\perp)$ such that $\phi=\psi\alpha$. By \cite[Theorem~D(a)]{Chermak/Henke}, we have $E(\C)P=\F_{T_0^\perp P}(\M P)$. Notice that $T_0^\perp P\leq C_S(\H)$ and $\M P\subseteq C_\L(\H)$. Hence, $E(\C)P$ is contained in $C_\F(\E)=\F_{C_S(\H)}(C_\L(\H))$ and $\psi$ is a morphism in $C_\F(\E)$. In particular, $P\psi\leq C_S(\H)$.

\smallskip

By Lemma~\ref{L:BasisNLtT} applied with $C_\L(\H)$ in place of $\L$, we have $T_0^\perp=\M\cap S_0=E(C_\L(\H))\cap S_0\in\delta(\C)$. Hence, by \cite[Lemma~3.10(b)]{Grazian/Henke}, $N_{\C}(T_0^\perp)$ is realized by the group $N_{C_\L(\H)}(T_0^\perp)$. In particular, $\alpha=c_g|_{P\psi}$ for some $g\in N_{C_\L(\H)}(T_0^\perp)\subseteq C_\L(\H)$. As $P\psi$ and $Q$ are contained in $C_S(\H)=C_S(\E)$, it follows that $\alpha$ is a morphism in $C_\F(\E)$. Hence, $\phi=\psi\alpha$ is a morphism in $C_\F(\E)$. This shows $\C|_{C_S(\E)}\subseteq C_\F(\E)$ and thus $C_\F(\E)=\F\cap \C=\C|_{C_S(\E)}$ as required.
\end{proof}

\begin{lemma}\label{L:CommuteWithE}
Every saturated subsystem of $\F$ which commutes with $\E$ is contained in $C_\F(\E)$.
\end{lemma}

\begin{proof}
 Let $\G$ be a saturated subsystem of $\F$ which commutes with $\E$. Then by Lemma~\ref{L:CentralProduct}, $\mD:=\E *\G$ is a saturated subsystem of $\F$ such that $\E\unlhd\mD$, and moreover $\E$ and $\G$ commute in $\mD$. Hence, by Proposition~\ref{P:DinNFE}, $\mD\subseteq N_\F(\E)$. Suppose $S_0$ is chosen such that $N_S(\E)\leq S_0$. Then $\mD\subseteq N_\F(\E)\subseteq \F_0$. As $\E$ and $\G$ commute in $\mD$, they commute also in $\F_0$. Since $\E\unlhd \F_0$, it follows now from Theorem~\ref{T:Henke2018} that $\G\subseteq C_{\F_0}(\E)$. Using Lemma~\ref{L:CFEWeaklyInvariant}, we can conclude that $\G\subseteq \F\cap C_{\F_0}(\E)=C_\F(\E)$.
\end{proof}

\begin{proof}[\textbf{Proof of Theorem~\ref{T:FusionCentralizer}}]
We have $C_\L(\H)\subseteq C_\L(T)$ and, by \cite[Lemma~3.5]{Henke:Regular}, $\H\subseteq C_\L(C_\L(\H))\subseteq C_\L(C_S(\H))$. In particular, $T\cap C_S(\H)\leq Z(\H)\cap Z(C_\L(\H))$. As $\E=\F_T(\H)$ and $C_\F(\E)=\F_{C_S(\H)}(C_\L(\H))$, it follows that $T\cap C_S(\H)\leq Z(\E)\cap Z(\C_\F(\E))$, $C_\F(\E)\subseteq C_\F(T)$ and $\E\subseteq C_\F(C_S(\H))$. This means that $\E$ and $C_\F(\E)$ commute in $\F$. Lemma~\ref{L:CLTinbNLH} implies that $C_\F(\E)$ is contained in $N_\F(\E)$, and clearly $C_\F(\E)$ is a fusion system over $C_S(\H)=C_S(\E)$.
Now part (a) follows from Lemma~\ref{L:CommuteWithE}.

\smallskip

Part (b) is a consequence of (a) and Proposition~\ref{P:CFEDescribe}. Part (c) follows from Lemma~\ref{L:CentralizerTFullyCentralized}(b), \eqref{E:FullyNormFullyCent} and Lemma~\ref{L:CFEWeaklyInvariant}. 
\end{proof}

The following lemma together with Lemma~\ref{L:EFullyCentralized}(b) and Proposition~\ref{P:CFEDescribe} yields a description of $C_\F(\E)$ in terms of the fusion systems $\E$ and $\F$.

\begin{lemma}\label{L:MoveToFullyCentralized}
Fix $\alpha\in\Hom_\F(TC_S(\E),S)$ such that $\E^\alpha$ is fully $\F$-centralized. Then for all $P,Q\leq C_S(\E)$,  
\[\Hom_{C_\F(\E)}(P,Q)=\{\phi\in\Hom_\F(P,Q)\colon \alpha^{-1}\phi\alpha\in\Hom_{C_\F(\E^\alpha)}(P\alpha,Q\alpha)\},\]
i.e. $\alpha$ induces an isomorphism from $C_\F(\E)$ to $C_\F(\E^\alpha)|_{C_S(\E)\alpha}$. 
In particular, if $\E$ is also fully $\F$-centralized, then $C_S(\E)\alpha=C_S(\E^\alpha)$ and $\alpha$ induces an isomorphism from $C_\F(\E)$ to $C_\F(\E^\alpha)$. 
\end{lemma}

\begin{proof}
Let $\alpha\in\Hom_\F(TC_S(\E),S)$ be such that $\E^\alpha$ is fully centralized. By Lemma~\ref{L:EFullyCentralized}(c), $\alpha$ is a morphism in $N_\F(\tT)=\F_S(G)$ and $C_S(\E)\alpha\leq C_S(\E^\alpha)$. In particular, there exists $g\in G$ such that $\alpha=c_g|_{TC_S(\E)}$.

\smallskip

By Lemma~\ref{L:ConjH}(a), we have $\E^\alpha=\F_{T^g}(\H^g)$ and $T^g=\H^g\cap S$. Fix $\hat{S}_0\in\Syl_p(N_G(\H^g))$ such that $T^gC_S(\H^g)\leq N_S(\H^g)=N_S(\E^\alpha)\leq \hat{S}_0$. As $\E^\alpha$ is fully centralized, it follows from Lemma~\ref{L:EFullyCentralized}(a) that $C_S(\H^g)=C_{\hat{S}_0}(\H^g)$. As $G$ acts on $\L$, one can observe that $c_g$ induces an isomorphism from $N_G(\H)$ to $N_G(\H^g)$. Thus, $\hat{S}_0^{g^{-1}}$ is a Sylow $p$-subgroup of $N_G(\H)$ containing $T$. Hence, we may choose $S_0$ such that
\[S_0:=\hat{S}_0^{g^{-1}}.\]
Notice that $C_S(\H)^g=C_S(\E)\alpha\leq S\cap N_G(\H^g)=N_S(\H^g)\leq \hat{S}_0$ and thus $C_S(\H)\leq S_0$. Moreover, 
\[C_{S_0}(\H)^g=C_{S_0^g}(\H^g)=C_{\hat{S}_0}(\H^g)=C_S(\H^g).\]
Set $\hat{\F}_0:=\F_{\hat{S}_0}(\bN_\L(\H^g))$. By Proposition~\ref{P:CFEDescribe}, $C_{\hat{\F}_0}(\E^\alpha)=C_\F(\E^\alpha)$.  Lemma~\ref{L:CF0ECollect} gives now that $(C_\L(\H),\delta(C_{\F_0}(\E)),C_{S_0}(\H))$ is a regular locality over $C_{\F_0}(\E)$, and $(C_\L(\H^g),\delta(C_\F(\E^\alpha)),C_S(\H^g))$ is a regular locality over $C_\F(\E^\alpha)$. Moreover, $\tilde{\alpha}:=c_g\in\Aut(\L)$ induces an isomorphism from $C_\L(\H)$ to $C_\L(\H^g)$ which takes $C_{S_0}(\H)$ to $C_S(\H^g)$. Thus, it follows from Lemma~\ref{L:AutLFusion} that 
\begin{equation}\label{E:ConjCF0E}
C_{\F_0}(\E)^{\tilde{\alpha}}=\F_{C_{S_0}(\H)}(C_\L(\H))^{\tilde{\alpha}}=\F_{C_S(\H^g)}(C_\L(\H^g))=C_\F(\E^\alpha).
\end{equation}
Let now $P,Q\leq C_S(\H)=C_S(\E)$ and $\phi\in\Hom_\F(P,Q)$. As $\phi$ is a morphism in $\F$, Lemma~\ref{L:CFEWeaklyInvariant} yields that  $\phi$ is a morphism in $C_\F(\E)$ if and only if it is a morphism in $C_{\F_0}(\E)$. By \eqref{E:ConjCF0E}, this is the case if and only if $\alpha^{-1}\phi\alpha=\tilde{\alpha}^{-1}\phi\tilde{\alpha}\in \Hom_{C_\F(\E^\alpha)}(P\alpha,Q\alpha)$. This yields the first part of the  assertion.

\smallskip

If $\E$ is fully centralized, then $|C_S(\E)|\geq |C_S(\E^\alpha)|$ and hence $C_S(\E)\alpha=C_S(\E^\alpha)$. Moreover, $C_{\F_0}(\E)=C_\F(\E)$ by Proposition~\ref{P:CFEDescribe}. Thus, by \eqref{E:ConjCF0E}, $\alpha$ induces an isomorphisms from $C_\F(\E)$ to $C_\F(\E^\alpha)$. 
\end{proof}

\section{Examples and final remarks}\label{S:Examples}

Throughout this section let $\F$ be a saturated fusion system over $S$ and let $\E$ be a subnormal subsystem of $\F$ over $T$. The purpose of this subsection is to discuss further the relationship between the properties of $\E$ being fully normalized, fully centralized and fully automized as well as $T$ being fully normalized and fully centralized in $\F$. The main results we have shown so far in this context are the two implications in \eqref{E:FullyNormFullyCent} as well as the fact proved in Lemma~\ref{L:EfullyNormalized} that $\E$ is fully $\F$-normalized if and only if $T$ is fully $\F$-centralized and $\E$ is fully $\F$-automized.

\smallskip

We firstly show in our examples below that the converses of the two implications in \eqref{E:FullyNormFullyCent} are not true in general. Indeed, the subnormal subsystems we consider in our examples are simple and thus components of $\F$. Hence, the converses of the two implications in \eqref{E:FullyNormFullyCent} even fail to hold in general if $\E$ is a component of $\F$. 

\smallskip

We consider also the statement shown in Lemma~\ref{L:EfullyNormalized} that $\E$ is fully normalized if and only if $\E$ is fully automized and $T$ is fully centralized. As our first example shows, in this statement the condition that $T$ is fully centralized cannot be replaced by the slightly weaker condition that $\E$ is fully centralized. More precisely, we give an example where $\E$ is fully $\F$-centralized and fully $\F$-automized, but $\E$ is not fully $\F$-normalized.

\begin{example}
Set $T:=\<(13)(24),(34)(56)\>\leq A_6$, $\tH:=A_6\times A_6\times A_6$ and $\tT:=T\times T\times T$. Note that $T\in\Syl_2(A_6)$ and so $\tT\in\Syl_2(\tH)$. Let $\sigma\in \Aut(A_6)$ be an involution normalizing $T$ (which we will choose concretely later on). Consider the following two automorphisms of $\tH$ of order $2$.
\begin{eqnarray*}
 \tau_1&\colon& \tH\rightarrow \tH,\;(x_1,x_2,x_3)\mapsto (x_2,x_1,x_3^\sigma)\\
 \tau_2&\colon& \tH\rightarrow \tH,\;(x_1,x_2,x_3)\mapsto (x_1^\sigma,x_3,x_2).
\end{eqnarray*}
Set 
\[A:=\<\tau_1,\tau_2\>,\;G:=\tH\rtimes A,\;S:=\tT\rtimes\<\tau_1\>\mbox{ and }\F:=\F_S(G).\]
One calculates that $\tau_1\tau_2\colon (x_1,x_2,x_3)\mapsto (x_2^\sigma,x_3^\sigma,x_1)$ is of order $3$. Hence, $A$ is isomorphic to $S_3$. In particular, $\<\tau_1\>$ is a Sylow $2$-subgroup of $A$, and so $S$ is a Sylow $2$-subgroup of $G$, which implies that $\F$ is saturated. 

\smallskip

For $i=1,2,3$ let $H_i\cong A_6$ be the $i$th component of the direct product $\tH=A_6\times A_6\times A_6$, and similar $T_i\cong T$ the $i$th component of $\tT=T\times T\times T$. Set 
\[\E_i:=\F_{T_i}(H_i)\mbox{ for }i=1,2,3.
\]
Note that $H_i$ is subnormal in $G$ and $T_i=H_i\cap S$. Thus, $\E_i$ is a subnormal subsystem of $\F$ for $i=1,2,3$. Moreover, $\{H_1,H_2,H_3\}$ is a $G$-conjugacy class of subgroups of $G$, and so $\{\E_1,\E_2,\E_3\}$ is an $\F$-conjugacy class of subnormal subsystems of $\F$. Similarly, $\{T_1,T_2,T_3\}$ is an $\F$-conjugacy class of subgroups of $S$. 

\smallskip

Observe that
\[N_S(\E_i)=\tT \mbox{ for }i=1,2\mbox{ and }N_S(\E_3)=S.\]
Hence, $\E_3$ is fully normalized, whereas $\E_1$ and $\E_2$ are not. 

\smallskip

We consider now two cases for the automorphism $\sigma\in\Aut(A_6)$ which will lead to different properties of the $\E_i$ and $T_i$ in relation to $\F$.

\smallskip

\textbf{Case 1:} $\sigma=c_{(56)}$ is the automorphism of $A_6$ induced by conjugation by the transposition $(56)\in S_6$. Note that $\sigma$ centralizes $T$. %On the other hand, the $\F_T(A_6)$-automorphism group of $\<(12)(56),(34)(56)\>$ 

\smallskip

One sees now that 
\begin{eqnarray*}
 C_S(T_1)&=&C_{\tT}(T_1)=Z(T_1)\times T_2\times T_3,\\
 C_S(T_2)&=&C_{\tT}(T_2)= T_1\times Z(T_2)\times T_3,\\
 C_S(T_3)&=&C_{\tT}(T_3)\<\tau_1\>= (T_1\times T_2\times Z(T_3))\<\tau_1\>.
\end{eqnarray*}
Thus, $T_3$ is fully centralized, whereas $T_1$ and $T_2$ are not. Moreover, we show in the proof below that
\begin{equation}\label{E:Eifullycent}
C_S(\E_i)=T_jT_k\mbox{ for }\{i,j,k\}=\{1,2,3\}.
\end{equation}
So $\E_i$ is fully $\F$-centralized for $i=1,2,3$. In particular, we have shown: 

\smallskip

\textit{
The subnormal subsystems $\E_1$ and $\E_2$ are fully $\F$-centralized while their Sylow subgroups $T_1$ and $T_2$ are not fully $\F$-centralized.}

\smallskip

Observe now that $N_G(T_3)=(H_1 H_2 T_3)\<\tau_1\>=C_G(T_3)T_3$ and thus $\Aut_\F(T_3)=\Inn(T_3)$. As $T_1,T_2,T_3$ are $\F$-conjugate, this implies that $\Aut_\F(T_i)=\Inn(T_i)$ for $i=1,2,3$. In particular, $\Aut_S(T_i)=\Aut_\F(T_i)$ is a $p$-group and so $\E_i$ is fully $\F$-automized for $i=1,2,3$. To summarize, we have thus shown:

\textit{For each $i=1,2$, the subnormal subsystem $\E_i$ is fully $\F$-centralized and fully $\F$-automized, but not fully $\F$-normalized.}

\smallskip

\textbf{Case 2:} The involution $\sigma\in\Aut(A_6)$ is chosen such that $A_6\rtimes \<\sigma\>\cong PGL_2(9)$. (Note that such $\sigma$ exists, as $A_6\cong L_2(9)$ and since $PGL_2(9)$ has dihedral Sylow $2$-subgroups.) We have then $\sigma|_T\not\in\Inn(T)$, as $T$ and $T\<\sigma\>$ are both dihedral and $\sigma$ is an involution.

\smallskip

Note that 
\[C_S(T_i)=Z(T_i)\times T_j\times T_k\mbox{ for }\{i,j,k\}=\{1,2,3\}.\]
(This is easy to see for $i=1,2$, whereas for $i=3$ one needs to use that $\sigma|_T\not\in\Inn(T)$ and thus $\tau_1|_{T_3}\not\in\Inn(T_3)$.) We can conclude that $T_1$, $T_2$ and $T_3$ are all fully $\F$-centralized. In particular, we have seen:

\smallskip

\textit{For $i=1,2$, the subnormal subsystem $\E_i$ has a fully $\F$-centralized Sylow subgroup $T_i$, but $\E_i$ is not fully $\F$-normalized.} 

\smallskip

Using Lemma~\ref{L:EfullyNormalized} or a direct argument, one can moreover observe that $\E_1$ and $\E_2$ are not fully $\F$-automized. 
\end{example}

\begin{proof}[Proof of \eqref{E:Eifullycent}]
Let $\sigma$ be as in Case~1 and $\{i,j,k\}=\{1,2,3\}$. One observes easily that $T_jT_k\leq C_S(\E_i)$ and $C_S(\E_i)\cap T_i=Z(T_i)=1$. Hence, it is sufficient to prove that $C_S(\E_i)\leq \tT=T_1T_2T_3$. For $i=1,2$, this follows from $C_S(\E_i)\leq N_S(\E_i)\leq \tT$. Thus, it remains only to prove that $C_S(\E_3)\leq \tT$.

\smallskip

Assume that $C_S(\E_3)\not\leq \tT$. Then $G_0:=\tH S=\tH C_S(\E_3)$. Note that $S\in\Syl_2(G_0)$ and set $\F_0:=\F_S(G_0)$. As $\tH\unlhd G$, an $\F$-automorphism of $C_S(\E_3)$ must be realized by an element of $N_G(G_0)=G_0$. Hence, $C_S(\E_3):=C_S(\E_3)_\F=C_S(\E_3)_{\F_0}$. As $H_3\unlhd G_0$, it follows from \cite[Theorem~B]{Henke/Semeraro:2015} that $C_S(\E_3)=C_S(H_3)\leq \tT$, contradicting our assumption. This proves $C_S(\E_3)\leq \tT$, and as argued above this shows \eqref{E:Eifullycent}. 
\end{proof}

Suppose now for a moment that $\E$ is a product of components. Generalizing work of Aschbacher \cite[Sections~2.1 and 2.2]{AschbacherFSCT} on normalizers of components, Oliver \cite{Oliver:NormalizerComponents} constructed a normalizer of $\E$  if $T$ is fully $\F$-normalized. He showed furthermore that this normalizer is a saturated subsystem. Recall that we showed that our normalizer of $\E$ is saturated provided $\E$ is fully $\F$-normalized. Comparing these results, the following lemma seems to be of interest.

\begin{lemma}
Suppose $\E$ is a product of components of $\F$. Then $N_S(T)=N_S(\E)$. In particular, $T$ is fully $\F$-normalized if and only if $\E$ is fully $\F$-normalized.
\end{lemma}

\begin{proof}
Clearly $N_S(\E)\leq N_S(T)$. Let $\fC$ be a set of components of $\F$ such that $\E$ is the central product of the components in $\fC$ and pick $s\in N_S(T)$. Then $\fC^s=\{\K^s\colon\K\in\fC\}$ is a set of components of $\F$, and $\E^s$ is the product of the components in $\fC^s$. Assume there exists $\C\in\fC^s\backslash\fC$ and pick $R\leq S$ such that $\C$ is a subsystem over $R$. Then $[R,T]=1$ by \cite[(9.6)]{Aschbacher:2011} (or alternatively by \cite[Lemma~7.15(b)]{Chermak/Henke}). Since $s\in N_S(T)$, we have $R\leq T^s=T$, so $R$ is abelian contradicting \cite[(9.1)]{Aschbacher:2011}. Hence, $\fC^s=\fC$ and $\E^s=\E$ proving $s\in N_S(\E)$. This shows $N_S(\E)=N_S(T)$. As $\E$ was an arbitrary product of components, the same property holds for every $\F$-conjugate of $\E$. This implies the last part of the assertion.
\end{proof}

If $\E$ is an arbitrary subnormal subsystem, then the following example shows that the condition that $T$ is fully $\F$-normalized is neither necessary nor sufficient for $\E$ being fully $\F$-normalized.

\begin{example}
Set $G:=A_4\wr S_3$, $\tT:=O_2(A_4\times A_4\times A_4)$, $S:=\tT\<(12)\>$ and $\F:=\F_S(G)$. Note that $S$ is a Sylow $2$-subgroup of $G$ and so $\F$ is saturated.

\smallskip

Write $L_i\cong A_4$ for the $i$th component of the direct product $A_4\times A_4\times A_4$ inside of $G$ and set $T_i:=O_2(L_i)$. We will consider now two $G$-conjugacy classes of subnormal subgroups of $G$ leading to two different $\F$-conjugacy classes of subnormal subsystems of $\F$. Namely, set
\[H_i:=\tT L_i\mbox{ and }\F_{\tT}(H_i)\mbox{ for all }i=1,2,3\
\]
as well as
\[H_{ij}=L_iT_j\mbox{ and }\E_{ij}:=\F_{T_iT_j}(H_{ij})\mbox{ for all }i,j\in\{1,2,3\}\mbox{ with }i\neq j.\]

\smallskip

\textbf{Properties of the $\E_i$ for $i=1,2,3$:} For $i=1,2,3$ note that $H_i\unlhd A_4\times A_4\times A_4\unlhd G$ so that $H_i$ is subnormal in $G$ with $\tT=H_i\cap S$. Hence, $\E_i$ is a subnormal subsystem of $\F$ over $\tT$ for $i=1,2,3$. Note furthermore that $\{H_1,H_2,H_3\}$ is a $G$-conjugacy class of subgroups of $G$ so that $\{\E_1,\E_2,\E_3\}$ is an $\F$-conjugacy class of subnormal subsystems of $\F$. Clearly $\tT$ is normal in $S$ and thus fully $\F$-normalized. Hence, each $\E_i$ is a subsystem of $\F$ on a fully $\F$-normalized Sylow subgroup. Moreover, 
\[N_S(\E_1)=N_S(\E_2)=\tT\mbox{ and }N_S(\E_3)=S.\]
Hence, $\E_3$ is fully $\F$-normalized, whereas $\E_1$ and $\E_2$ are not. So we have seen:

\smallskip

\textit{The subnormal subsystems $\E_1$ and $\E_2$ are not fully $\F$-normalized, but their Sylow subgroup $\tT$ is fully $\F$-normalized.}

\smallskip

\textbf{Properties of the $\E_{ij}$ for $i\neq j\in\{1,2,3\}$:}
Observe that  $H_{ij}\unlhd L_iL_j\unlhd A_4\times A_4\times A_4\unlhd G$ is a subnormal chain for $H_{ij}$ in $G$ and $T_iT_j=H_{ij}\cap S$ for all $i\neq j$. Hence, $\E_{ij}$ is subnormal subsystem of $\F$ over $T_iT_j$. Note moreover that 
\[N_S(\E_{ij})=\tT\mbox{ for all }i\neq j.\]
Hence, each of the $\E_{ij}$ is fully $\F$-normalized. Observe also that $T_1T_2\unlhd S$ and $N_S(T_1T_3)=\tT$. This gives us the following example:

\smallskip

\textit{The subnormal subsystem $\E_{13}$ is fully $\F$-normalized, but its Sylow subgroup $T_1T_3$ is not fully $\F$-normalized.}
\end{example}

\bibliographystyle{amsplain}
\bibliography{repcoh}

\end{document}